\documentclass[numbers,webpdf]{ima-authoring-template}%

\graphicspath{{Fig/}}

\theoremstyle{thmstyletwo}%
\newtheorem{theorem}{Theorem}%  meant for continuous numbers
\newtheorem{lemma}{Lemma}%
\newtheorem{assumption}{Assumption}
\newtheorem{remark}{Remark}%

\newtheorem{definition}{Definition}

\numberwithin{equation}{section}

\usepackage{graphicx}
\usepackage{subfigure}
\usepackage{color}
\usepackage{circuitikz}
\usepackage{tikz}
\usepackage{comment}
\usetikzlibrary{calc,positioning,shapes}
\usetikzlibrary{patterns,decorations.pathmorphing,decorations.markings}
\usetikzlibrary{external}
\usepackage{pgfplots}
\pgfplotsset{compat=newest}
\usepackage[margin=10pt,font=small,labelfont=bf,labelsep=endash]{caption}
\usetikzlibrary{intersections, backgrounds}
\usetikzlibrary{circuits}
\usetikzlibrary{circuits.ee.IEC}
\usepackage{pgfplots}
\pgfplotsset{compat=newest}

\definecolor{color0}{rgb}{0.12156862745098,0.466666666666667,0.705882352941177}
\definecolor{color1}{rgb}{1,0.498039215686275,0.0549019607843137}
\definecolor{color2}{rgb}{0.172549019607843,0.627450980392157,0.172549019607843}
\definecolor{color3}{rgb}{0.83921568627451,0.152941176470588,0.156862745098039}
\definecolor{color4}{rgb}{0.580392156862745,0.403921568627451,0.741176470588235}
\definecolor{color5}{rgb}{0,0,0}

\definecolor{mycolor1}{rgb}{0.00000,0.44700,0.74100}% blue
\definecolor{mycolor2}{rgb}{0.85000,0.32500,0.09800}% red
\definecolor{mycolor3}{rgb}{0.92900,0.69400,0.12500}% orange/yellow
\definecolor{mycolor4}{rgb}{0.46600,0.67400,0.18800}% green
\definecolor{mycolor5}{rgb}{0.49400,0.18400,0.55600}% purple

\newcommand{\lineWidth}{1.2pt}
\newcommand{\imageWidth}{2.0in}
\newcommand{\imageHeight}{1.8in}

\DeclareMathOperator*{\argmax}{arg\,max}

\newcommand{\N}{\ensuremath\mathbb{N}}

\newcommand{\R}{\ensuremath\mathbb{R}}
\newcommand{\ri}{\ensuremath\mathrm{i}}

\newcommand{\T}{\ensuremath\mathsf{T}}

\DeclareMathOperator{\ee}{e}

\DeclareMathOperator{\diag}{diag}

\DeclareMathOperator{\range}{im}

\DeclareMathOperator{\real}{Re}
\DeclareMathOperator{\imag}{Im}

\newcommand{\calD}{\mathcal{D}}
\newcommand{\calE}{\mathcal{E}}

\newcommand{\calG}{\mathcal{G}}

\newcommand{\calN}{\mathcal{N}}
\newcommand{\calO}{\mathcal{O}}

\newcommand{\calS}{\mathcal{S}}

\newcommand{\calU}{\mathcal{U}}
\newcommand{\calV}{\mathcal{V}}

\newcommand{\calX}{\mathcal{X}}
\newcommand{\calY}{\mathcal{Y}}

\definecolor{mycolor1}{rgb}{0.00000,0.44700,0.74100}% blue
\definecolor{mycolor2}{rgb}{0.85000,0.32500,0.09800}% red
\definecolor{mycolor3}{rgb}{0.92900,0.69400,0.12500}% orange/yellow
\definecolor{mycolor4}{rgb}{0.46600,0.67400,0.18800}% green
\definecolor{mycolor5}{rgb}{0.49400,0.18400,0.55600}% purple

\usepackage{xspace}
\newcommand{\abbr}[1]{\textsf{#1}\xspace}

\newcommand{\SCM}{\abbr{SCM}}
\newcommand{\PDE}{\abbr{PDE}}
\newcommand{\RBM}{\abbr{RBM}}

\newcommand{\FEM}{\abbr{FEM}}

\newcommand{\SSCM}{\abbr{SSCM}}

\newcommand{\GM}{\abbr{GM}}
\newcommand{\CT}{\abbr{CT}}
\newcommand{\LP}{\abbr{LP}}
\newcommand{\EIGOPT}{\abbr{EigOpt}}
\newcommand{\EIGTOOL}{\abbr{EigTool}}

\newcommand{\sn}{\mathsf{n}}
\newcommand{\sr}{\mathsf{r}}	% for standard macros
\xdefinecolor{paleblue}{rgb}{0.3,0.3,0.7}

\newcommand{\sd}{\mathsf{d}}
\newcommand {\UB}{{\rm{UB}}}
\newcommand {\LB}{{\rm{LB}}}
\newcommand {\tmu}{{\tilde{\mu}}}
\newcommand {\dimPar}{p}

\newcommand{\lapVar}{z}
\newcommand{\PSm}{M}

\usepackage{mathtools}
\usepackage{cleveref}

\crefname{proposition}{Proposition}{Propositions}
\Crefname{proposition}{Proposition}{Propositions}

\crefname{lemma}{Lemma}{Lemmas}
\Crefname{lemma}{Lemma}{Lemmas}

\crefname{assumption}{Assumption}{Assumptions}
\Crefname{assumption}{Assumption}{Assumptions}

\crefname{remark}{Remark}{Remarks}
\Crefname{remark}{Remark}{Remarks}

\usepackage{algorithm}
\usepackage{algpseudocode}
\usepackage{booktabs}
\usepackage{newtxtext,newtxmath}

\begin{document}

\DOI{DOI HERE}
\copyrightyear{2021}
\vol{00}
\pubyear{2021}
\access{Advance Access Publication Date: Day Month Year}
\appnotes{Paper}
\copyrightstatement{Published by Oxford University Press on behalf of the Institute of Mathematics and its Applications. All rights reserved.}
\firstpage{1}

%\subtitle{Subject Section}

\title[Uniform Approximation of Eigenproblems on a Domain]{Uniform Approximation of Eigenproblems of a Large-Scale Parameter-Dependent Hermitian Matrix}

\author{Mattia Manucci${}^{\ast}$
\address{\orgdiv{Stuttgart Center for Simulation Science (SC SimTech)}, \orgname{University of Stuttgart}, \orgaddress{\street{Universit\"{a}tsstr.~32}, \postcode{70569}, \state{Stuttgart}, \country{Germany}}}}
\author{Emre Mengi
\address{\orgdiv{Department of Mathematics}, \orgname{Ko\c{c} University}, \orgaddress{\street{Rumeli Feneri Yolu}, \postcode{34450}, Sar{\i}yer, \state{Istanbul}, \country{Turkey}}}}
\author{Nicola Guglielmi
\address{\orgdiv{Division of mathematics}, \orgname{Gran Sasso Science Institute}, \orgaddress{\street{Viale Francesco Crispi~7}, \postcode{67100}, \state{L'Aquila}, \country{Italy}}}}

\authormark{Mattia Manucci, Emre Mengi and Nicola Guglielmi}

\corresp[$\ast$]{Corresponding author: \href{mattia.manucci@simtech.uni-stuttgart.de}{mattia.manucci@simtech.uni-stuttgart.de}}

\received{Date}{0}{Year}
\revised{Date}{0}{Year}
\accepted{Date}{0}{Year}

%\editor{Associate Editor: Name}

\abstract{We consider the uniform approximation 
% over the continuum of a compact domain
of the smallest eigenvalue of a large parameter-dependent 
Hermitian matrix by that of a smaller counterpart obtained through projections.
 The projection subspaces are constructed iteratively 
by means of a greedy strategy; at each iteration the parameter where a surrogate error
is maximal is computed and the eigenvectors associated with the smallest eigenvalues 
at the maximizing parameter value are added to the subspace. Unlike the classical approaches, such as the successive constraint method, that maximize 
such surrogate errors over a discrete and finite set, we maximize the surrogate error over 
the continuum of all permissible parameter values globally. We formally prove 
that the projected eigenvalue function converges to the actual eigenvalue function
uniformly. In the second part, we focus on the uniform approximation of the smallest singular value of 
a large parameter-dependent matrix, in case it is non-Hermitian. The proposed frameworks on 
numerical examples, including those arising from discretizations of parametric PDEs, reduce the 
size of the large matrix-valued function drastically, while retaining a high
accuracy over all permissible parameter values.}

\maketitle

\section{Introduction and motivations}
\label{sec:intro}
We aim for an accurate uniform approximation of the smallest eigenvalue function of a large analytic 
Hermitian matrix-valued function by that of a smaller counterpart. Formally, 
given $A \, : \, \underline{\calD} \rightarrow {\mathbb C}^{\sn \times \sn}$ 
analytic\footnote{$A \, : \, \underline{\calD} \rightarrow {\mathbb C}^{\sn \times \sn}$ is analytic amounts
to the real analyticity of $\text{Re}(A) = (A + \overline{A})/2$ and $\text{Im}(A) = {\rm i} (\overline{A} - A)/2$.} Hermitian on $\underline{\calD} \subseteq \R^{\dimPar}$
and given an error tolerance $\varepsilon$, we want to find a subspace $\calV$ of ${\mathbb C}^n$ that satisfies 
\[
	\max_{\mu \in \calD} | \lambda_{\min}(\mu) - \lambda^{\calV}_{\min}(\mu) | 	\;	\leq	\;	 \varepsilon	\:	,
\]
where $\lambda_{\min}(\mu)$, $\lambda^{\calV}_{\min}(\mu)$ denote the smallest eigenvalues 
of $A(\mu)$, $V^\ast A(\mu) V$, respectively, $V$ is a matrix whose columns form an orthonormal basis for $\calV$,
and $\calD \subset \underline{\calD}$ is compact. It is preferable that $\calV$ is as small dimensional as possible.

The main motivation for this approximation problem 
arises from the estimation of the coercivity constant for a parametrized partial differential equation (\PDE); 
see, e.g., \cite{RozHP08}. In fact, the coercivity constant appears in the a posteriori error estimates used in the 
reduced basis method (\RBM) to numerically solve a parameterized \PDE~\cite{HesRS16}. For a given discretization 
method, such as the finite element method (\FEM), finite difference, or finite volume, if the corresponding discrete 
operator $A(\mu)$ is positive definite and Hermitian, then the role of the coercivity constant associated with the 
differential operator of the \PDE~in the continuous setting is played by the smallest eigenvalue of $A(\mu)$. 
Since numerical approximations of {\PDE}s usually lead to large problems, applying a
standard eigensolver, such as the Lanczos method \cite{BaiDDRV00}, may be computationally very expensive
and may not be suitable for computing the smallest eigenvalue for several values of $\mu$.

If the nature of the \PDE~is such that after discretization it does not lead to Hermitian matrices - for instance, this is the case in the presence of a convective term - then the coercivity constant 
may be replaced by the inf-sup stability constant, which after discretization corresponds to the 
smallest singular value of a general non-Hermitian matrix $A(\mu)$. Moreover, for non-Hermitian problems, the smallest eigenvalue of the negative Hermitian part of $A(\mu)$, i.e., $-(A(\mu) + A(\mu)^{\ast})/2$, provides a priori insights on the asymptotic stability of any reduced-order model obtained via Galerkin projection of $A(\mu)$; see \cite{Emb19}. Consequently, even when the original system matrix is not Hermitian, it is still valuable to consider their Hermitian components.

Parametric eigenvalue problems also appear in the context of quantum spin systems \cite{ParF10}, 
where the lowest energy of the system is the quantity of interest and corresponds to the smallest eigenvalue of the 
system Hamiltonian. Such systems generally have very large dimensionality because the size of the state space grows exponentially with the length of the chain under consideration, and determining the corresponding smallest eigenvalue for every $\mu$ in a prescribed set $\calD$ is essential for constructing the phase diagram of the system.

The remark below further makes the assumptions on the infinite-dimensional PDE problem explicit so that
the (continuous) infinite-dimensional parametrized PDE and (discrete) parameter-dependent matrix perspectives are consistent.

\begin{remark}\label{rm:motivate_inf_dim} \rm $\;$
\\
Let $\Omega \subset \mathbb{R}^d$ be a bounded domain and consider, as an illustrative example, the parametric diffusion operator
\[
\mathcal{L}_\mu u := -\nabla \cdot \big(f(\mu) \nabla u\big),
\]
with homogeneous Dirichlet boundary conditions, where the diffusion coefficient $f(\mu)$ is uniformly bounded above and below by positive constants for all admissible $\mu$. Under these assumptions, $\mathcal{L}_\mu$ is a self-adjoint, positive (unbounded) operator on $L^2(\Omega)$ and its inverse
\[
\mathcal{L}_\mu^{-1}: L^2(\Omega) \to L^2(\Omega),
\]
is compact because the embedding $H^1_0(\Omega) \hookrightarrow L^2(\Omega)$ is compact. 
Consequently, the spectral problem for $\mathcal{L}_\mu$ is equivalent to a compact eigenproblem for $\mathcal{L}_\mu^{-1}$; 
the eigenvalues of $\mathcal{L}_\mu$ form a discrete sequence
\[
0 < \lambda_1(\mu) \le \lambda_2(\mu) \le \dots, \quad \lambda_j(\mu) \to \infty,
\]
while the eigenvalues of $\mathcal{L}_\mu^{-1}$ accumulate at $0$. In particular, the coercivity constant $\alpha(\mu)$ in the variational formulation coincides with $\lambda_1(\mu)$, and approximating $\lambda_1(\mu)$ is equivalent to approximating the largest eigenvalue of $\mathcal{L}_\mu^{-1}$, or equivalently the smallest eigenvalue of $-\mathcal{L}_\mu^{-1}$. \\ For our analysis it therefore suffices to assume that we are dealing with a family of compact, self-adjoint operators 
$\{-\mathcal{L}_\mu^{-1}\}$, or with their finite-dimensional Galerkin discretizations. 
Under these standard hypotheses, the passage from the continuous PDE to the discrete matrix setting is justified 
and our uniform convergence results apply. More generally, we are interested in the resolvent of an elliptic operator, and for it our approach appears justified
as explained on the particular elliptic operator ${\mathcal L}_\mu$ above.
\end{remark}

\subsection{State of the art}
An important assumption in \RBM, which we also consider in this work, is that 
$A(\mu)$ can be written in an affine form \cite[Sec.~3.3]{HesRS16} 
of the form
\begin{equation}\label{int1}
	A(\mu)	\;	=	\;	\theta_1(\mu) A_1	+	\dots		+	\theta_\kappa(\mu) A_\kappa	\:
\end{equation}
for a small $\kappa \ll \sn$,
where the matrices $A_1, \dots, A_\kappa \in {\mathbb C}^{\sn\times \sn}$ and real analytic scalar-valued 
functions $\theta_1, \dots , \theta_\kappa : {\mathbb R}^{p} \rightarrow {\mathbb R}$ are available for use. In the context of approximating the smallest eigenvalue of $A(\mu)$, the matrices
$A_1, \dots, A_\kappa$ are also assumed to be Hermitian. This assumption holds for a number of important applications, including some classes of linear parametric {\PDE}s, parameter-dependent quantum spin systems \cite{HerSWR22,BreHWRS23}, and more generally, when considering the Hermitian part of non-Hermitian $A(\mu)$. Within the context of \RBM, a few approaches have been developed to deal with the 
approximation of the smallest eigenvalue $\lambda_{\min}(\mu)$ of $A(\mu)$ by that of a smaller counterpart. Especially,
the successive constraint method (\SCM) \cite{HuyRSP07} is a well-known approach in the \RBM~community. 
It is based on the construction of an upper bound $\lambda_{\UB}(\mu)$ and a lower bound $\lambda_{\LB}(\mu)$ 
for $\lambda_{\min}(\mu)$ in a greedy fashion. Specifically, at every iteration, the parameter
\begin{equation}\label{eq:max_error2}
	\widetilde{\mu} \; = \;  \argmax_{\mu\in\Xi} \, \{ \lambda_{\UB}(\mu)-\lambda_{\LB}(\mu) \}
\end{equation}
is computed, where the maximization is over a discrete, finite set $\Xi\subset\calD$ chosen a priori. 
Then $\lambda_{\UB}(\mu)$, $\lambda_{\LB}(\mu)$ are modified in such a way that they interpolate 
$\lambda_{\text{min}}(\mu)$ at $\mu = \widetilde{\mu}$. A downside of \SCM~is that it often exhibits 
slow convergence, which can partly be attributed to the lack of the Hermite interpolation 
property in the lower bound $\lambda_{\LB}(\mu)$.
An approach making use of subspace projections of the form $V^\ast A(\mu) V$ is proposed in \cite{SirK16},
which we refer to as the subspace-\SCM~(\SSCM) method. This is also a greedy procedure and is based
on a maximization problem, as in \eqref{eq:max_error2}, over a discrete, finite set  $\Xi\subset\calD$.
The authors propose to use the smallest eigenvalue 
of the projected problem $V^\ast A(\mu) V$ as an upper bound and also derive a computationally efficient 
lower bound, for which the main ingredients rely on eigenvalue perturbation theory. They show that, with the 
modified bounds, the algorithm converges faster than the original \SCM method, especially since the new 
lower bound is proven to satisfy the Hermite interpolation property with the original smallest eigenvalue function.

\subsection{Main contributions}
The approach we propose and analyze here is also a greedy procedure that relies on the upper and lower bounds proposed in \cite{SirK16}, but it is based on the
computation of the parameter
\begin{equation}\label{eqn:par:con}
	\widehat{\mu}	\;	\in	\;
	\argmax_{\mu\in\calD} \, \{ \lambda_{\UB}(\mu)-\lambda_{\LB}(\mu) \}	\:	,
\end{equation}
thus,
we maximize the gap between the upper and lower bounds over the continuum of the
domain $\calD$ rather than over a discrete, finite subset $\Xi \subset \calD$. For this, we provide a rigorous proof of convergence of the approach here for any compact set $\calD\in\R^p$ 
and any positive integer $p$. Specifically, when $A(\mu)$ is an 
infinite-dimensional self-adjoint compact operator, we establish that $\lambda^{\calV}_{\min}(\mu)$
converges to $\lambda_{\min}(\mu)$ uniformly as the dimension of $\calV$ goes to infinity. 
(We refer to Remark \ref{rm:motivate_inf_dim} above on the significance of operating in
the infinite-dimensional self-adjoint compact operator setting.)
One of the advantages of this strategy is that, assuming that \eqref{eqn:par:con} is exactly computed, the subspace $\calV$ constructed with the approach here is
such that $\lambda_{\min}(\mu)$ is approximated by $\lambda^{\calV}_{\min}(\mu)$ with a uniform 
error certificate over the continuum of the domain $\calD$, unlike \SCM \cite{HuyRSP07} and its improvement \SSCM \cite{SirK16} where the guarantees on the approximation holds only for a discrete grid $\Xi\subseteq\calD$. Naturally, this approach comes with the difficulty of solving a non-smooth and non-convex optimization problem. If the number of parameters $p$ is small, e.g., $p = 1$ 
or $p = 2$, a global solution to the optimization problem in \eqref{eqn:par:con} may be possible. For instance, \EIGOPT, introduced in \cite{MenYK14}, exploits the
Lipschitz continuity of the eigenvalue functions and is suitable for such problems, as demonstrated by its application in \cite{KanMMM18} and in several numerical examples towards the end of this text. 
It should be emphasized that \EIGOPT we rely on here is not based on a grid search, rather based on the repeated
maximization of a piecewise-quadratic approximation of the objective, that gradually becomes more accurate,
and is guaranteed to converge to the global maximizer if a few parameters are set properly.
Still, we stress that global optimization over $\calD$, when feasible, 
is usually more expensive than optimization over a discrete, finite set. Furthermore, the solver addressing problem \eqref{eqn:par:con} 
such as \EIGOPT may not achieve convergence after a specified number of iterations, consequently yielding a suboptimal solution. Nonetheless, when the goal is to maintain a specified error level across the domain $\calD$, optimizing over the continuum set may be preferable to identify a better suited $\hat \mu$ for the construction of a subspace that controls the error all over $\calD$, even when the global optimization solver provides suboptimal solutions. A numerical example later in the text will demonstrate this point. Our numerical experiments further demonstrate that subspaces generated through global optimization, even when 
performing a limited number of iterations and with a moderate parameter count ($p>2$),  
can be further be refined by repeated optimization on finitely many points 
in order to attain accuracy across a dense discrete grid. Finally, let us remark that, even in the situation where the solver for \eqref{eqn:par:con} does not converge, and thus not necessarily yields a global solution in the last iteration of the greedy strategy, the surrogate error, which is efficiently computable, 
still provides an upper bound on the error of the approximation over all $\mu\in\calD$.

The second part of this work focuses on the 
uniform approximation of the smallest singular value on a compact set 
$\calD \subset \underline{\calD}$ of a large matrix-valued function
$A \, : \, \underline{\calD} \rightarrow {\mathbb C}^{\sn \times \sn}$ 
analytic but not Hermitian on $\underline{\calD} \subseteq \R^{\dimPar}$.
In this case, we still assume $A(\mu)$ has the expression of the form (\ref{int1}), 
and the availability of $A_1, \dots, A_\kappa \in {\mathbb C}^{\sn\times \sn}$ and 
$\theta_1, \dots , \theta_\kappa : {\mathbb R}^{p} \rightarrow {\mathbb R}$. But the matrices $A_1, \dots, A_\kappa$ are no longer Hermitian. The state of the art in addressing such a problem is to determine the approximation of the smallest eigenvalue of the Hermitian matrix $A(\mu)^{*}A(\mu)$, which is simply the square of the smallest singular value of $A(\mu)$. However, this has several drawbacks, which we discuss in the related section later in this text. To mitigate such drawbacks, we propose two-sided projections that can be applied directly to $A(\mu)$, and which, despite not guaranteeing the satisfaction of a specified uniform error tolerance, ensures the exact recovery of the smallest eigenvalue function at the interpolation points (in the Hermite sense) 
and overcomes the drawbacks of the state of art approach as we show in the numerical experiments.

In summary, our contributions are primarily in the following two directions.
\begin{itemize}
    \item We introduce innovative computational methodologies for estimating the smallest eigenvalue and singular value functions over the continuum of domain. These methodologies are based on subspace frameworks and utilize greedy-like approaches to construct projection spaces with the aid of a surrogate
error function, that is efficiently computable. The principal advancement over existing methods 
is the maximization of the surrogate error over the continuum of the domain rather than over
a finite subset of the domain. Additionally, for the singular value scenario, we present a novel strategy that 1) overcomes the limitations of the state-of-the-art methodology, and 2) revels to be extremely accurate despite not leading approximations with certified error tolerances.
    \item We demonstrate the convergence of our approximation framework in the infinite-dimensional context. This convergence proof trivially holds for the finite-dimensional setting, and involves an in-depth analysis of the regularity conditions associated with the surrogate error, which provides valuable insights into the primary 
 properties of the subspace framework when applied to approximate an eigenvalue-type function. Furthermore, we investigate the eigenvalue lower bound of \cite{SirK16}, and present several formal results shedding 
 light into the properties of this lower-bound function, which are essential for adjusting the parameters of our approximation algorithms. These results also appear to be vital for comprehending the potential 
behavior of the lower-bound function throughout the iterations of the proposed greedy strategy, 
as we illustrate later in the manuscript.
\end{itemize}

\subsection{Outline}
The rest of this paper is organized as follows. In \Cref{sec2}, we recall the definitions of 
the upper and lower bounds for the smallest eigenvalue of $A(\mu)$ as introduced in \cite{SirK16}.
Moreover, in this section, we prove several properties of the quantities involved in the computation 
of the lower bound, which have not been proven previously and will be crucial for deriving 
the global convergence results. In \Cref{sec3}, 
we present our subspace framework that operates on the continuum of the compact 
domain ${\mathcal D}$. The outcome of our framework is a subspace ${\mathcal V}$ 
such that the maximal error 
$\max_{\mu \in {\mathcal D}} \lambda^{\mathcal V}_{\min}(\mu) - \lambda_{\min}(\mu)$ is below a 
prescribed tolerance. In \Cref{sec4}, we prove the uniform convergence of our framework, which is based on 
the uniform Lipschitz continuity of the gap between the upper and lower bounds for $\lambda_{\min}(\mu)$. 
Formal arguments showing the Lipschitz continuity of the gap between the upper and lower bounds 
under some assumptions are given in \Cref{sec5}. In \Cref{sec7}, we focus on the approximation of 
the smallest singular value of $A(\mu)$ on all $\mu \in {\mathcal D}$ when $A(\mu)$ is an analytic 
non-Hermitian matrix-valued function. Finally, \Cref{sec8} is devoted to numerical experiments on synthetic and real examples.

\subsection{Notation}
In the finite-dimensional case, for $x\in {\mathbb C}^{\sn}$ and $A\in {\mathbb C}^{\sn\times\sn}$, we have
$\|x\|$ denoting the Euclidean norm, and $\|A\| := \max_{w \in {\mathbb C}^{\sn} \, , \, \| w \|= 1} \| A w \|$ 
the associated induced norm, i.e., the spectral norm of $A$.
For a matrix $B$, we represent with $\text{Col}(B)$ and $\text{Null}(B)$ the column space and null space of $B$, respectively.
Furthermore, $B(i_1:i_2,j_1:j_2)$ for positive integers $i_1, i_2, j_1, j_2$ such that $i_1 < i_2$ and $j_1 < j_2$ represent
the submatrix of the matrix $B$ consisting of its rows $i_1$ through $i_2$ and columns $j_1$ through $j_2$. The notation $\ell^2(\N)$ is reserved for the Hilbert 
space of square summable infinite complex sequences equipped with the inner product 
$\langle w, v \rangle = \sum_{i = 1}^\infty \overline{w}_i v_i$ and the norm 
$\| w \| =  \sqrt{\langle w, w \rangle} = \sqrt{\sum_{i = 1}^\infty  | w_i |^2}$. For a linear bounded operator $A$
on $\ell^2(\N)$, we have $\|A\|$ representing the induced operator norm of $A$.
The symbol $I$ denotes the identity matrix of appropriate size in the finite-dimensional setting or
the identity operator on $\ell^2(\N)$. In the finite-dimensional setting, sometimes we use $I_m$ to denote the identity matrix of size $m$, and ${\mathbf e}_j$ to denote the
$j$-th column of the identity matrix $I$. We use $\range({\mathcal F})$ to denote the image
of a map ${\mathcal F}$, while $\imag(\lapVar)$ denotes the imaginary part of $\lapVar\in {\mathbb C}$. The notation ${\mathcal S}(A)$ represents the set of eigenvalues
of a matrix $A \in {\mathbb C}^{\sn\times\sn}$, or the point spectrum of a linear operator $A$. Finally, for a given vector $v\in{\mathbb C}^{\sn}$, $\diag(v)$ represents the square diagonal matrix with the elements of the vector $v$ on the main diagonal.

%-----------------------------------------------------------------------------%

\section{Practical Lower and Upper Bounds for \texorpdfstring{$\lambda_{\min}(\mu)$}{TEXT}} \label{sec2}

In \cite{SirK16}, upper and lower bounds for $\lambda_{\min}(\mu)$ are proposed. Since we employ the same bounds, we recall their definitions in the following paragraphs. Regarding the lower bound, we establish  formally several properties concerning this bound and its ingredients. These results — specifically \Cref{thm:rho0} and \Cref{lem:props_eta} — are not established in \cite{SirK16}. Besides providing 
key insights into the behavior of the ingredients in the lower bound’s definition, these results are also 
employed in this work to derive a novel and simpler proof of the Hermite interpolation 
property of the lower bound, which is first proven in \cite[Thm.~3.6]{SirK16}.
\subsection{Upper Bound}\label{sec:up_bound}

The upper bound is straightforward. For any subspace ${\mathcal V}$ of ${\mathbb C}^{\sn}$, we have:
\begin{equation}\label{eq:monotone}
\lambda_{\min}(\mu) = \min_{v \in {\mathbb C}^{\sn}, \| v \| = 1} v^\ast A(\mu) v 
			\leq 
	\min_{v \in {\mathcal V}, \| v \| = 1} v^\ast A(\mu) v = \lambda^{\mathcal V}_{\min}(\mu),
\end{equation}
where $\lambda^{\mathcal V}_{\min}(\mu)$ denotes the smallest eigenvalue of $A^V(\mu) = V^\ast A(\mu) V$. Here, $V$ is a matrix whose columns form an orthonormal basis for ${\mathcal V}$. Remarkably, if $\lambda_{\min}(\widehat{\mu})$ is simple at some $\widehat{\mu} \in {\mathcal D}$, and the corresponding eigenvector of $A(\widehat{\mu})$ lies in ${\mathcal V}$, then $\lambda_{\min}(\mu)$ and $\lambda^{\mathcal V}_{\min}(\mu)$ are differentiable at $\mu = \widehat{\mu}$ with:
\begin{equation}
\lambda_{\min}(\widehat{\mu}) = \lambda^{\mathcal V}_{\min}(\widehat{\mu}) \quad \text{and} \quad \nabla \lambda_{\min}(\widehat{\mu}) = \nabla \lambda^{\mathcal V}_{\min}(\widehat{\mu}),
\label{eq:interpolate}
\end{equation}
that is a Hermite interpolation property (see, e.g., \cite[Lem.~2.6]{KanMMM18} for the infinite
dimensional case; the finite-dimensional counterpart above in 
${\mathbb C}^\sn$ holds using the standard inner product in ${\mathbb C}^\sn$ in the arguments).
The left equality in (\ref{eq:interpolate}) holds even when 
$\lambda_{\min}(\widehat{\mu})$ is not a simple eigenvalue, as long as the 
eigenvector corresponding to $\lambda_{\min}(\widehat{\mu})$ is in ${\mathcal V}$.

\subsection{Lower Bound}\label{sec:low_bound}

The lower bound in \cite{SirK16} is more complicated.
It involves an iterative process that generates points $\mu_1, \dots, \mu_j \in {\mathcal D}$ 
after $j$ iterations, and also an associated subspace ${\mathcal V} = {\mathcal V}_j$ 
used in the upper bound in Section \ref{sec:up_bound}.
For a fixed integer $\ell \geq 1$, the projection subspace ${\mathcal V}$ 
is defined as:
\begin{equation} \label{eq:low_bound}
{\mathcal V} = {\mathcal V}_j = \text{span} \left\{ v^{(1)}_1, \dots, v^{(1)}_\ell, \dots, v^{(j)}_1, \dots, v^{(j)}_\ell \right\},
\end{equation}
where $v^{(i)}_k$ denotes a unit eigenvector of $A(\mu_i)$ corresponding to its $k$-th smallest eigenvalue, $\lambda^{(i)}_k$. Let $V_j$ be a matrix whose columns form an orthonormal basis for ${\mathcal V}_j$. The $k$th smallest eigenvalue of the projected matrix $A^{V_j}(\mu) = V_j^\ast A(\mu) V_j$ 
and a corresponding unit eigenvector are denoted by 
$\lambda^{{\mathcal V}_j}_k(\mu)$ and $w^{V_j}_k(\mu)$, respectively.

We define $U_j(\mu)$ as the matrix formed from the eigenvectors 
of $A^{V_j}(\mu)$ lifted to the full space:
\begin{equation} \label{eq:6}
U_j(\mu) = \left[ V_j w^{V_j}_1(\mu), \dots, V_j w^{V_j}_\sr(\mu) \right],
\end{equation}
where $\sr \leq \ell$, and ${\mathcal U}_j(\mu)$ as the column space of $U_j(\mu)$. 
The orthogonal complement of this subspace is denoted as ${\mathcal U}_j^\bot(\mu)$, 
with $U_j^\bot(\mu)$ being the matrix whose columns form an orthonormal basis 
for ${\mathcal U}_j^\bot(\mu)$.

The matrix $A(\mu)$ is unitarily similar to 
\[
	\left[ \begin{array}{cc} 
		U_j(\mu)^\ast A(\mu) U_j(\mu) & U_j(\mu)^\ast A(\mu) U_j^\bot(\mu) \\ 
		U_j^\bot(\mu)^\ast A(\mu) U_j(\mu) & U_j^\bot(\mu)^\ast A(\mu) U_j^\bot(\mu) 
	\end{array} \right] ,
\]
so the transformed matrix above has the smallest eigenvalue $\lambda_1(\mu) = \lambda_{\min}(\mu)$. 
Disregarding the off-diagonal blocks of the transformed matrix, 
the smallest eigenvalue of the remaining block diagonal matrix is:
\[
	\min 
	\left\{ \lambda^{{\mathcal U}_j(\mu)}_1(\mu), \lambda^{{\mathcal U}_j^\bot(\mu)}_1(\mu) \right\}.
\]
A lower bound for $\lambda_{\min}(\mu)$ in terms of the minimum above
(involving  two projected eigenvalue problems) can be deduced as elaborated on next.
As we shall see, the deduced lower bound improves as more eigenvectors are 
included in ${\mathcal V}_j$.

%%%%%

In particular, it follows from eigenvalue perturbation theory \cite[Thm.~2]{Li2005} that
\begin{equation}\label{eq:lbound_int}
\begin{split}
	& \left| 
	\lambda_1(\mu) - \min \left\{ \lambda^{{\mathcal V}_j}_1(\mu), \lambda^{{\mathcal U}_j^{\bot}(\mu)}_1(\mu) \right\} 
	\right|
	\; \leq  \;	
    \frac{2 \rho^{(j)}(\mu)^2}{\xi^{(j)}(\mu)
    + \sqrt{ \xi^{(j)}(\mu)^2 + 4 \rho^{(j)}(\mu)^2 }} \,,	\quad	\text{where}	\\[.5em]
%	\text{where} \;\;
	& \rho^{(j)}(\mu) := \| U_j^{\bot}(\mu)^\ast A(\mu) U_j(\mu) \|,  \;\;
	\xi^{(j)}(\mu) := 
	\big| \lambda^{{\mathcal U}_j}_1(\mu) - \lambda^{{\mathcal U}_j^{\bot}(\mu)}_1(\mu) \big|		=
	\big| \lambda^{{\mathcal V}_j}_1(\mu) - \lambda^{{\mathcal U}_j^{\bot}(\mu)}_1(\mu) \big|.
\end{split}
\end{equation}
The inequality in \eqref{eq:lbound_int} yields the following lower bound for the 
smallest eigenvalue $\lambda_1(\mu)$ of $A(\mu)$:
\begin{equation}\label{eq:low_bound1}
\begin{split}
	f^{(j)}\left( \lambda^{{\mathcal U}_j^{\bot}(\mu)}_1(\mu) \right) &\leq \lambda_1(\mu), \quad \text{with} \\
	f^{(j)}(\eta) &:= 
	\min \left\{ \lambda^{{\mathcal V}_j}_1(\mu), \eta \right\}
	- \frac{2\rho^{(j)}(\mu)^2}{\left| \lambda^{{\mathcal V}_j}_1(\mu) - \eta \right|  
	+ \sqrt{ \left( \lambda^{{\mathcal V}_j}_1(\mu) - \eta \right)^2 + 4 \rho^{(j)}(\mu)^2 }}.
\end{split}
\end{equation}

One observation that facilitates the use of (\ref{eq:low_bound1})
is that $\rho^{(j)}(\mu)^2$ can be obtained efficiently, i.e.,
\begin{equation}\label{eq:rhomu}
\begin{split}
	\rho^{(j)}(\mu)^2 &= \| (I - U_j(\mu) U_j(\mu)^\ast) A(\mu) U_j(\mu) \|^2 
	= \| A(\mu) U_j(\mu) - U_j(\mu) \Lambda^{{\mathcal U}_j}(\mu) \|^2 \\
	&= \lambda_{\max}\left( U_j(\mu)^\ast A(\mu)^\ast A(\mu) U_j(\mu) - \Lambda^{{\mathcal U}_j}(\mu)^2 \right),
\end{split}
\end{equation}
where
\begin{equation}\label{eq:defn_Lambda_Uj}
\Lambda^{{\mathcal U}_j}(\mu) := U_j(\mu)^\ast A(\mu) U_j(\mu)
= \diag\left( \lambda^{{\mathcal V}_j}_1(\mu), \ldots, \lambda^{{\mathcal V}_j}_{\sr}(\mu) \right).
\end{equation}
Moreover, $\rho^{(j)}(\mu)$ vanishes 
at $\mu_1, \dots, \mu_j$, as proven next.

\begin{lemma}\label{thm:rho0}
Consider $\rho^{(j)}(\mu)$ defined as in \eqref{eq:rhomu} and $\mu_i$ with $i\in\{1,\ldots,j\}$, the parameters used to construct the subspace \eqref{eq:low_bound}. It follows that $\rho^{(j)}(\mu_i) = 0\, $ for every $\, i \in \{ 1, \dots , j \}$.
\end{lemma}
\begin{proof}.
Let $i \in \{1, \dots, j\}$. We have 
$\lambda^{(i)}_k = \lambda^{{\mathcal V}_j}_k(\mu_i)$ 
for $k = 1, \dots, \ell$ (see \cite[Lem.~2.3]{KanMMM18} in
the infinite dimensional setting, extending to finite dimension by
using the standard inner product in ${\mathbb C}^n$).
By the Courant–Fischer theorem \cite[Thm.~4.2.11]{Horn1985}, if $w$ is an eigenvector of $V_j^\ast A(\mu_i) V_j$ corresponding to $\lambda^{{\mathcal V}_j}_k(\mu_i)$, then $V_j w$ is an eigenvector of $A(\mu_i)$ corresponding to $\lambda^{(i)}_k$.
It follows that
\vskip -3ex
\begin{align*}
	A(\mu_i) U_j(\mu_i) 
	&= A(\mu_i)
	\begin{bmatrix}
		V_j w^{V_j}_1(\mu_i) & \dots & V_j w^{V_j}_{\sr}(\mu_i)
	\end{bmatrix} \\
	&= 
	\begin{bmatrix}
		\lambda^{(i)}_1 V_j w^{V_j}_1(\mu_i) & \dots & \lambda^{(i)}_{\sr} V_j w^{V_j}_{\sr}(\mu_i)
	\end{bmatrix} % \\
	% &
        = U_j(\mu_i) \cdot \diag\left( \lambda^{(i)}_1, \ldots, \lambda^{(i)}_{\sr} \right),
\end{align*}
\vskip -1ex
\noindent
implying
$\rho^{(j)}(\mu_i) = \| U_j^{\bot}(\mu_i)^\ast A(\mu_i) U_j(\mu_i) \| = \| U_j^{\bot}(\mu_i)^\ast U_j(\mu_i) \cdot \diag(\lambda^{(i)}_1, \ldots, \lambda^{(i)}_{\sr}) \| = 0.$
\end{proof}

The lower bound \eqref{eq:low_bound1} is not practical, as $\lambda^{{\mathcal U}_j^{\bot}(\mu)}_1(\mu)$ involves computing the smallest eigenvalue of the large matrix $U_j^{\bot}(\mu)^\ast A(\mu) U_j^{\bot}(\mu)$, almost as expensive as computing $\lambda_1(\mu)$.
A remedy to this difficulty is observing that $f^{(j)}(\eta)$ defined in \eqref{eq:low_bound1} is monotonically increasing \cite[Lem.~3.1]{SirK16}, implying any $\eta^{(j)}(\mu) \leq \lambda^{{\mathcal U}_j^{\bot}(\mu)}_1(\mu)$ (cheaply computable) yields the lower bound:
$f^{(j)}\left( \eta^{(j)}(\mu) \right) \leq \lambda_1(\mu).$
The next subsection explains an efficient way to obtain a lower bound $\eta^{(j)}(\mu)$ 
satisfying $\eta^{(j)}(\mu) \leq \lambda^{{\mathcal U}_j^{\bot}(\mu)}_1(\mu)$.% efficiently.

%%%%%%%%%%

\subsubsection{Determining \texorpdfstring{$\eta^{(j)}(\mu)$}{TEXT} such that \texorpdfstring{$\eta^{(j)}(\mu)\leq \lambda^{{\mathcal U}_j^{\bot}(\mu)}_1(\mu)$.}{TEXT}}
To efficiently approximate $\lambda^{{\mathcal U}_j^{\bot}(\mu)}_1(\mu)$ with a lower bound, 
we adopt the optimization-based technique from \cite[Sec.~3]{SirK16}
analogous to the \SCM \cite{HuyRSP07}. 
To be specific, for any $\widehat{\mu} \in \mathcal{D}$, we have
\[
	\lambda^{{\mathcal U}_j^{\bot}(\mu)}_1(\widehat{\mu}) 
	= \min_{z \in \mathbb{C}^{\sn - r} \setminus \{ 0 \}} 
	\sum_{m = 1}^\kappa \theta_m(\widehat{\mu})  
	\frac{z^* U_j^\bot(\mu)^* A_m U_j^\bot(\mu) z}{z^* z}.
\]
The minimization problem above can alternatively be rewritten as 
\begin{equation}\label{eq:bigproject_var}
	\lambda^{{\mathcal U}_j^{\bot}(\mu)}_1(\widehat{\mu}) 
	= \min_{y \in \mathcal{Y}_j(\mu)} \theta(\widehat{\mu})^T y,
\end{equation}
where $\theta(\widehat{\mu}) := [\theta_1(\widehat{\mu}), \dots, \theta_\kappa(\widehat{\mu})]^T$ and
$\mathcal{Y}_j(\mu) := \range(\mathcal{Q}_j(\mu))$
for the mapping $\mathcal{Q}_j(\mu) : \mathbb{C}^{\sn - r} \setminus \{ 0 \} \to \mathbb{R}^\kappa$,
\[
	\mathcal{Q}_j(\mu)(z) := \left[
		\frac{z^* U_j^\bot(\mu)^* A_1 U_j^\bot(\mu) z}{z^* z}, \dots,
		\frac{z^* U_j^\bot(\mu)^* A_\kappa U_j^\bot(\mu) z}{z^* z}
	\right].
\]
What makes the minimization problem difficult is the nonconvex nature 
of the feasible region $\mathcal{Y}_j(\mu)$. We next explain 
an outer polyhedral approximation of this nonconvex feasible region. Replacing $\mathcal{Y}_j(\mu)$
with this polyhedron leads to a linear program whose solution gives a lower bound for 
$\lambda^{{\mathcal U}_j^{\bot}(\mu)}_1(\widehat{\mu})$ as desired.

To relax $\mathcal{Y}_j(\mu)$ into a polyhedron,
 it is shown in \cite[Lem.~3.2]{SirK16} that for each $i \in \{ 1, \dots, j \}$, 
\begin{equation}\label{eq:beta_i}
	\begin{split}
		& \lambda^{{\mathcal U}_j^{\bot}(\mu)}_1(\mu_i)
		\;\geq\;
		\lambda^{(i)}_1 + \beta^{(i,j)}(\mu) \,, \quad \text{where} \\[0.7em]
		& \beta^{(i,j)}(\mu) \;:=\;
		\lambda_{\min} \left( 
			\Lambda^{(i)} - \lambda^{(i)}_1 I_\ell
			- [V^{(i)}]^* U_j(\mu) U_j(\mu)^* V^{(i)} (\Lambda^{(i)} - \lambda^{(i)}_{\ell+1} I_\ell)
		\right) \,,
	\end{split}
\end{equation}
with $\Lambda^{(i)} := \diag(\lambda^{(i)}_1, \dots, \lambda^{(i)}_\ell), \;
	V^{(i)} := [ 
				\begin{array}{ccc}
					v^{(i)}_1 & \cdots & v^{(i)}_\ell 
				\end{array}
			]$ holds.
Thus, for any $i \in \{ 1, \dots, j \}$, we have:
\[
	\lambda^{{\mathcal U}_j^{\bot}(\mu)}_1(\mu_i)
	= \min_{y \in \mathcal{Y}_j(\mu)} \theta(\mu_i)^T y
	\quad \Rightarrow \quad
	\theta(\mu_i)^T y \geq 
	\lambda^{{\mathcal U}_j^{\bot}(\mu)}_1(\mu_i) \geq
	\lambda^{(i)}_1 + \beta^{(i,j)}(\mu)
	\quad \forall y \in \mathcal{Y}_j(\mu),
\]
where the last inequality follows from \eqref{eq:beta_i}.
Furthermore, for any $y \in \mathcal{Y}_j(\mu)$, the entries $y_i$ are Rayleigh 
quotients and hence bounded by the spectrum of $A_i$, i.e.,
$	y_i \in [\lambda_{\min}(A_i), \lambda_{\max}(A_i)]$.
This implies that $\mathcal{Y}_j(\mu) \subseteq \mathcal{B} := [\lambda_{\min}(A_1), \lambda_{\max}(A_1)] \times \cdots \times [\lambda_{\min}(A_\kappa), \lambda_{\max}(A_\kappa)].
$
Hence, we define
\begin{equation}\label{eq:LP_polytope}
	\mathcal{Y}^{(j)}_{\mathrm{LB}}(\mu)
	:= \left\{ y \in \mathcal{B} \; \mid \; 
		\theta(\mu_i)^T y \geq \lambda^{(i)}_1 + \beta^{(i,j)}(\mu) \quad \text{for all } i = 1, \dots, j \right\}
\end{equation}
 as the polyhedral outer approximation satisfying 
 $\mathcal{Y}_j(\mu) \subseteq \mathcal{Y}^{(j)}_{\mathrm{LB}}(\mu)$.
Replacing  $\mathcal{Y}_j(\mu)$ in (\ref{eq:bigproject_var}) with $\mathcal{Y}^{(j)}_{\mathrm{LB}}(\mu)$
leads to a lower bound on $\lambda^{{\mathcal U}_j^{\bot}(\mu)}_1(\mu)$ in terms
of a linear program,  
which is stated formally below, and which is first established in \cite[Sec.~3.2.1]{SirK16}.
\begin{theorem}\label{thm:LB1}
For every $\mu \in \mathcal{D}$, the following inequality holds
(with $\mathcal{Y}^{(j)}_{\mathrm{LB}}(\mu)$ given by \eqref{eq:LP_polytope}):
\begin{equation}\label{eq:LP}
	\lambda^{{\mathcal U}_j^{\bot}(\mu)}_1(\mu)
	\;\geq\;
	\eta_\ast^{(j)}(\mu)
	\;:=\;
	\min \left\{ \theta(\mu)^T y \mid y \in \mathcal{Y}^{(j)}_{\mathrm{LB}}(\mu) \right\} .
\end{equation}
\end{theorem}
The linear program (i.e., the minimization problem) in~\eqref{eq:LP},
since its feasible region ${\mathcal Y}^{(j)}_{\rm{LB}}(\mu)$ is compact,
must attain its minimum. We denote the minimizer by $y^{(j)}(\mu) \in \mathbb{R}^{\kappa}$ 
throughout this text.
Using the bound $\eta^{(j)}_\ast(\mu) \leq \lambda^{{\mathcal U}_j^{\bot}(\mu)}_1(\mu)$,
monotonicity of $f^{(j)}$, and inequality in \eqref{eq:low_bound1}, we deduce:
\begin{equation}\label{eq:defn_LB}
	\lambda^{(j)}_{\rm{LB}}(\mu) \; \vcentcolon= \;
	f^{(j)}(\eta^{(j)}_\ast(\mu))
	\;\leq\;
	\lambda_1(\mu).
\end{equation}
Algorithm \ref{alg:lb} outlines the computation of $\lambda^{(j)}_{\rm{LB}}(\mu)$ based on the definition (\ref{eq:defn_LB}).

Important properties of $\eta^{(j)}_\ast(\mu)$ and $\beta^{(i,j)}(\mu)$ are listed in the following result.

\begin{lemma}\label{lem:props_eta}
The following assertions hold for $\eta^{(j)}_\ast(\mu)$ and $\beta^{(i,j)}(\mu)$ defined as in~\eqref{eq:LP} and~\eqref{eq:beta_i}, respectively.
\begin{enumerate}
	\item\label{item1}
	$\beta^{(i,j)}(\mu) \geq 0 \:$ for every $\mu \in {\mathcal D}$ and $i \in \{ 1, \dots , j \}$.
	
	\item\label{item2}
	$\eta^{(j)}_\ast(\mu_i) \geq \lambda^{(i)}_1 \,$ for every $i \in \{ 1, \dots , j \}$.
	
	\item\label{item3}
	If $\lambda^{(i)}_1$ is a simple eigenvalue of $A(\mu^{(i)})$, then
	$\beta^{(i,j)}(\mu_i) > 0 \: $ for every $i \in  \{ 1, \dots , j \}$.
	
	\item\label{item4}
	If $\lambda^{(i)}_1$ is a simple eigenvalue of $A(\mu^{(i)})$, then
	$\eta^{(j)}_\ast(\mu_i) > \lambda^{(i)}_1 \,$ for every $i \in \{ 1, \dots , j \}$.
	
	\item
	For every $i \in \{ 1, \dots , j \}$,
	if $\lambda^{(i)}_{\sr +1} > \lambda^{(i)}_{\sr}$, then 
	\begin{equation}\label{eq:beta_lb}
		\beta^{(i,j)}(\mu_i) = \lambda^{(i)}_{\sr+1} - \lambda^{(i)}_1 > 0.
	\end{equation}

	\item \label{item6}
	For every $i \in \{ 1, \dots , j \}$,
	if $\lambda^{(i)}_{\sr +1} > \lambda^{(i)}_{\sr}$, then
	\begin{equation}\label{eq:lbound_eta}
		\eta^{(j)}_\ast(\mu_i) \;\geq\; \lambda^{(i)}_{\sr +1}.
	\end{equation}
\end{enumerate}
\end{lemma}

\begin{proof}. We proceed point-by-point.
\begin{enumerate}
	\item 
	Observe that $\beta^{(i,j)}(\mu)$ is the smallest eigenvalue of
	\[
		M^{(i)}_j(\mu) \;\vcentcolon=\;
		(\Lambda^{(i)} - \lambda^{(i)}_1 I_\ell) +
		B^{(i)}_j(\mu) B^{(i)}_j(\mu)^\ast (\lambda^{(i)}_{\ell+1} I_\ell - \Lambda^{(i)}),
	\]
	where $B^{(i)}_j(\mu) = [V^{(i)}]^\ast U_j(\mu)$. Moreover, $M^{(i)}_j(\mu)$ is similar to the Hermitian matrix
	\begin{equation}\label{eq:def_tilM}
	\begin{split}
		\widetilde{M}^{(i)}_j(\mu) \;\vcentcolon=\; &
		(\lambda^{(i)}_{\ell+1} I_\ell - \Lambda^{(i)})^{1/2} 
			M^{(i)}_j(\mu)
		(\lambda^{(i)}_{\ell+1} I_\ell - \Lambda^{(i)})^{-1/2} \\[0.5em]
		= \; &
		(\Lambda^{(i)} - \lambda^{(i)}_1 I_\ell) + 
		(\lambda^{(i)}_{\ell+1} I_\ell - \Lambda^{(i)})^{1/2} 
		B^{(i)}_j(\mu) B^{(i)}_j(\mu)^\ast
		(\lambda^{(i)}_{\ell+1} I_\ell - \Lambda^{(i)})^{1/2},
	\end{split}
	\end{equation}
	which is Hermitian positive semidefinite. Hence, $\beta^{(i,j)}(\mu) \geq 0$.

	\item 
	This follows from
	\begin{equation}\label{eq:eta_vs_l1}
		\eta^{(j)}_\ast(\mu_i) = \theta(\mu_i)^T y^{(j)}(\mu_i) 
		\;\geq\; \lambda^{(i)}_1 + \beta^{(i,j)}(\mu_i)
		\;\geq\; \lambda^{(i)}_1,
	\end{equation}
	where the first inequality is due to the definition of $\mathcal{Y}^{(j)}_{\mathrm{LB}}(\mu_i)$
	and $y^{(j)}(\mu_i) \in \mathcal{Y}^{(j)}_{\mathrm{LB}}(\mu_i)$, while the second 
	inequality is due to part~\ref{item1}.

	\item We proceed as in part \ref{item1}. Now the first column of $U_j(\mu_i)$ is an
		eigenvector of $A(\mu_i)$ corresponding to $\lambda^{(i)}_1$. By simplicity
		assumption on $\lambda^{(i)}_1$, this first column is $c v^{(i)}_1$ for some
		$c \in {\mathbb C}$ such that $|c| = 1$, where the eigenvector $v^{(i)}_1$
		is the first column of $V^{(i)}$. By the orthonormality of the columns of $U_j(\mu_i)$
		and $V^{(i)}$, the first column and row of $B^{(i)}_j(\mu_i) = [ V^{(i)} ]^\ast U_j(\mu_i)$
		must be zero except the $(1,1)$ entry which is $c$. The same holds for 
		$B^{(i)}_j(\mu_i) B^{(i)}_j(\mu_i)^\ast$ with the $(1,1)$ entry equal to $| c |^2$, also for 
        \[C \;\vcentcolon=\; (\lambda^{(i)}_{\ell+1} I_\ell - \Lambda^{(i)})^{1/2} 
					B^{(i)}_j(\mu) B^{(i)}_j(\mu)^\ast (\lambda^{(i)}_{\ell+1} I_\ell - \Lambda^{(i)})^{1/2}
                    \]
		with $(1,1)$ entry equal to $|c|^2 (\lambda^{(i)}_{\ell+1} - \lambda^{(i)}_1)$. Now consider $z^\ast \widetilde{M}^{(i)}_j(\mu_i) z$ for any nonzero $z \in {\mathbb C}^{\ell}$,
		where $\widetilde{M}^{(i)}_j(\mu_i)$ is defined as in \eqref{eq:def_tilM}.
		If the first entry of $z$, say $z_1$, is not zero, then letting $\widetilde{z} \in {\mathbb C}^{\ell-1}$ 
		the vector formed of the remaining entries of $z$ excluding its first entry, we have
		\[
			\hskip 5ex
			z^\ast  C z = | z_1 |^2 |c|^2 (\lambda^{(i)}_{\ell+1} - \lambda^{(i)}_1)
					\;	+	\;
							\widetilde{z}^\ast C(2:\ell,2:\ell) \widetilde{z}
					\;\; \geq \;\;	| z_1 |^2 |c|^2 (\lambda^{(i)}_{\ell+1} - \lambda^{(i)}_1) \;\; > \;\; 0,
		\]
		so $z^\ast \widetilde{M}^{(i)}_j(\mu_i) z > 0$. If $z_1=0$, at least one
		of the remaining entries of $z$ is not zero, so
		\[
			z^{\, \ast}
			(\Lambda^{(i)} - \lambda^{(i)}_1 I_\ell)
			z	\;	>	\;	0 \,	,
		\]
		and again $z^\ast \widetilde{M}^{(i)}_j(\mu_i) z > 0$.
		This implies that the smallest eigenvalues of $\widetilde{M}^{(i)}_j(\mu_i)$
		and $M^{(i)}_j(\mu_i)$ are positive, so $\beta^{(i,j)}(\mu_i) > 0$.			
		\smallskip					
		\item 
		This follows from a line of reasoning similar to part \ref{item2}. Specifically,
		(\ref{eq:eta_vs_l1}) holds, but now the last inequality in (\ref{eq:eta_vs_l1}) 
		is satisfied strictly as $\beta^{(i,j)}(\mu_i) > 0$ from part \ref{item3}.		
		\smallskip
		\item Due to the assumption $\: \lambda^{(i)}_{\sr +1} > \lambda^{(i)}_{\sr} \:$, the columns of $U_j(\mu_i)$ 
		and the first $\sr$ columns of $V^{(i)}$ form orthonormal bases for the same invariant 
		subspace of $A(\mu_i)$, namely
		$\text{Null}(A(\mu_i) - \lambda^{(i)}_1 I) \oplus \dots \oplus \text{Null}(A(\mu_i) - \lambda^{(i)}_{\sr} I)$.
		Hence, there is an $\sr \times \sr$ unitary matrix $Q$ such that
		$U_j(\mu_i)	\;	=	\;	 \, V^{(i)}(1:\sn ,  1:\sr) \, Q \:$.
		Now let us first suppose $\ell > \sr$.
		By the orthonormality of the columns of $V^{(i)}$, we have
		\[
			B^{(i)}_j(\mu_i)
				\;	:=	\;
				[ V^{(i)} ]^\ast U_j(\mu_i)
				\;	=	\;
				[ V^{(i)} ]^\ast  V^{(i)}(1:\sn ,  1:\sr) \, Q
				\;	=	\;
			\left[
				\begin{array}{c}
					Q	\\
					0
				\end{array}
			\right]	\:	,
		\]
		which in turn implies
		\begin{align*}
			\begin{aligned}
			C \; 	\vcentcolon= & \; (\lambda^{(i)}_{\ell+1} I_\ell - \Lambda^{(i)})^{1/2} 
				B^{(i)}_j(\mu_i) B^{(i)}_j(\mu_i)^\ast (\lambda^{(i)}_{\ell+1} I_\ell - \Lambda^{(i)})^{1/2}% \\	
			\;	= 	% &\;
	\begin{bmatrix}
					\lambda^{(i)}_{\ell+1} I_{\sr} - \Lambda^{(i)}(1:\sr,1:\sr)	 	&	\;\; 0		\\
						0	&	\;\; 0
				\end{bmatrix}	\:	
			\end{aligned}
		\end{align*}
		so that, recalling (\ref{eq:def_tilM}),
		\begin{equation*}
		\begin{split}
			\widetilde{M}^{(i)}_j(\mu_i)	\;	&	=	\;
			  (\Lambda^{(i)} - \lambda^{(i)}_1 I_\ell)	+	C\;	=\;
				\begin{bmatrix}
					(\lambda^{(i)}_{\ell+1} - \lambda^{(i)}_1)  I_\sr 	 	&	 0		\\
						0	&	 \Lambda^{(i)}(\sr+1:\ell, \sr+1:\ell) - \lambda^{(i)}_1 I_{\ell - \sr}
				\end{bmatrix}	\:	.
		\end{split}
		\end{equation*}	
		Hence, $\beta^{(i,j)}(\mu_i)$, that is the smallest eigenvalue of $M^{(i)}_j(\mu_i)$,
		is also the smallest eigenvalue of $\widetilde{M}^{(i)}_j(\mu)$, which is $\lambda^{(i)}_{\sr +1} - \lambda^{(i)}_1$. If $\ell = \sr$, following the steps of the derivation above, we have $B^{(i)}_j(\mu_i) = Q$,
		$C = \lambda^{(i)}_{\sr+1} I_{\sr} - \Lambda^{(i)}$, 
		$\widetilde{M}^{(i)}_j(\mu) = (\lambda^{(i)}_{\sr+1} - \lambda^{(i)}_1)  I_\sr$, 
		so again $\beta^{(i,j)}(\mu_i) = \lambda^{(i)}_{\sr+1} - \lambda^{(i)}_1$. 
		\smallskip
		\item
		This follows from arguments similar to those used in part \ref{item2}. In particular,
		\eqref{eq:eta_vs_l1} holds, but, using \eqref{eq:beta_lb}, the last inequality in \eqref{eq:eta_vs_l1} can be replaced by $\lambda^{(i)}_1 + \beta^{(i,j)}(\mu_i)  
						\: = \: 
			\lambda^{(i)}_{\sr+1}$.
	\end{enumerate}
        This concludes the proof.
    \end{proof}

Exactly as in the upper bound case, the lower bound $\lambda^{(j)}_{\rm LB}(\mu)$ 
defined in~\eqref{eq:defn_LB} interpolates $\lambda_{\min}(\mu)$ at the points
$\mu_1, \dots , \mu_j \in {\calD}$ in the Hermite sense, which we formally present in the
next theorem. Note that, for part \ref{item:lbound_dinter} of the theorem, a different proof is also stated in \cite[Thm.~3.6]{SirK16}.

\begin{theorem}\label{thm:Hermite_int_LB}
For $i = 1,\dots,j$, we have:
\begin{enumerate}
    \item \label{item:lbound_inter}
    $\lambda^{(j)}_{\rm LB}(\mu_i) = \lambda_{\min}(\mu_i)$.
    \item \label{item:lbound_dinter}
    If $\lambda_{\min}(\mu_i)$ is simple, then
    $\lambda^{{\mathcal V}_j}_1(\mu) , \lambda^{(j)}_{\rm LB}(\mu)$ are differentiable at $\mu_i$ and
    $\nabla \lambda^{(j)}_{\rm LB}(\mu_i) = 
        \nabla \lambda^{{\mathcal V}_j}_1(\mu_i) = 
        \nabla \lambda_{\min}(\mu_i)$.
\end{enumerate}
\end{theorem}
\begin{proof}$\;\;$ \\
\begin{enumerate}
\item
From \eqref{eq:interpolate}, we know $\lambda^{{\mathcal V}_j}_1(\mu_i) = \lambda^{(i)}_1$ as 
$\mathcal{V}_j$ contains an eigenvector of $A(\mu_i)$ associated with $\lambda^{(i)}_1$, and
we have $\eta^{(j)}_\ast(\mu_i) \geq \lambda^{(i)}_1 = \lambda^{{\mathcal V}_j}_1(\mu_i)$,
where the inequality follows from part \ref{item2} of \Cref{lem:props_eta}. Also, \Cref{thm:rho0} gives $\rho^{(j)}(\mu_i) = 0$. Recalling the definition of $f^{(j)}(\eta)$ from (\ref{eq:low_bound1}),
these imply the interpolation property of the lower bound at $\mu_i$:
\[
\lambda^{(j)}_{\rm{LB}}(\mu_i) = f^{(j)}(\eta^{(j)}_\ast(\mu_i)) = \min\left\{ \lambda^{{\mathcal V}_j}_1(\mu_i), \eta^{(j)}_\ast(\mu_i) \right\} = \lambda_{\min}(\mu_i).
\]

\item
Assume now that $\lambda_{\min}(\mu_i)$ is a simple eigenvalue of $A(\mu_i)$. Then, by part \ref{item4} of \Cref{lem:props_eta}, we have $\eta^{(j)}_\ast(\mu_i) > \lambda^{{\mathcal V}_j}_1(\mu_i)$. Since both $\eta^{(j)}_\ast(\mu)$ and $\lambda^{{\mathcal V}_j}_1(\mu)$ vary continuously with respect to $\mu$, this strict inequality holds in a neighborhood of $\mu_i$.
Hence, for $\mu$ near $\mu_i$, the lower bound is expressed as
\begin{equation}\label{eq:LB_near_mu}
\lambda^{(j)}_{\rm{LB}}(\mu) = \lambda^{{\mathcal V}_j}_1(\mu) -
\frac{2\rho^{(j)}(\mu)^2}{\left| \lambda^{{\mathcal V}_j}_1(\mu) - \eta^{(j)}_\ast(\mu) \right| + \sqrt{ \left( \lambda^{{\mathcal V}_j}_1(\mu) - \eta^{(j)}_\ast(\mu) \right)^2 + 4\rho^{(j)}(\mu)^2 }}.
\end{equation}

\noindent
Simplicity of $\lambda_{\min}(\mu_i)$ implies simplicity 
of $\lambda^{{\mathcal V}_j}_1(\mu_i)$.
As $\lambda^{{\mathcal V}_j}_1(\mu_i)$ is a simple eigenvalue of $A^{V_j}(\mu_i)$, 
the function $\lambda^{{\mathcal V}_j}_1(\mu)$ is differentiable at $\mu = \mu_i$. 
From \Cref{thm:rho0}, $\rho^{(j)}(\mu_i) = 0$, so differentiating \eqref{eq:LB_near_mu} 
at $\mu = \mu_i$ yields $\nabla \lambda^{(j)}_{\rm{LB}}(\mu_i) = \nabla \lambda^{{\mathcal V}_j}_1(\mu_i)$.
Finally, \eqref{eq:interpolate} gives 
$\nabla \lambda^{{\mathcal V}_j}_1(\mu_i) = \nabla \lambda_{\min}(\mu_i)$.
\end{enumerate}
This completes the proof.
\end{proof}

\begin{algorithm}[H]
\caption{Computation of a Lower Bound for $\lambda_{\min}(\mu)$}
\label{alg:lb}
\begin{algorithmic}[1]
\Require $\mu, \mu_1, \dots , \mu_j \in \mathcal{D}$, $\ell \in \mathbb{N}$; eigenvalues $\lambda_k(\mu_i)$ and eigenvectors $v_k(\mu_i)$ for $k = 1, \dots, \ell$, $i = 1, \dots, j$; $\lambda_{\ell+1}(\mu_i)$ for $i = 1, \dots, j$; bounds $\lambda_{\min}(A_i)$, $\lambda_{\max}(A_i)$; orthonormal basis $V_j$ for $\mathcal{V}_j$ as in \eqref{eq:low_bound}.
\Ensure $\lambda^{(j)}_{\rm{LB}}(\mu)$ as defined in \eqref{eq:defn_LB}.
\State Compute the smallest $\sr$ eigenpairs $(\lambda^{{\mathcal V}_j}_k(\mu), w^{V_j}_k(\mu))$ of $V_j^* A(\mu) V_j$ for $k = 1, \dots, \sr$.
\State Set $U_j(\mu) \gets [ V_j w^{V_j}_1(\mu) \; \cdots \; V_j w^{V_j}_\sr(\mu) ]$.
\vskip .3ex
\State Compute $A(\mu) U_j(\mu)$ and then $\rho^{(j)}(\mu)^2$ via \eqref{eq:rhomu}.
\vskip .6ex
\State Compute $\beta^{(i,j)}(\mu)$ from \eqref{eq:beta_i} for $i = 1, \dots, j$.
\vskip .5ex
\State Solve the linear program \eqref{eq:LP} to obtain $\eta^{(j)}_\ast(\mu)$. \label{alg:lb:lp}
\vskip .5ex
\State Compute $\lambda^{(j)}_{\rm{LB}}(\mu) \gets f^{(j)}(\eta^{(j)}_\ast(\mu))$, where $f^{(j)}$ is as in \eqref{eq:low_bound1}.
\end{algorithmic}
\end{algorithm}

\begin{remark}\label{rmk:int}
Let $\lambda_{\rm SCM}^{(j)}(\mu)$ be the lower bound computed via 
the \textsc{SCM} method \cite{HuyRSP07}, which is indeed $\eta^{(j)}_\ast(\mu)$ 
from \eqref{eq:LP}, but with $\beta^{(i,j)}(\mu) = 0$ for every $i \in \{ 1, \dots, j \}$. 
In \cite[eq.~(3.9)]{SirK16}, it is claimed -- without proof -- that
\begin{equation}\label{eqn:false:rel}
\lambda_{\rm SCM}^{(j)}(\mu) \le \lambda_{\LB}^{(j)}(\mu).
\end{equation}
It is also claimed in \cite{SirK16} that the interpolation property $\lambda^{(j)}_{\rm LB}(\mu_i) = \lambda_{\min}(\mu_i)$
holds based on this inequality, as $\lambda^{(j)}_{\rm SCM}(\mu_i) = \lambda_{\min}(\mu_i)$.
However, our numerical results suggest that this inequality does \emph{not} always hold; see \Cref{sec:NE:exm1}.
Motivated by this, it appears that a sharper lower bound is given by
\begin{equation}\label{eqn:new:LB}
	\widetilde{\lambda}_{\LB}^{(j)}(\mu) := 
		\max \left\{ \lambda_{\LB}^{(j)}(\mu), \lambda_{\rm SCM}^{(j)}(\mu) \right\}.
\end{equation}
While this new bound is tighter, it requires solving two linear programs per evaluation of $\mu$, 
increasing computational cost. As shown in \Cref{sec:NE:exm2}, these linear programs form a substantial 
portion of the total runtime of the framework we propose in the next section.
Moreover, this improved accuracy does \emph{not} necessarily reduce the total iteration 
count required for convergence. In fact, overall computation time typically \emph{increases} 
when using \eqref{eqn:new:LB}.
Thus, we use in our framework the original lower bound $\lambda_{\LB}^{(j)}(\mu)$ 
(rather than $\widetilde{\lambda}_{\LB}^{(j)}(\mu)$ above), which has the additional 
advantage of satisfying Hermite interpolation properties (\Cref{thm:Hermite_int_LB}), 
unlike $\lambda_{\rm SCM}^{(j)}(\mu)$. See \Cref{figINT} in \Cref{sec:NE:exm1} for an 
illustrative example.
\end{remark}

%%---------------------------------------------------------------------------
%%---------------------------------------------------------------------------

\section{A Subspace Framework}\label{sec3}
The bounds described in the previous section are used in \cite{SirK16}
to form a subspace ${\mathcal V}$ such that
$\lambda_{\min}^{\mathcal V}(\mu)$ is an approximation for $\lambda_{\min}(\mu)$ for all $\mu \in {\mathcal D}$. 
The practice used in \cite{SirK16} is a greedy procedure to gradually reduce the maximal error (i.e., the maximal gap between the upper and lower bounds of the previous section) in
a discrete subset of ${\mathcal D}$, similar to the strategy adopted in \SCM \cite{HuyRSP07}.
Formally, at the $j$-th iteration, given a set of points $\mu_1, \dots , \mu_j \in {\mathcal D}$ and a subspace
${\mathcal V}_j$ as in (\ref{eq:low_bound}) constructed based on these points,  
the estimate for the maximal error
\begin{equation} \label{eq:errXi}
\max_{\mu \in \Xi} \;
\frac{ \left( \lambda^{{\mathcal V}_j}_{\min}(\mu)	-	\lambda^{(j)}_{\rm{LB}}(\mu) \right) }
{ | \lambda^{{\mathcal V}_j}_{\min}(\mu) |}
\end{equation}
on a finite subset $\Xi$ of ${\mathcal D}$ is computed. The points  $\mu_1, \dots , \mu_j$ are enriched with $\mu_{j+1}\in \Xi$,
which is a maximizer of the maximization problem in \eqref{eq:errXi}, and the subspace ${\mathcal V}_j$ is 
expanded into ${\mathcal V}_{j+1}$
with the inclusion of the eigenvectors of $A(\mu_{j+1})$ corresponding to its $\ell$-smallest eigenvalues. 
Then the ($j+1$)-st iteration is carried out similarly. This procedure is repeated until the estimate for the 
maximal error (i.e., (\ref{eq:errXi}) for some $j$) is less than a prescribed tolerance. 

Here, we propose to maximize the error estimate over the whole domain ${\mathcal D}$
rather than over a finite set $\Xi \subset {\mathcal D}$. That is, at iteration $j$, we solve
\begin{equation}\label{def:H}  %\label{eq:max_error}
		\max_{\mu \in {\mathcal D}} \;H^{(j)}(\mu),\quad\;\; \text{where}\;\;\;
		H^{(j)}(\mu)	\:	\vcentcolon=			\:
				\lambda^{{\mathcal V}_j}_{\min}(\mu)	-	\lambda^{(j)}_{\rm{LB}}(\mu)	\:,
\end{equation}
employing the software \EIGOPT \cite{MenYK14}. 
Afterward, the subspace ${\mathcal V}_j$ and the points $\mu_1, \dots , \mu_j$ are
updated as described in the previous paragraph, but using the maximizer $\mu$ for the optimization
problem in (\ref{def:H}). We thus propose the framework outlined in \Cref{alg:sf}. Note that the only large size problem that needs to be solved
at every iteration of the proposed subspace framework is the large-scale computation
of the eigenvalues and eigenvectors of $A(\mu_{j+1})$ in line \ref{alg:large_eigs}.
On the other hand, the maximization problem in line \ref{alg:sub_prob} requires
the computation of the smallest eigenvalue of the small-scale matrix $A^{V_{j}}(\mu)$,
and the solution of the linear program needed for  $\lambda^{(j)}_{\rm{LB}}(\mu)$
(see \Cref{alg:lb}) for several values of $\mu$.
\begin{algorithm}
	\begin{algorithmic}[1]
		\Require{The real analytic scalar functions $\theta_i(\mu) : {\mathbb R}^{p} \rightarrow {\mathbb R}$,
			Hermitian matrices $A_i \in {\mathbb C}^{\sn \times \sn}$ for $i = 1, \dots , \kappa$
			s.t. $A(\mu) = \theta_1(\mu) A_1 + \dots + \theta_\kappa(\mu) A_\kappa$; compact
			domain ${\mathcal D} \subset {\mathbb R}^{p}$; 
			$\ell\in\N$; termination tolerance $\varepsilon$.}
		\Ensure{A reduced matrix-valued function $A^V(\mu)$ and the subspace 
			${\mathcal V} = \text{Col} (V)$ such that 
			$\max_{\mu \in {\mathcal D}} \lambda^{\mathcal V}_{\min}(\mu) - \lambda_{\min}(\mu) \leq \varepsilon$.}	
		\vskip .7ex
		\State{Compute $\lambda_{\min}(A_i)$ and $\lambda_{\max}(A_i)$ for $i = 1, \dots , \kappa$.}
		\vskip .3ex
		\State{Choose the initial point $ \mu_{1} $, and let $P_1 \gets \{ \mu_1 \}$.}\label{line2}
		\vskip .3ex
		\State{Compute $\lambda_{k}(\mu_1)$, $v_{k}(\mu_1)$
			for $k = 1, \dots , \ell$, and $\lambda_{\ell+1}(\mu_{1})$.}\label{line3}
		\vskip .3ex
		\State{$
				V_1 \gets 
			\text{orth}
			\left(
			\left[
			\begin{array}{ccccccc}
				v_1(\mu_{1})  &  \dots  & v_\ell(\mu_{1})	
			\end{array}
			\right]
			\right) \;\;$  and  
			$\;\;\: 
			{\mathcal V}_1 \gets \text{span}\{ v_1(\mu_{1}), \:\, \dots \:\, , v_\ell(\mu_{1}) \}$.}\label{line4}
		\vskip .2ex
		\For{$j = 1, 2, \dots$}
		\vskip .2ex
		\State{Solve the maximization problem
		\[
			\max_{\mu \in {\mathcal D}} \;
			\lambda^{\mathcal{V}_{j}}_{\min}(\mu)	-	\lambda^{(j)}_{\rm{LB}}(\mu) \: ;
		\]	
		\hskip 3.2ex	
		see (\ref{eq:defn_LB}) for the definition of $\lambda^{(j)}_{\rm{LB}}(\mu)$ together with
		(\ref{eq:low_bound1}) and (\ref{eq:LP}). \phantom{aaaaaaaaaaaaaaaaaaaaa}
		\phantom{aa}$\;\;\;$\hskip -.35ex
		Let $\: \varepsilon_j :=
				\max_{\mu \in {\mathcal D}} \;
			\lambda^{\mathcal{V}_{j}}_{\min}(\mu)	-	\lambda^{(j)}_{\rm{LB}}(\mu) \:$, $\:$and
		$\; \mu_{j+1} := \argmax_{\mu \in {\mathcal D}} \;
			\lambda^{\mathcal{V}_{j}}_{\min}(\mu)	-	\lambda^{(j)}_{\rm{LB}}(\mu) \:$.}
		 \label{alg:sub_prob}
		 \vskip .9ex
        		\If{	$\varepsilon_j \leq \varepsilon \:$}
                \State \textbf{Terminate} with $A^{V_{j}}(\mu) = V_{j}^\ast A(\mu) V_{j}$ and 
				${\mathcal V}_{j}$.
                \EndIf
                \vskip .3ex
		\State Include $\mu_{j+1}$ in the set of points, i.e., $P_{j+1} \gets P_j \cup \{ \mu_{j+1} \}$. \label{eqn:set:int:pnt}
		\vskip .3ex
		\State{Compute $\lambda_{k}(\mu_{j+1})$, $v_{k}(\mu_{j+1})$
			for $k = 1, \dots , \ell$, and $\lambda_{\ell+1}(\mu_{j+1})$.\label{alg:large_eigs}} 
		\vskip .3ex
		\State{$V_{j+1} \gets 
			\text{orth}
			\left(
			\left[
			\begin{array}{cccc}
				V_{j}  & v_1(\mu_{j+1})  &  \dots  &  v_\ell(\mu_{j+1})
			\end{array}
			\right]
			\right) \;\;$ and $\;\;\: {\mathcal V}_{j+1} \gets \text{Col}(V_{j+1})$.}
			\label{alg:sf:line:orth}
		\vskip .3ex
		\EndFor
	\end{algorithmic}
	\caption{Subspace framework for uniform approximation of $\lambda_{\min}(\mu)$ over ${\mathcal D}$}
	\label{alg:sf}
\end{algorithm}

In the next section, we focus on the convergence of the framework, in particular, investigating
the gap $H^{(j)}(\mu)$ between the bounds 
$\lambda^{\mathcal{V}_{j}}_{\min}(\mu)$ and $\lambda^{(j)}_{\rm{LB}}(\mu)$
generated by \Cref{alg:sf}. Our aim is to show that $\max_{\mu \in {\mathcal D}} H^{(j)}(\mu)$ 
gets closer to zero as $j$ gets larger. This also has implications about
the actual error 
\begin{equation}\label{eqn:act:err}
   \mathcal{E}^{(j)}(\mu)
   	\;	\vcentcolon=	\;
	\mathcal{E}^{\calV_j}(\mu)
   	\;	\vcentcolon=	\;
	\lambda^{\mathcal{V}_{j}}_{\min}(\mu)  -  \lambda_{\min}(\mu),
\end{equation}
decaying to zero uniformly over all $\mu \in {\mathcal D}$, as $H^{(j)}(\mu)$ bounds 
the actual error from above. To this end, the following result is helpful.

\begin{theorem}\label{thm:inter_0}
	Regarding \Cref{alg:sf}, for every $j \geq 1$ and $i = 1, \dots, j$, the following assertions hold:
	\begin{enumerate}
		\item   $H^{(j)}(\mu_i)	\;	=	\;	0$.
		\item   If $\lambda_{\min}(\mu_i)$ is simple, then $H^{(j)}(\mu)$ is differentiable at $\mu_i$
		with $\nabla H^{(j)}(\mu_i) = 0$.	
	\end{enumerate}
\end{theorem}
\begin{proof}$\:$
It follows from \Cref{thm:Hermite_int_LB} that 
	\begin{enumerate}
		\item[(i)] $\lambda^{(j)}_{\rm{LB}}(\mu_i) = \lambda_{\min}(\mu_i)$, and 
		\item[(ii)] if $\lambda_{\min}(\mu_i)$ is simple, 
		then $ \lambda^{{\mathcal V}_j}_{\min}(\mu), \lambda^{(j)}_{\rm{LB}}(\mu)$
		are differentiable at $\mu_i$ with 
		$\nabla \lambda^{(j)}_{\rm{LB}}(\mu_i) = \nabla \lambda^{{\mathcal V}_j}_{\min}(\mu_i)$.
	\end{enumerate}	
	Moreover, since the eigenvector corresponding to the smallest eigenvalue $\lambda_{\min}(\mu_i)$
	is included in ${\mathcal V}_j$, from (\ref{eq:interpolate}), we have
	$\lambda_{\min}(\mu_i) \: = \: \lambda^{{\mathcal V}_j}_{\min}(\mu_i)$. 
	Thus, it is immediate from the definition of $H^{(j)}(\mu)$ in (\ref{def:H})
	that $H^{(j)}(\mu_i) = 0$, and if $\lambda_{\min}(\mu_i)$ is simple, $\nabla H^{(j)}(\mu_i) = 0$.
\end{proof}
\begin{remark}\label{rem:initialize}
We remark that the framework outlined in \Cref{alg:sf} starts with only one point $\mu_1$
and corresponding initial subspace ${\mathcal V}_1$ in lines \ref{line2} and \ref{line4}, respectively.
Alternatively, one can initiate the framework with multiple points in line \ref{line2}, say $\mu_{1,1}, \dots, \mu_{1,\eta}$
so that $P_1 \gets \{ \mu_{1,1}, \dots, \mu_{1,\eta} \}$, 
then compute the eigenvalues, eigenvectors at these points in line \ref{line3}, and form the initial subspace $\calV_1$ accordingly as
${\mathcal V}_1 \gets \oplus_{i=1}^\eta \text{span}\{ v_1(\mu_{1,i}), \:\, \dots \:\, , v_\ell(\mu_{1,i}) \} \,$,
as well as the corresponding matrix $V_1$ with orthonormal columns
in line \ref{line4}. Inequality constraints corresponding to these points with normals 
$\theta(\mu_{1,1}), \dots , \theta(\mu_{1,\eta})$ should also be incorporated into the
linear programs determining the lower bounds. Related to \Cref{thm:inter_0},
associated with these initialization points, we have $H^{(j)}(\mu_{1,i}) = 0$, and,
if $\lambda_{\min}(\mu_{1,i})$ is simple, we additionally have $\nabla H^{(j)}(\mu_{1,i}) = 0$ for $i = 1, \dots, \eta$. 
\end{remark}
\begin{remark}
In the optimization problem in (\ref{def:H}), the objective $H^{(j)}(\mu)$, a surrogate
for the absolute error, could also be replaced by its relative counterpart defined as
    \begin{equation}\label{eqn:rel:H}
        H_r^{(j)}(\mu)
        		\;\vcentcolon=\;
	\frac{\left(\lambda^{{\mathcal V}_j}_{\min}(\mu)	-	\lambda^{(j)}_{\rm{LB}}(\mu)\right)}
		{|\lambda^{\calV_j}_{\min}(\mu)|}.
    \end{equation}
This formulation is preferable when one is more interested in having error guarantees in terms of the relative actual error
    \begin{equation}\label{eqn:rel:err}
        {\calE}^{(j)}_r(\mu)\;\vcentcolon=\;
        {\calE}^{\calV_j}_r(\mu)
   	\;	\vcentcolon=	\;\frac{\left(
	\lambda^{{\mathcal V}_j}_{\min}(\mu)	-	\lambda_{\min}(\mu)\right)}{|\lambda^{\calV_j}_{\min}(\mu)|},
    \end{equation}
since \eqref{eqn:rel:H} is naturally an upper bound for \eqref{eqn:rel:err}. For simplicity, our theoretical
analysis in \Cref{sec4}, as well as the Lipschitz continuity arguments for the 
surrogate error in \Cref{sec5} on which it is based, are 
performed employing the absolute surrogate error $H^{(j)}(\mu)$
and the associated optimization problem in \eqref{def:H}, although it could
be naturally extended to this relative formulation based on the surrogate error $H_r^{(j)}(\mu)$ in (\ref{eqn:rel:H}). 
For instance, \Cref{thm:inter_0} also holds when the instances of $H^{(j)}$ are replaced
by $H^{(j)}_r$ under the additional condition that $\lambda_{\min}(\mu_i) \neq 0$.
\end{remark}
\begin{remark}[Dynamic choice of $\ell$]\label{rmk:adp:ell}
    In \Cref{alg:sf}, the value of $\ell$ is fixed over the iterations. However, the 
    separability of the smallest eigenvalue from the $(\ell+1)$-st smallest eigenvalue of $A(\mu)$ is crucial 
    for the convergence of the algorithm, i.e., we need that the condition 
    \begin{equation}\label{eqn:spe:gap:def}
	   \calG_\ell(\mu) \; \vcentcolon = \; \lambda_{\ell+1}(\mu)-\lambda_{\min}(\mu) \; > \;0,
    \end{equation}
    holds. Small values of $\calG_\ell(\mu)$ may influence the accuracy of the 
    lower bound % \eqref{eq:LB_nearmui} 
    from a numerical perspective due to finite precision arithmetic. 
    One possibility to mitigate this issue is to dynamically choose $\ell$ along the iterations $j$ of \Cref{alg:sf}. To this end,
    we choose $\ell$ as follows:
    At iteration $j$ of \Cref{alg:sf} right before line~\ref{alg:large_eigs}, 
    we first set $\ell(j)=1$. If $\, \calG_{\ell(j)}(\mu_{j+1})$ is smaller than or equal to a user-prescribed value, 
    we then update $\ell(j)$ as $\ell(j) = \ell(j)+1$ and evaluate again $\calG_{\ell(j)}(\mu_{j+1})$. 
    We iterate until $\calG_{\ell(j)}(\mu_{j+1})$ is larger than the prescribed value. We note that a 
    necessary condition for the existence of $\ell\in\N$ such that \eqref{eqn:spe:gap:def} holds 
    is that $A(\mu)$ is not $zI_{\sn}$ for some $z\in{\mathbb C}$ for all $\mu\in\calD$. 
    However, for the subspace procedure to effectively provide small subspaces, it is crucial 
    that \eqref{eqn:spe:gap:def} is verified for $\ell\ll\sn$. It is feasible to set $\ell(j)>1$ by default, which would enhance the convergence of \Cref{alg:sf} with fewer iterations compared to when $\ell(j)=1$. However, this adjustment might cause the subspace $\calV_j$ to be augmented with eigenvectors that do not significantly contribute to the approximation of the smallest eigenvalues, thereby increasing the dimension of $\calV_j$ relative to the default setting of $\ell(j)=1$. An increased dimension of $\calV_j$ would decelerate both the offline and online phases, which is why this alternative is not favored. \\
    We conclude by noting that \cite[Sec.~2.1]{ManSZ25} provides an extensive discussion linking the subspace framework for computing the smallest eigenvalue of parametric Hermitian matrices with projection-based model order reduction in the context of parametric problems. From these connections, it becomes evident that eigenvalue crossings tend to slow down convergence and, in a natural way, lead to the construction of larger approximation subspaces.
\end{remark}

%-----------------------------------------------------------------------------%

\section{Global Convergence of the Subspace Framework}\label{sec4}
Let us now consider the actual error $\mathcal{E}^{(j)}(\mu)$ as in \eqref{eqn:act:err}, 
and the maximal actual error 
\begin{equation}\label{eqn:max:act:err}
\overline{\mathcal{E}}^{(j)}
			\vcentcolon
		=	\max_{\mu \in \mathcal{D}} \: \mathcal{E}^{(j)}(\mu)
\end{equation}
of the reduced eigenvalue function $\lambda_{\min}^{\mathcal{V}_j}(\mu)$ at the end of the $j$-th
subspace iteration. Note that for all $\mu \in \mathcal{D}$, as 
$\lambda_{\min}^{\mathcal{V}_j}(\mu) \geq \lambda_{\min}(\mu) \geq \lambda^{(j)}_{\rm{LB}}(\mu)$, we have
\begin{equation}\label{eq:bound_acterror}
	H^{(j)}(\mu)		\;	\geq		\;		\mathcal{E}^{(j)}(\mu) 	\;		\geq		\;	0	\,	.
\end{equation}
Our convergence results are built on the following monotonicity assumption.
\begin{assumption}\label{ass:monotone}
The following inequality holds for every $\mu \in {\mathcal D}$ and every integer $j \geq 1$:
\[
	H^{(j+1)}(\mu)	\;	\leq	\;	 	H^{(j)}(\mu)		\:	.
\]
\end{assumption}
\begin{remark}
The above assumption may appear strong at first glance. Recalling
\[
	H^{(j)}(\mu)	 = 
			\lambda^{{\mathcal V}_j}_{\min}(\mu) - \lambda_{\rm{LB}}^{(j)}(\mu),
\]
we indeed have $\lambda^{{\mathcal V}_{j+1}}_{\min}(\mu) \leq \lambda^{{\mathcal V}_j}_{\min}(\mu)$,
since ${\mathcal V}_{j+1} \supseteq {\mathcal V}_j$ (by an argument similar to
that in (\ref{eq:monotone})). However, it does not seem clear that
$\lambda_{\rm{LB}}^{(j+1)}(\mu) \geq \lambda_{\rm{LB}}^{(j)}(\mu)$ holds due to the non-monotonicity of $\beta^{(i,j)}(\mu)$, defined in \eqref{eq:beta_i}, with respect to $j$.
One way of ensuring the satisfaction of \Cref{ass:monotone} is
to replace the lower bound $\lambda_{\rm{LB}}^{(j)}(\mu)$ with the lower bound
$\overline{\lambda}_{\rm{LB}}^{(j)}(\mu) := \max_{k = 1, \dots , j} \lambda_{\rm{LB}}^{(k)}(\mu)$
in line \ref{alg:sub_prob} of \Cref{alg:sf} and in the definition of $H^{(j)}(\mu)$
in (\ref{def:H}). 
However, \Cref{ass:monotone} is always satisfied after the first few iterations of \Cref{alg:sf}
in our numerical experiments, even without such a modification, and we always observe it when $H^{(j)}(\mu)$ is sufficiently small. 
\end{remark}

\begin{comment}
We first focus on a condition that guarantees $\overline{\mathcal{E}}^{(l)} = 0$ 
for every integer $l$ large enough.
\begin{theorem}
	Suppose that \Cref{ass:monotone} holds and that $\mu_{j} = \mu_{j+s}$ for some integer $s \geq 1$. Then we have
	\[
	\overline{\mathcal{E}}^{(j+m)}	 \; 	=	 \;  0			
	\]
	for every integer $m \geq s-1$.
\end{theorem}
\begin{proof}$\:$
	Observe that
	\[
	\max_{\mu \in \mathcal{D}} \: H^{(j+s-1)}(\mu)	\;	=	\;	H^{(j+s-1)}(\mu_{j+s})
	\;	=	\;	H^{(j+s-1)}(\mu_{j})	\;	\leq		\;	H^{(j)}(\mu_j)		\;	=	\;		0	\,	
	\]
	where the inequality is due to \Cref{ass:monotone}, and the last equality 
	due to \Cref{thm:inter_0}. Hence, $\max_{\mu \in \mathcal{D}} \: H^{(j+s-1)}(\mu)	 =	0$, 
	which together with \Cref{ass:monotone} imply
	$
	\max_{\mu \in \mathcal{D}} \: H^{(j+m)}(\mu)	\; = \;  0
	$
	for all $m \geq s-1$. Finally, it follows from (\ref{eq:bound_acterror}) that
	\[
	\overline{\mathcal{E}}^{(j+m)}	\;	=	\;	
	\max_{\mu \in \mathcal{D}} \: \mathcal{E}^{(j+m)}(\mu)	\; = \;    0
	\]
	for all $m \geq s-1$.
\end{proof}

\end{comment}
We are interested in showing the convergence results for the infinite dimensional case. With this aim we assume that
$A_i : \ell_2(\mathbb{N}) \rightarrow \ell_2(\mathbb{N})$ is a compact self-adjoint
operator for $i = 1, \dots, \kappa$.
Intuitively, each $A_i$ can be considered as an infinite-dimensional Hermitian
matrix. We assume, without loss of generality, that 
$A(\mu) = \sum_{i=1}^\kappa \theta_i(\mu) A_i$ has a negative eigenvalue for all $\mu$
for the well-posedness of $\lambda_{\min}(\mu)$ at every $\mu$ (i.e., 0 is
an accumulation point of the point spectrum of every compact self-adjoint operator
\cite[page 185, Thm.~6.26]{Kato1995}, \cite[Sec.~1.3]{KanMMM18}, so
the infimum of the eigenvalues of $A(\mu)$ is zero and not attained if all eigenvalues 
of $A(\mu)$ are positive).

The next result asserts that the maximal actual error $\overline{\mathcal{E}}^{(j)}$ 
defined as in \eqref{eqn:max:act:err} 
decays to zero in the infinite-dimensional setting in the limit as $j\rightarrow \infty$.
This main convergence result is proven
under the monotonicity assumption on $H^{(j)}(\mu)$ (i.e., \Cref{ass:monotone}), as well as 
assumptions that ensure the uniform Lipschitz continuity of $H^{(j)}(\mu)$ with respect to $j$
stated formally in the appendix (i.e., \Cref{ass:eig_sep} and \Cref{ass:svals_away_zero}).
\begin{theorem}\label{teo0}
	Suppose \Cref{alg:sf} in the infinite-dimensional setting described above
	generates a sequence $\{ \mu_j \}$ and a sequence of subspaces $\{ \mathcal{V}_j \}$
	such that \Cref{ass:monotone}, as well as \Cref{ass:eig_sep} and \Cref{ass:svals_away_zero}
	in \Cref{sec5} are satisfied. Then the sequences 
	$\{ \mu_j \}$ and  $\{ \mathcal{V}_j \}$ generated are such that
	\[
	\lim_{j \rightarrow \infty}  \overline{\mathcal{E}}^{(j)} 		\;	=	\;
	\lim_{j \rightarrow \infty}  \mathcal{E}^{(j)}(\mu_{j+1})		\;	=	\;	
	\lim_{j \rightarrow \infty}  H^{(j)}(\mu_{j+1})		\;	=	\;	0.
	\]
\end{theorem}
\begin{proof}$\:$
	First, we show that $\lim_{j \rightarrow \infty}  H^{(j)}(\mu_{j+1}) = 0$. Since the members of the
sequence $\{ \mu_j \}$ lie in the compact set $\mathcal{D}$, it must have a convergent subsequence, say $\{ \mu_{\ell_j} \}$. Moreover, $H^{(\ell_j)}(\mu_{\ell_j}) = 0$ by \Cref{thm:inter_0}.
	By the uniform Lipschitz continuity of $H^{(j)}(\mu)$ over all $j$ 
	(i.e., by \Cref{thm:main_Lips_C} in \Cref{sec5}), there exists $\gamma$ such that for all $j$ we have
	\[
	H^{(\ell_j)}(\mu_{\ell_{j+1}})
		\;\;	=	\;\;
	| H^{(\ell_j)}(\mu_{\ell_{j+1}})  -   H^{(\ell_j)}(\mu_{\ell_j})   |
	\;\;	\leq		\;\;
	\gamma	\|   \mu_{\ell_{j+1}}  -   \mu_{\ell_j}	\|	\,	.
	\]
	Additionally, by \Cref{ass:monotone}, we have 
	$H^{(\ell_{j+1}-1)}(\mu_{\ell_{j+1}}) \leq H^{(\ell_j)}(\mu_{\ell_{j+1}})$ so that
	\[
	H^{(\ell_{j+1}-1)}(\mu_{\ell_{j+1}})	\;\;	\leq		\;\;	\gamma	\|   \mu_{\ell_{j+1}}  -   \mu_{\ell_j}	\|
	\quad\quad	\Longrightarrow		\quad\quad
	\lim_{j\rightarrow \infty}	\:
	H^{(\ell_{j+1}-1)}(\mu_{\ell_{j+1}})
	\;\;	=	\;\;	0.
	\]
  Now, since the sequence $\{ H^{(j)}(\mu_{j+1}) \} = \{ \max_{\mu \in \mathcal{D}} H^{(j)}(\mu) \}$ 
  is monotonically non-increasing by \Cref{ass:monotone}, and is bounded below by 0,
	it must be convergent. As shown above, the subsequence $\{ H^{(\ell_{j+1}-1)}(\mu_{\ell_{j+1}}) \}$
	of the convergent sequence $\{ H^{(j)}(\mu_{j+1}) \}$ converges to 0, so
	\begin{equation}\label{eq:Hj_conv}
		\lim_{j\rightarrow \infty}	\:	H^{(j)}(\mu_{j+1})		\;\;	=	\;\;	0
	\end{equation} 
	as well. By \eqref{eq:bound_acterror}, we have $0 \leq \mathcal{E}^{(j)}(\mu_{j+1}) \leq H^{(j)}(\mu_{j+1})$, which
	together with \eqref{eq:Hj_conv} imply $\lim_{j\rightarrow \infty}	\:	\mathcal{E}^{(j)}(\mu_{j+1}) = 0$.	
	Similarly, 
	\[
	0 \;\; \leq  \;\; \overline{\mathcal{E}}^{(j)} = \max_{\mu \in \mathcal{D}} \: \mathcal{E}^{(j)}(\mu) \;\;  \leq  \;\; 
	\max_{\mu \in \mathcal{D}} \: H^{(j)}(\mu) 	 = H^{(j)}(\mu_{j+1}) \, ,
	\]
	where the second inequality is again due to \eqref{eq:bound_acterror}. Hence, it follows again from
	\eqref{eq:Hj_conv} that $\lim_{j\rightarrow \infty} \: \overline{\mathcal{E}}^{(j)}  = 0$, completing the proof.
\end{proof}

%-----------------------------------------------------------------------------%

%-----------------------------------------------------------------------------%

\section{Approximation for the smallest singular value}\label{sec7}

When the matrices in the sum \eqref{int1} are not Hermitian, it is natural to replace the problem of approximating 
the smallest eigenvalue with the approximation of the smallest singular value $\sigma_{\min}(\mu)$
of the matrix $A(\mu)$. This reformulation arises naturally in connection with 
the a posteriori error estimate used to construct reduced spaces. Indeed, suppose we have 
\begin{equation*}
	A(\mu) x(\mu)\;=\;b(\mu)
\end{equation*}
with $A(\mu)$ invertible and non-Hermitian $\forall \mu \in \calD$. This could represent a linear system 
arising from the discretization of an advection-diffusion \PDE. If we inject the solution $\widetilde{x}(\mu)$, obtained 
by solving a reduced problem, we get
\begin{align*}
	r(\mu)\;=&\;A(\mu) \widetilde{x}(\mu)\;-b(\mu)
	\;=\;\;A(\mu) \left\{ \widetilde{x}(\mu)-x(\mu) \right\}	\:	,
\end{align*}
which, by taking the norms, leads to
\begin{equation}\label{eqn:aposterior:non:co}
    \|\widetilde{x}(\mu)-x(\mu)\|\;\le\;\|A(\mu)^{-1}\|\|r(\mu)\|=\frac{\|r(\mu)\|}{\sigma_{\min}(\mu)}	\:	.
\end{equation}
The expression \eqref{eqn:aposterior:non:co} gives the a posteriori error estimate for the discrete problem 
arising from a non-coercive \PDE. In the variational formulation discretization setting (e.g. \FEM, \RBM), the 
scaling factor in \eqref{eqn:aposterior:non:co} coincides with the discrete inf-sup stability constant, that 
is defined as
\begin{align}\label{sec7.7}
	\begin{aligned}
		\beta(\mu)
		\;\; \vcentcolon=&\;\;\min_{u\in{\mathbb C}^\sn,\;\|u\|=1} \max_{v\in {\mathbb C}^\sn\;\|v\|=1} \left| u^\ast A(\mu)v \right| \: .
	\end{aligned}
\end{align}
It is straightforward to verify that $\beta(\mu) = \sigma_{\min}(\mu)$. 
Thus, the expression in \eqref{sec7.7} suggests natural lower and upper bounds for $\sigma_{\min}(\mu)$; 
indeed, given two subspaces $\calU,\calV\subseteq{\mathbb C}^\sn$, we have
\begin{equation}\label{sec7.1}
\begin{split}
	\sigma_{\LB}(\mu)\;\vcentcolon=&\;\min_{u\in{\mathbb C}^\sn,\;\|u\|=1} 
			\max_{v\in \calV\;\|v\|=1}  \left| u^\ast A(\mu)v \right| \;\le\;\sigma_{\min}(\mu) , \\
	\sigma_{\UB}(\mu)\;\vcentcolon=&\; \min_{u\in\calU,\;\|u\|=1} 
			\max_{v\in {\mathbb C}^\sn\;\|v\|=1} \left| u^* A(\mu)v \right| \;\ge\;\sigma_{\min}(\mu).
\end{split}
\end{equation}
However, efficient computation of the bounds in \eqref{sec7.1} 
do not appear straightforward. 

On the other hand, the smallest singular value $\sigma_{\min}(A(\mu) V)$ of $A(\mu)V$ 
for a given matrix $V \in {\mathbb C}^{\sn \times \sd}$ with orthonormal columns yields 
an upper bound for $\sigma_{\min}(\mu)$. This conclusion can be drawn from the variational characterization
\begin{equation}\label{eq:sigma_ubound}
\begin{split}
	\sigma_{\min}(A^V_{\rm{R}}(\mu))\;	&	= \; \sqrt{\min_{v\in{\mathbb C}^{\sd},\;\|v\|=1} v^* V^* A(\mu)^* A(\mu)V v}	\\
	\;\; &	\ge \; \sqrt{\min_{v\in{\mathbb C}^\sn,\;\|v\|=1} v^* A(\mu)^* A(\mu) v}\;=\;\sigma_{\min}(\mu)	\,	,
\end{split}
\end{equation}
where
\begin{equation}\label{eq:sig_red}
	A^V_{\rm{R}}(\mu)
		\;	:=	\;
	A(\mu) V
		\;	=	\;
	\theta_1(\mu) A_1 V +	\dots		+	\theta_{\kappa}(\mu) A_{\kappa} V	\:	.
\end{equation}

\subsection{Standard approach for singular values}\label{sec:sv:issues}
In literature, there have been several attempts to approximate the smallest singular value 
of a parameter-dependent matrix on a domain via \SCM type-methods; 
see \cite{HuyKCHP10,RozPM13} and \cite[Sec.~5]{SirK16}. 
Some of them (e.g., \cite{SirK16}) rely on working with the Hermitian parameter-dependent matrix 
$\widehat{A}(\mu)\vcentcolon=A^\ast (\mu)A(\mu)$, that has an affine decomposition of type \eqref{int1} 
involving $\kappa(1+\kappa)/2$ Hermitian matrices. For instance, one can apply \Cref{alg:sf}  to 
$\widehat{A}(\mu) = A^\ast (\mu)A(\mu)$,
and generate a subspace to approximate $\lambda_{\min}(\widehat{A}(\mu))=\sigma^2_{\min}(A(\mu))$. 
However, such an approach has the following drawbacks.
\begin{enumerate}
   \item  The presence of ${\kappa(1+\kappa)}/2$ terms in the affine decomposition makes 
    the stable evaluation of the squared residual norm $\rho(\mu)^2$ more expensive in terms of computational time and storage costs. Indeed, the stable evaluation of $\rho(\mu)^2$ has a computational cost that scales as the square of the number of affine terms in the decomposition \cite{BuhEOR14}, i.e. $\calO(\kappa^4)$ with this approach. The effectiveness of \Cref{alg:sf} demands $\kappa^2\ll\sn$ when employing the stable residual assessment. In this scenario, the requirement becomes more stringent, specifically $\kappa^4\ll\sn$, thereby significantly limiting the range of problems that can be addressed using \Cref{alg:sf}. 

\item Linear programming (\LP) is frequently observed to have inadequate convergence in this context, meaning that the lower bound derived through \LP yields meaningful outcomes only when the subspace dimensionality no longer adheres to the condition $\sr\ll\sn$; see \Cref{figPS_BSb}. This occurs because $(A(\mu)^{*}A(\mu))$ maintains a symmetric positive semidefinite nature for every $\mu\in\R^{p}$, a critical constraint that is not clear how to transfer within the linear programming framework. Linear programming attempts to approximate the smallest eigenvalue over a set significantly larger than the joint numerical range, failing to incorporate $\lambda_1(A(\mu)^{*}A(\mu))\ge0$ for all $\mu\in\R^p$. Consequently, the lower bounds it produces tend to be overly pessimistic, often being less than or equal to zero. \label{it3} \\
\end{enumerate}

\subsection{A two-sided procedure}\label{sub:sec:6.2}
We propose a two-sided projection procedure that overcomes the downsides of the standard 
approach listed in \Cref{sec:sv:issues}; however, it does not guarantee a uniformly small error over the parametric domain.

Let $\calU, \; \calV$ be appropriately chosen 
subspaces of ${\mathbb C}^{\sn}$ of equal dimension, and $U$, $V$ be matrices 
whose columns form orthonormal bases for these subspaces. Let us consider
\begin{align}
	& \sigma^{{\mathcal V}}_{\min}(\mu)\;\vcentcolon=  \sigma_{\min}(A^V_{\rm{R}}(\mu))
	= \sqrt{  v^V_{\rm{UB}}(\mu)^\ast V^* A^\ast(\mu) A(\mu)V v^V_{\rm{UB}}(\mu)  }  \, ,  \label{eqn:SV:UB} \\
	& \sigma^{{\mathcal V}, {\mathcal U}}(\mu)\;\vcentcolon= 
	\sqrt{ 
			v^V_{\rm{UB}}(\mu)^\ast V^* A^\ast(\mu)UU^* A(\mu)V v^V_{\rm{UB}}(\mu)} \, , \label{eqn:SV:LB}
\end{align}
where $v^V_{\rm{UB}}(\mu)$ is a unit right singular vector corresponding to the smallest singular
value of $A^V_{\rm{R}}(\mu) = A(\mu) V$.
Note that the quantity $\sigma^{{\mathcal V}, {\mathcal U}}(\mu)$ is a lower bound for $\sigma^{{\mathcal V}}_{\min}(\mu)$,
which in turn is an upper bound for $\sigma_{\min}(\mu)$ by the arguments in (\ref{eq:sigma_ubound}).
The framework outlined in \Cref{alg:sv} is a greedy procedure based on these two quantities.
In line \ref{alg:large_sigv} of this description, $\sigma_{k}(\mu_{j+1})$ denotes the $k$-th smallest
singular value of $A(\mu_{j+1})$, and $u_k(\mu_{j+1})$, $v_k(\mu_{j+1})$ corresponding consistent
unit left, unit right singular vectors of $A(\mu_{j+1})$, respectively.
\begin{algorithm}[t]
	\begin{algorithmic}[1]
		\Require{The real analytic functions $\theta_i(\mu) : {\mathbb R}^{p} \rightarrow {\mathbb R}$,
			 matrices $A_i \in {\mathbb C}^{\sn \times \sn}$ for $i = 1, \dots , \kappa$
			s.t. $A(\mu) = \theta_1(\mu) A_1 + \dots + \theta_\kappa(\mu) A_\kappa$; compact
			domain ${\mathcal D} \subset {\mathbb R}^{p}$; $\ell\in\N$; termination tolerance $\varepsilon$.}
		\Ensure{A reduced matrix-valued function $A^V_{\rm{R}}(\mu)$ as in \eqref{eq:sig_red}
			and the subspace ${\mathcal V} = \text{Col} (V)$. }	
		\vskip .7ex
		\State{Choose the initial point $\mu_1$, and let $P_1 \gets \{ \mu_1 \}$.}
		\vskip .3ex
		\State{Compute $\sigma_{k}(\mu_{1})$, $u_{k}(\mu_{1})$, $v_{k}(\mu_{1})$
			for $k = 1, \dots , \ell$.}
		\vskip .3ex
		\State{$V_1 \gets 
			\text{orth}
			\left(
			\left[
			\begin{array}{ccc}
				v_1(\mu_{1})  &  \dots  & v_\ell(\mu_{1})
			\end{array}
			\right]
			\right) \;$
			 and
			$\;\: {\mathcal V}_1 \gets \text{span}\{ v_1(\mu_{1}), \dots , v_\ell(\mu_{1}) \}$.}
		\vskip .3ex
  		 \State{$U_1 \gets 
			\text{orth}
			\left(
			\left[
			\begin{array}{ccc}
				u_1(\mu_{1})  &  \dots  & u_\ell(\mu_{1})
			\end{array}
			\right]
			\right) \;$ and
			$\;\: {\mathcal U}_1 \gets \text{span}\{ u_1(\mu_{1}), \dots , u_\ell(\mu_{1}) \}$.}
		\vskip .3ex
		\For{$j = 1, 2, \dots$}
		\vskip .3ex
		\State Solve the maximization problem
		\vskip -2ex
		\[
			\max_{\mu \in {\mathcal D}} \;
			\sigma^{\calV_{j}}_{\min}(\mu)	\,  -  \:	\sigma^{\calV_{j},\calU_{j}}(\mu)	\: .
		\]
		\vskip -1.5ex
		\hskip -1ex 
		Let $\: \varepsilon_j := \max_{\mu \in {\mathcal D}} \;
				\sigma^{\calV_{j}}_{\min}(\mu)	\,  -  \:	\sigma^{\calV_{j},\calU_{j}}(\mu) \:$, $\:$and 
			$\; \mu_{j+1} := \argmax_{\mu \in {\mathcal D}} \;
				\sigma^{\calV_{j}}_{\min}(\mu)	\,  -  \:	\sigma^{\calV_{j},\calU_{j}}(\mu)$ . \label{alg:sub_prob:sv}
		\vskip .3ex
        		\If{	$\varepsilon_j \leq \varepsilon$}
                \State \textbf{Terminate} with
						$A^{V_{j}}_{\rm{R}}(\mu)$ and ${\mathcal V}_{j}$ 	. 	\label{linesv:termin}	
		\vskip .5ex
		\EndIf
		\vskip .3ex
      		\State Include $\mu_{j+1}$ in the set of points, i.e., $P_{j+1} \gets P_j \cup \{ \mu_{j+1} \}$.
		\vskip .3ex
		\State{Compute $\sigma_k(\mu_{j+1})$, $u_{k}(\mu_{j+1})$, $v_{k}(\mu_{j+1})$
			for $k = 1, \dots , \ell$.\label{alg:large_sigv}} 
		\vskip .6ex
		\State{$V_{j+1} \gets 
			\text{orth}
			\left(
			\left[
			\begin{array}{cccc}
				V_{j}  & v_1(\mu_{j+1})  &  \dots  &  v_\ell(\mu_{j+1})
			\end{array}
			\right]
			\right) \;$ and $\; {\mathcal V}_{j+1} \gets \text{Col}(V_{j+1})$. \label{alg_sg:expand}}
		\vskip .3ex
		\State{$U_{j+1} \gets 
			\text{orth}
			\left(
			\left[
			\begin{array}{cccc}
				U_{j}  & u_1(\mu_{j+1})  &  \dots  &  u_\ell(\mu_{j+1})
			\end{array}
			\right]
			\right) \;$ and  $\; {\mathcal U}_{j+1} \gets \text{Col}(U_{j+1})$. \label{alg_sg:expand2}}
		\vskip .3ex
		\EndFor
		\vskip .3ex
	\end{algorithmic}
	\caption{Subspace framework for uniform approximation of $\sigma_{\min}(\mu)$ over ${\mathcal D}$}
	\label{alg:sv}
\end{algorithm}
At iteration $j$ of the framework, we compute the parameter maximizing the absolute 
surrogate error
\begin{equation}\label{eqn:sure:err:sv}	
	S^{(j)}(\mu)
		\;	:=	\;
	\sigma^{{\mathcal V}_j}_{\min}(\mu) - \sigma^{{\mathcal V}_j , {\mathcal U}_j}(\mu)
\end{equation}
for subspaces ${\mathcal V}_j$ and ${\mathcal U}_j$ of equal dimension constructed so far. Then these subspaces
are expanded in lines \ref{alg_sg:expand}-\ref{alg_sg:expand2}
based on the left and right singular vectors of $A(\mu)$ at the maximizing parameter value,
reminiscent of the subspace expansion strategy in \Cref{alg:sf} to approximate the smallest eigenvalue.
An alternative to the absolute surrogate error $S^{(j)}(\mu)$ at iteration $j$ is its relative counterpart defined as
\begin{equation}\label{eqn:sure:err:sv_rel}	
	S^{(j)}_r(\mu)\;\vcentcolon=\; 
	\frac{\left( \sigma^{{\mathcal V}_j}_{\min}(\mu) - \sigma^{{\mathcal V}_j , {\mathcal U}_j}(\mu) \right)}
			{\sigma^{{\mathcal V}_j}_{\min}(\mu)}.
\end{equation}
\Cref{alg:sv} can also be applied so that it is driven by the relative surrogate error,
in particular by maximizing $S^{(j)}_r(\mu)$ rather than $S^{(j)}(\mu)$ in line \ref{alg:sub_prob:sv}, and letting
$\varepsilon_j$, $\mu_{j+1}$ be the maximal value and the maximizer of $S^{(j)}_r(\mu)$.
Again, we note that \Cref{alg:sv} can be initialized with multiple points $\mu_{1,1}, \dots, \mu_{1,\eta}$
rather than only with $\mu_1$.

\begin{remark}
Note that $S_r^{(j)}(\mu)$ can be interpreted as a relative squared residual norm, in a similar fashion to $\rho^{(j)}(\mu)^2$. Indeed, it holds that 
   \begin{align}\label{eqn:rel:err:sv}
\begin{aligned}
    \frac{\|U_j^{\perp}A(\mu)V_jv^{V_j}_{\UB}(\mu)\|^2}{\sigma^{\calV_j}_{\min}(\mu)^2}\;=&\;\frac{\left(\sigma^{\calV_j}_{\min}(\mu)^2-\sigma^{\calV_j,\calU_j}_{\min}(\mu)^2\right)}{\sigma^{\calV_j}_{\min}(\mu)^{2}}  \;\\
    \ge&	 \; 
	\frac{\sigma^{\calV_j}_{\min}(\mu)-\sigma^{\calV_j,\calU_j}_{\min}(\mu)}{\sigma^{\calV_j}_{\min}(\mu)}\frac{\left(\sigma^{\calV_j}_{\min}(\mu)+\sigma^{\calV_j,\calU_j}_{\min}(\mu)\right)}{\left(\sigma^{\calV_j}_{\min}(\mu)+\sigma^{\calV_j,\calU_j}_{\min}(\mu)\right)}\;=\; S_r^{(j)}(\mu) \\
    \ge& \; \frac{\left(\sigma^{\calV_j}_{\min}(\mu)^2-\sigma^{\calV_j,\calU_j}_{\min}(\mu)^2\right)}{2\sigma^{\calV_j}_{\min}(\mu)^{2}} 
    \;=\; \frac{\|U_j^{\perp}A(\mu)V_jv^{V_j}_{\UB}(\mu)\|^2}{2\sigma^{\calV_j}_{\min}(\mu)^2}.
\end{aligned}
\end{align}
\end{remark}
\Cref{alg:sv} overcomes the issue listed since it does not evaluate terms that require $\calO(\kappa^4)$ floating point operations or storage memory costs and does not call any \LP routine. The downside is that, even though $\sigma^{{\mathcal V}, {\mathcal U}}(\mu)$ in \eqref{eqn:SV:LB}
is a lower bound for $\sigma^{{\mathcal V}}(\mu)$, it is not, in general, 
a lower bound for $\sigma_{\min}(\mu)$. Thus,  the termination criterion of \Cref{alg:sv} in line \ref{linesv:termin}
does not necessarily yield an upper bound for the actual error 
\begin{equation*}
	\max_{\mu \in {\mathcal D}} \;\left( \sigma^{{\mathcal V}_{j}}_{\min}(\mu) - \sigma_{\min}(\mu)\right)	\: 
\end{equation*} 
(or $\: \max_{\mu \in {\mathcal D}} \; \{ \sigma^{{\mathcal V}_{j}}_{\min}(\mu) - \sigma_{\min}(\mu) \} /
				\sigma^{{\mathcal V}_{j}}_{\min}(\mu) \:$
if the relative surrogate error $S^{(j)}_r(\mu)$ is used). However, in the next section, we show that the subspace constructed through \Cref{alg:sv} is such that the error $\sigma^{{\mathcal V}_{j}}_{\min}(\mu) - \sigma_{\min}(\mu)$ 
(or $\{ \sigma^{{\mathcal V}_{j}}_{\min}(\mu) - \sigma_{\min}(\mu) \} /
				\sigma^{{\mathcal V}_{j}}_{\min}(\mu)$) and its gradient
are generically zero at the selected interpolation points $\mu_1, \dots ,\mu_j$. 
Thus, replacing $\sigma_{\min}(\mu)$ with $\sigma^{\calV_j}_{\min}(\mu)$ is equivalent to using a Hermite-interpolant function, 
where the interpolation points are constructed in a greedy fashion by minimizing a residual type surrogate error function.

\subsection{Hermitian interpolation results for the bounds}
In this section, we show that the functions $\sigma^{{\mathcal V}_j}_{\min}(\mu)$ 
and $\sigma^{{\mathcal V}_j, {\mathcal U}_j}(\mu)$ at the $j$-th subspace iteration 
of \Cref{alg:sv} Hermite interpolate the actual smallest singular value function $\sigma_{\min}(\mu)$
at the points $\mu_1, \dots \mu_j$. 
\begin{theorem}
    The sequence of subspaces $\{ {\mathcal V}_j \}$, $\{ {\mathcal U}_j \}$, and the points $\{ \mu_j \}$
    by \Cref{alg:sv} are such that 
    \begin{equation}\label{eqn:int:UB}
       		\sigma^{{\mathcal V}_j}_{\min}(\mu_i) \;=\; \sigma_{\min}(\mu_i)
    \end{equation}
    for $i = 1, \dots , j$. 
    Moreover, if $\sigma_{\min}(\mu_i)$ is simple, then
     \begin{equation}\label{eqn:int:LB}
		\sigma^{{\mathcal V}_j , {\mathcal U}_j}(\mu_i) \; = \; \sigma_{\min}(\mu_i),
    \end{equation}
    and if $\sigma_{\min}(\mu_i)$ is simple and nonzero, then
     \begin{subequations}
       \begin{align}
  &		 \nabla\sigma^{{\mathcal V}_j}_{\min}(\mu_i)	\;=\;	\nabla\sigma_{\min}(\mu_i), \label{eqn:int:der:UB}  \\[.25em]
  &		 \nabla\sigma^{{\mathcal V}_j , {\mathcal U}_j}(\mu_i)	\;=\; \nabla\sigma_{\min}(\mu_i) \label{eqn:int:der:LB}
        \end{align}
\end{subequations}
for $i = 1, \dots , j$.       
\end{theorem}
\begin{proof}$\:$
    We start by showing \eqref{eqn:int:UB}. 
    To this end, observe that for every $\mu \in {\mathcal D}$ we have
    \begin{equation}\label{eq:sv_monotone}
    	\sigma_{\min}(\mu) \;\; = \min_{v \in {\mathbb C}^{\sn}, \| v \| = 1} \| A(\mu) v \|
		\; \leq \;	
	\min_{v \in {\mathcal V}_j, \| v \| = 1} \| A(\mu) v \|		\; = \;\;	\sigma^{{\mathcal V}_j}_{\min}(\mu).
    \end{equation}
    In particular, $\sigma_{\min}(\mu_i) \leq \sigma^{{\mathcal V}_j}_{\min}(\mu_i)$.
    Additionally, as $v_1(\mu_i) \in {\mathcal V}_j$, we deduce
    \[
    		\sigma_{\min}(\mu_i)		\;	=	\;
		\| A(\mu_i) v_1(\mu_i) \|		\;	\geq		\;
		\min_{v \in {\mathcal V}_j, \| v \| = 1} \| A(\mu_i) v \|	\;	=	\;
		\sigma^{{\mathcal V}_j}_{\min}(\mu_i).
    \]
    \vskip -1.9ex
    \noindent
    Hence, $\sigma_{\min}(\mu_i) = \sigma^{{\mathcal V}_j}_{\min}(\mu_i)$ as desired.

    Let us next prove (\ref{eqn:int:der:UB}). The simplicity of $\sigma_{\min}(\mu_i)$ implies
    the simplicity of $\sigma^{{\mathcal V}_j}_{\min}(\mu_i)$. Moreover, by assumption and \eqref{eqn:int:UB},
    we have $\sigma^{{\mathcal V}_j}_{\min}(\mu_i) \;=\; \sigma_{\min}(\mu_i) \neq 0$. Consequently, 
    both $\sigma^{{\mathcal V}_j}_{\min}(\mu)$ and $\sigma_{\min}(\mu)$ are differentiable at $\mu_i$.
%     As $v_1(\mu_i) \in {\mathcal V}_j$, there is $\alpha$ such that $v_1(\mu_i) = V_j \alpha$.
    Since $\sigma_{\min}(\mu_i)$ is a singular value of $A(\mu_i)$ with $u_1(\mu_i), v_1(\mu_i)$
    corresponding consistent unit left, unit right singular vectors, we have
    \begin{equation}\label{eq:sval_defn}
    		A(\mu_i) v_1(\mu_i) = \sigma_{\min}(\mu_i) u_1(\mu_i)	 \:	,	\quad	
		 u^\ast_1(\mu_i) A(\mu_i) = \sigma_{\min}(\mu_i) v_1(\mu_i)^\ast	\,	.
    \end{equation}
    As $v_1(\mu_i) \in {\mathcal V}_j$, there is a unit vector $\alpha$ such that $v_1(\mu_i) = V_j \alpha$, and the equations above can be rewritten as
    \[
    		A(\mu_i) V_j \alpha = \sigma_{\min}(\mu_i) u_1(\mu_i)	 \:	,	\quad
		 u^\ast_1(\mu_i) A(\mu_i) = \sigma_{\min}(\mu_i) \alpha^\ast V_j^\ast \: ,
    \]
    implying
    \[
    		A(\mu_i) V_j \alpha = \sigma_{\min}(\mu_i)	u_1(\mu_i)	 \:	,	\quad
		u^\ast_1(\mu_i) A(\mu_i) V_j  =  \sigma_{\min}(\mu_i) \alpha^\ast.
    \]
    This shows that $\alpha$ and $u_1(\mu_i)$ are consistent unit right and unit left singular vectors of 
    $A^{V_j}_{\rm{R}}(\mu_i) = A(\mu_i) V_j$
    corresponding to its smallest singular value $\sigma^{{\mathcal V}_j}_{\min}(\mu_i) \;=\; \sigma_{\min}(\mu_i)$.
    Using the analytical formulas for the derivative of a singular value function 
    (see for instance \cite[Sec.~3.3]{MenYK14}, \cite[Lem.~1]{GugLM21}), we deduce
    \begin{equation}
    \begin{split}
          \frac{\partial \sigma_{\min}}{\partial \mu^{(\ell)} }(\mu_i)
          	\;	&	=	\;
	    \real\left( u^\ast_1(\mu_i) \: \frac{\partial A \;}{\partial \mu^{(\ell)} }(\mu_i) \: v_1(\mu_i) \right)	
	    		\;	=	\;
	   \real\left( u^\ast_1(\mu_i) \: \frac{\partial A \;}{\partial \mu^{(\ell)} }(\mu_i) \: V_j \alpha \right)	\\
	    		&	=	\;
		\real\left( u^\ast_1(\mu_i) \: \frac{\partial A^{V_j}_{\rm{R}} \;}{\partial \mu^{(\ell)} }(\mu_i) \: \alpha \right)
		\;	=	\;	
	\frac{\partial \sigma^{{\mathcal V}_j}_{\min}}{\partial \mu^{(\ell)} }(\mu_i)
     \end{split}
     \end{equation}
     for $\ell = 1, \dots , p$. This proves (\ref{eqn:int:der:UB}).
     
     For proving (\ref{eqn:int:LB}), as shown in the previous paragraph $\alpha$ such that
     $v_1(\mu_i) = V_j \alpha$ and $u_1(\mu_i)$ form
     a pair of consistent unit right and unit left singular vectors of $A^{V_j}_{\rm{R}}(\mu_i) = A(\mu_i) V_j$
     corresponding to its smallest singular value $\sigma^{{\mathcal V}_j}_{\min}(\mu_i)$.
     Since $\sigma^{{\mathcal V}_j}_{\min}(\mu_i)$ is simple, we can assume, without loss of generality, 
     that $v_{\UB}^{V_j}(\mu) = \alpha$ and $u_1(\mu_i)$ is the corresponding consistent left
     singular vector of $A^{V_j}_{\rm{R}}(\mu_i)$ (more generally $v_{\UB}^{V_j}(\mu) = c \alpha$ with 
     the corresponding consistent left singular vector $c u_1(\mu_i)$ for some 
     $c \in {\mathbb C}$ such that $| c | = 1$,
     and the subsequent arguments still apply). Hence, we have
     \begin{equation*}
     	\begin{split}
    	\sigma^{{\mathcal V}_j,\calU_j}(\mu_i)
			\;	&	=	\;
	\sqrt{\alpha^*V_j^\ast A(\mu_i)^*U_jU_j^*A(\mu_i)V_j \alpha}	
			\;		=	\;
	\sqrt{v_1(\mu_i)^\ast A(\mu_i)^*U_jU_j^*A(\mu_i)v_1(\mu_i)}	\\	
			\;	&	=	\;
	\sqrt{\sigma_{\min}(\mu_i)^2 u_1(\mu_i)^*U_jU_j^*u_1(\mu_i)}\;=\;\sigma_{\min}(\mu_i),
	\end{split}
  \end{equation*}
  where the third equality is due to the left-hand equality in (\ref{eq:sval_defn}), and
  the last equality follows from $u_1(\mu_i)\in\calU_j$ so that $u_1(\mu_i)=$~$U_jU_j^* u_1(\mu_i)$. 
    
Finally, to show \eqref{eqn:int:der:LB}, for any $\mu \in {\mathcal D}$ such that
$\sigma^{{\mathcal V}_j}_{\min}(\mu)$
is simple and nonzero, we have
  \begin{equation}\label{eqn:der:eig:rig}
      0\;=\;\frac{\partial}{\partial \mu^{(\ell)}}\left(v_{\UB}^{V_j}(\mu)^*v_{\UB}^{V_j}(\mu)\right)
      			\;=\;
	2\real{\left(v_{\UB}^{V_j}(\mu)^* \frac{\partial v_{\UB}^{V_j}}{\partial \mu^{(\ell)}}(\mu)\right)} 
		\:	
  \end{equation}
  for $\ell = 1, \dots , p$.
  Moreover, for such a $\mu \in {\mathcal D}$,
   \begin{equation}\label{eq:sval_Hermite}
    \begin{split}
        \frac{\partial \sigma^{{\mathcal V}_j,\calU_j}}{\partial \mu^{(\ell)} }(\mu)
        		\;	&=	\;
	\frac{\partial }{\partial \mu^{(\ell)}}
	\left(\sqrt{v^{V_j}_{\UB}(\mu)^*V_j^*A(\mu)^*U_jU_j^*A(\mu)V_j v^{V_j}_{\UB}(\mu)} \right) 	\\
        		& = 
	\;\frac{1}{2\sigma^{{\mathcal V}_j,\calU_j}(\mu)}
	\left\{ 2\real{\left(v^{V_j}_{\UB}(\mu)^*V_j^* A(\mu)^* U_jU_j^* \frac{\partial A}
							{\partial \mu^{(\ell)}}(\mu) \, V_j v^{V_j}_{\UB}(\mu)\right)}	\right.	\\
		&	\quad\quad\quad\quad\quad\quad\quad	+	\left.
		2\real{\left(v^{V_j}_{\UB}(\mu)^*V_j^*A(\mu)^*U_jU_j^*
					A(\mu) V_j\frac{\partial v^{V_j}_{\UB}}{\partial \mu^{(\ell)}}(\mu) \right)}\right\}
    \end{split}
\end{equation}
for $\ell = 1, \dots , p$. Now take any $\mu_i, \: i \in \{1, 2, \dots , j \}$.
As in the previous paragraph, we assume, without loss of generality due to the simplicity assumption, 
$v^{V_j}_{\UB}(\mu_i) = \alpha$, where $\alpha$ is such that $v_1(\mu_i) = V_j \alpha$, and is a right singular vector of 
$A^{V_j}_{\rm{R}}(\mu_i)$ corresponding to its smallest singular value $\sigma^{{\mathcal V}_j}_{\min}(\mu_i)$.
In (\ref{eq:sval_Hermite}), the term in the last line is zero when $\mu = \mu_i$, since
\begin{equation}\label{eqn:zero:term}
\begin{split} 
\hskip -1.5ex
 	&
		\real{\left(v^{V_j}_{\UB}(\mu_i)^*V_j^*A(\mu_i)^*U_jU_j^*A(\mu_i)V_j\frac{\partial v_{\UB}^{V_j}}
					{\partial \mu^{(\ell)}}(\mu_i)\right)}		\;	\\
\hskip -1.5ex &	\;\;\;\;		=
 	\;\sigma_{\min}(\mu_i)
	\real{\left(u_1(\mu_i)^*U_jU_j^{*}A(\mu_i)V_j\frac{\partial v_{\UB}^{V_j}}{\partial \mu^{(\ell)}}(\mu_i)\right)}	\\
\hskip -1.5ex & 	\;\;\;\;		=
	\;\sigma_{\min}(\mu_i)^2 \real{\left(v_1(\mu_i)^{*}V_j\frac{\partial v_{\UB}^{V_j}}{\partial \mu^{(\ell)}}(\mu_i)\right)}   
		\:	=	\:
	\sigma_{\min}(\mu_i)^2 	\,
	\real{\left(v^{V_j}_{\UB}(\mu_i)^*\frac{\partial v_{\UB}^{V_j}}{\partial \mu^{(\ell)}}(\mu_i)\right)}
				\: = \:		0		
\end{split}
\end{equation}
for $\ell = 1, \dots , p$,
where the first equality follows from (\ref{eq:sval_defn}), in the second equality we use
$u_1(\mu_i)=U_jU_j^{*}u_1(\mu_i)$ since $u_1(\mu_i)\in\calU_{j}$, as well as (\ref{eq:sval_defn}),
and the last equality is due to \eqref{eqn:der:eig:rig}.
Consequently, it follows from (\ref{eq:sval_Hermite}),
using also the interpolation property $\sigma^{{\mathcal V}_j,\calU_j}(\mu_i) = \sigma_{\min}(\mu_i)$
proven in the previous paragraph, as well as steps similar to those in (\ref{eqn:zero:term}),
\begin{equation*}
   \begin{split}
    	&\frac{\partial \sigma^{{\mathcal V}_j,\calU_j}}{\partial \mu^{(\ell)} }(\mu_i)
		\;	=	\;
	\frac{1}{\sigma_{\min}(\mu_i)}
	\real{\left(v_1(\mu_i)^*A(\mu_i)^*U_jU_j^*\frac{\partial A}{\partial \mu^{(\ell)}}(\mu_i)v_1(\mu_i)\right)} \\
    			=	&\;
	\real{\left(u_1(\mu_i)^*U_jU_j^*\frac{\partial A}{\partial \mu^{(\ell)}}(\mu_i)v_1(\mu_i)\right)}	\;=	\;
	\real{\left(u_1(\mu_i)^*\frac{\partial A}{\partial \mu^{(\ell)}}(\mu_i)v_1(\mu_i)\right)}
			\;=\;
	\frac{\partial \sigma_{\min}}{\partial \mu^{(\ell)} }(\mu_i)
    \end{split}
\end{equation*}
for $\ell = 1, \dots, p$ as desired.
\end{proof}

\begin{remark}
When can we expect $S^{(j)}_r(\mu)$
to be a reliable error indicator, and when not? To understand this, consider the following example
\begin{equation}
    A(\mu)\;=\;\mu\begin{bmatrix}
        1&0\\
        0&0
    \end{bmatrix}+\mu^3\begin{bmatrix}
        0&0\\
        0&1
    \end{bmatrix},\quad\quad\mu\in\calD=[0.5,1.5],
\end{equation}
and let us suppose that the subspace $\calV_j \equiv\calU_j =\text{span}([0,1]^{\T})$. Therefore, 
$\sigma^{\calV_j}_{\min}(\mu)=\mu^3$ and $\|U_j^{\perp}A(\mu)V_jv^{V_j}_{\UB}(\mu)\|=S^{(j)}_r(\mu)=0\,$  for every $\mu\in\calD$, ensuring the convergence of \Cref{alg:sv}. Nonetheless, $\sigma_{\min}(\mu)=\min(\mu^3,\mu)$, which implies that \Cref{alg:sv} is unable to identify that the approximation is inadequate for $\mu\in(1,1.5]$.

This example demonstrates potential unreliability in the residual error indicator. Generally, a small residual suggests that the subspace method is accurately approximating one or more singular values (or eigenvalues), which may not necessarily be the smallest ones. The interpolation framework guarantees that the approximation is exact at the interpolated point for the smallest singular value. When the singular value function is analytic, the residual error indicator is capable of identifying singular values that are poorly approximated. However, issues arise when non-simple singular values occur or when analyticity is disrupted by a zero singular value, which can lead to the failure of the residual error indicator, as we illustrated here.
\end{remark}
\begin{comment}
The Hermitian interpolation property is crucial, as it leads to a faster convergence
of \SCM-type algorithms. Moreover, as a difference compared to the case of the smallest eigenvalue,
\Cref{alg:sv} does not depend on any linear program. The numerical experiments in \Cref{sec8}
indicate that linear programs constitute a significant source of computational cost for \Cref{alg:sf}.
\end{comment}

%-----------------------------------------------------------------------------%

\section{Numerical Results}\label{sec8}
In order to validate our results, we consider both 
randomly generated examples and examples arising from space discretizations of 
parametrized {\PDE}s. 
Recalling \Cref{rmk:adp:ell} and \eqref{eqn:spe:gap:def}, we use the condition
$\calG_\ell(\mu_{j+1})>10^{-7}$ when choosing $\ell$ dynamically at iteration $j$, and we follow the same strategy for the choice of $\sr$ on the projected problem.
All computations are performed using MATLAB 2023a on a MacBook Air Pro with an Apple M2 Pro processor and 16 GB of RAM. 

\vskip 2ex

\noindent\fbox{%
    \parbox{0.98\textwidth}{%
    \textbf{Code$\, \& \,$data availability.}
        The code and data used to generate the subsequent results are accessible via
		\begin{center}
			\url{https://zenodo.org/records/18154984}
		\end{center}
		under MIT Common License.
    }%
}\\[.2em]

\subsection{Hermitian randomly generated matrices}

For the experiments in this subsection, we always use \eqref{eqn:rel:H} as the surrogate error
in \Cref{alg:sf}, i.e., we aim to ensure that the relative  error \eqref{eqn:rel:err} is
less than the prescribed tolerance for all $\mu \in \calD$.

\subsubsection{Example 1}\label{sec:NE:exm1}

We consider the parameter-dependent matrix
\begin{equation}\label{eqn:ex1:aff}
	A(\mu) = \ee^{\mu} A_1 + \mu A_2, \quad\quad \mu \in \calD=[-1, 3],
\end{equation}
where $A_1$, $A_2 \in \R^{100 \times 100}$ are randomly generated full Hermitian matrices, and 
seek a subspace $\calV$ such that the relative error \eqref{eqn:rel:err} over the parameter 
domain ${\mathcal D} = [-1, 3]$ is below $10^{-8}$ uniformly.

\begin{figure}[t]	
	\centering
	\subfigure[Plots of $\lambda_{\min}(\mu)$ and its approximation $\lambda^{\calV}_{\min}(\mu)$ over $\calD$.]{
	    \scalebox{.9}{% This file was created by matlab2tikz.
%
\begin{tikzpicture}

\begin{axis}[%
	width=0.6*\imageWidth,
	height=\imageHeight,
	scale only axis,
	scaled ticks=false,
	grid=both,
	grid style={line width=.1pt, draw=gray!10},
	major grid style={line width=.2pt,draw=gray!50},
	axis lines*=left,
	axis line style={line width=\lineWidth},
xmin=-1,
xmax=3,
xlabel style={font=\color{white!15!black}},
xlabel={$\mu$},,
%ymode=log,
ymin=-80,
ymax=0,
yminorticks=true,
ylabel style={font=\color{white!15!black}},
	axis background/.style={fill=white},
	legend style={%
		legend cell align=left, 
		align=left, 
		font=\tiny,
		draw=white!15!black,
		at={(0.50,0.02)},
		anchor=south,},
]

\addplot [color=mycolor1, line width=\lineWidth]
  table[row sep=crcr]{%
-1	-57.3418664933627\\
-0.797979797979792	-46.3241993288746\\
-0.63636363636364	-37.5948986107554\\
-0.515151515151516	-31.1315432841586\\
-0.434343434343432	-26.890704694661\\
-0.353535353535349	-22.7458031957263\\
-0.313131313131308	-20.7332442193058\\
-0.272727272727266	-18.7859045060374\\
-0.232323232323239	-16.9389862354375\\
-0.191919191919197	-15.2449470662061\\
-0.151515151515156	-13.7468206856909\\
-0.111111111111114	-12.4218197428738\\
-0.0707070707070727	-11.2037934925863\\
-0.0303030303030312	-10.0456837639171\\
0.0101010101010104	-8.9258774686796\\
0.0505050505050519	-7.83578179978326\\
0.0909090909090935	-6.77268287159362\\
0.131313131313135	-5.73704854827089\\
0.171717171717177	-4.73204961064182\\
0.212121212121218	-3.7649463620634\\
0.25252525252526	-2.85254470481873\\
0.292929292929287	-2.04019844739332\\
0.333333333333329	-1.48549662337915\\
0.37373737373737	-1.51522049348372\\
0.414141414141412	-2.0065409905812\\
0.454545454545453	-2.7259329220644\\
0.494949494949495	-3.52975979076193\\
0.575757575757578	-5.16416476437013\\
0.616161616161619	-5.96033609309809\\
0.656565656565661	-6.73521154032944\\
0.696969696969703	-7.48579107380131\\
0.737373737373744	-8.20978647615514\\
0.777777777777771	-8.90528205181178\\
0.818181818181813	-9.57056378670467\\
0.858585858585855	-10.2040258463099\\
0.898989898989896	-10.8041162401627\\
0.939393939393938	-11.3693037642047\\
0.979797979797979	-11.898057192091\\
1.02020202020202	-12.3888319104918\\
1.06060606060606	-12.8400613467764\\
1.1010101010101	-13.2501517011504\\
1.14141414141415	-13.6174791622796\\
1.18181818181819	-13.9403891955329\\
1.22222222222223	-14.2171977705574\\
1.26262626262626	-14.4461946095646\\
1.3030303030303	-14.6256487334756\\
1.34343434343434	-14.7538167939541\\
1.38383838383838	-14.8289549399259\\
1.42424242424242	-14.8493353235153\\
1.46464646464646	-14.8132688770921\\
1.50505050505051	-14.7191368286061\\
1.54545454545455	-14.5654348506095\\
1.58585858585859	-14.3508364021098\\
1.62626262626263	-14.0742873332172\\
1.66666666666667	-13.7351567080801\\
1.70707070707071	-13.3335034896973\\
1.74747474747475	-12.8706270340247\\
1.78787878787878	-12.3504366987181\\
1.82828282828282	-11.7829197284655\\
1.90909090909091	-10.5772775792596\\
1.94949494949495	-9.97565301766781\\
1.98989898989899	-9.41833204617863\\
2.03030303030303	-8.95059945936939\\
2.07070707070707	-8.62297381252075\\
2.11111111111111	-8.50579581453489\\
2.15151515151516	-8.66314662000858\\
2.1919191919192	-9.11223227071832\\
2.23232323232324	-9.85950718767693\\
2.27272727272727	-10.9842349052136\\
2.31313131313131	-12.6229434724394\\
2.35353535353535	-14.5341978138861\\
2.39393939393939	-16.6941989763634\\
2.43434343434343	-19.0869937969295\\
2.47474747474747	-21.702737979771\\
2.51515151515152	-24.5362473206558\\
2.55555555555556	-27.5858321285114\\
2.5959595959596	-30.8524430764001\\
2.63636363636364	-34.3390530575946\\
2.67676767676768	-38.0502082696597\\
2.71717171717172	-41.9917047551942\\
2.75757575757576	-46.1703589243645\\
2.79797979797979	-50.5938473923831\\
2.83838383838383	-55.2705969509001\\
2.87878787878788	-60.2097103491896\\
2.91919191919192	-65.4209175636119\\
2.95959595959596	-70.9145452860412\\
3	-76.7014995702321\\
};
\addlegendentry{$\lambda_{{\min}}(\mu)$}

\addplot [color=mycolor2, dashed, line width=\lineWidth, %mark options={solid, mycolor2}, mark indices={0, 2, 4, 6, 8, 10, 12, 14, 16, 18, 20, 22
 %24, , 26, 30, 34, 40, 46, 54, 60, 66, 72, 78, 84, 86, 88, 90, 92, 94, 96, 98, 100}
 ]
  table[row sep=crcr]{%
-1	-57.3418664933627\\
-0.959595959595958	-55.1312988224902\\
-0.919191919191917	-52.9239269355301\\
-0.878787878787875	-50.7200460436168\\
-0.838383838383834	-48.5200010625541\\
-0.797979797979792	-46.324199323644\\
-0.757575757575751	-44.1331275473439\\
-0.717171717171723	-41.9473748591508\\
-0.676767676767682	-39.7676645327682\\
-0.63636363636364	-37.5948986062839\\
-0.595959595959599	-35.4302219383417\\
-0.555555555555557	-33.2751164089837\\
-0.515151515151516	-31.1315432839677\\
-0.474747474747474	-29.0021651804785\\
-0.434343434343432	-26.890704693933\\
-0.393939393939391	-24.8025476497469\\
-0.353535353535349	-22.7458031946772\\
-0.313131313131308	-20.7332442191508\\
-0.272727272727266	-18.7859045059252\\
-0.232323232323239	-16.9389862353512\\
-0.191919191919197	-15.244947066064\\
-0.151515151515156	-13.7468206856346\\
-0.111111111111114	-12.421819742698\\
-0.0707070707070727	-11.2037934919626\\
-0.0303030303030312	-10.0456837623469\\
0.0101010101010104	-8.92587746847687\\
0.0505050505050519	-7.83578179938485\\
0.0909090909090935	-6.77268287073549\\
0.131313131313135	-5.7370485482206\\
0.171717171717177	-4.73204961040788\\
0.212121212121218	-3.76494636205878\\
0.25252525252526	-2.85254470481679\\
0.292929292929287	-2.04019844737937\\
0.333333333333329	-1.48549662337756\\
0.37373737373737	-1.51522049345927\\
0.414141414141412	-2.00654099051543\\
0.454545454545453	-2.72593292205623\\
0.494949494949495	-3.52975979076112\\
0.535353535353536	-4.35118253909012\\
0.575757575757578	-5.16416476436603\\
0.616161616161619	-5.96033609308591\\
0.656565656565661	-6.73521154032926\\
0.696969696969703	-7.48579107378534\\
0.737373737373744	-8.2097864761033\\
0.777777777777771	-8.90528205174344\\
0.818181818181813	-9.57056378664942\\
0.858585858585855	-10.2040258462803\\
0.898989898989896	-10.8041162401528\\
0.939393939393938	-11.3693037642033\\
0.979797979797979	-11.8980571920909\\
1.02020202020202	-12.3888319104918\\
1.06060606060606	-12.840061346775\\
1.1010101010101	-13.2501517011409\\
1.14141414141415	-13.6174791622501\\
1.18181818181819	-13.9403891954722\\
1.22222222222223	-14.2171977704645\\
1.26262626262626	-14.4461946094535\\
1.3030303030303	-14.6256487333725\\
1.34343434343434	-14.7538167938842\\
1.38383838383838	-14.8289549398976\\
1.42424242424242	-14.8493353235131\\
1.46464646464646	-14.8132688770869\\
1.50505050505051	-14.719136828581\\
1.54545454545455	-14.5654348505768\\
1.58585858585859	-14.3508364020938\\
1.62626262626263	-14.0742873332168\\
1.66666666666667	-13.7351567080747\\
1.70707070707071	-13.333503489694\\
1.74747474747475	-12.8706270340228\\
1.78787878787878	-12.350436698718\\
1.82828282828282	-11.7829197284582\\
1.86868686868686	-11.1863953127714\\
1.90909090909091	-10.5772775788089\\
1.94949494949495	-9.97565301750262\\
1.98989898989899	-9.4183320458007\\
2.03030303030303	-8.95059945935591\\
2.07070707070707	-8.62297381251663\\
2.11111111111111	-8.50579581448048\\
2.15151515151516	-8.66314661990758\\
2.1919191919192	-9.11223227071611\\
2.23232323232324	-9.85950718767079\\
2.27272727272727	-10.984234905213\\
2.31313131313131	-12.6229434724307\\
2.35353535353535	-14.5341978138858\\
2.39393939393939	-16.694198976249\\
2.43434343434343	-19.0869937969231\\
2.47474747474747	-21.7027379796595\\
2.51515151515152	-24.5362473205115\\
2.55555555555556	-27.5858321285078\\
2.5959595959596	-30.8524430762705\\
2.63636363636364	-34.3390530571803\\
2.67676767676768	-38.050208269192\\
2.71717171717172	-41.9917047549602\\
2.75757575757576	-46.1703589243532\\
2.79797979797979	-50.5938473922488\\
2.83838383838383	-55.2705969501655\\
2.87878787878788	-60.2097103474883\\
2.91919191919192	-65.4209175608369\\
2.95959595959596	-70.9145452823688\\
3	-76.7014995660513\\
};
\addlegendentry{$\lambda^{\mathcal{V}}_{{\min}}(\mu)$}

\end{axis}
\end{tikzpicture}%
}
     \label{fig1:a}
	}
    \subfigure[Approximation error over $\calD$.]{
	    \scalebox{.9}{% This file was created by matlab2tikz.
%
\begin{tikzpicture}

\begin{axis}[%
	width=0.6*\imageWidth,
	height=\imageHeight,
	scale only axis,
	scaled ticks=false,
	grid=both,
	grid style={line width=.1pt, draw=gray!10},
	major grid style={line width=.2pt,draw=gray!50},
	axis lines*=left,
	axis line style={line width=\lineWidth},
xmin=-1,
xmax=3,
xlabel style={font=\color{white!15!black}},
xlabel={$\mu$},
ymode=log,
ymin=1e-15,
ymax=1e-8,
yminorticks=true,
ylabel style={font=\color{white!15!black}},
	axis background/.style={fill=white},
	legend style={%
		legend cell align=left, 
		align=left, 
		font=\tiny,
		draw=white!15!black,
		at={(0.98,0.98)},
		anchor=north east,},
]

\addplot [color=mycolor1, line width=\lineWidth]
  table[row sep=crcr]{%
-1	2.47826859923489e-16\\
-0.95959595959596	6.45169977292552e-12\\
-0.919191919191919	2.4445448075322e-11\\
-0.878787878787879	5.10902457994062e-11\\
-0.838383838383839	8.23091691265148e-11\\
-0.797979797979798	1.12911908770281e-10\\
-0.757575757575758	1.36960970924921e-10\\
-0.717171717171718	1.48528309176687e-10\\
-0.676767676767676	1.4300597215311e-10\\
-0.636363636363637	1.18937936070018e-10\\
-0.595959595959597	8.01426219001949e-11\\
-0.555555555555555	3.71588527982979e-11\\
-0.515151515151516	6.13186470224288e-12\\
-0.474747474747474	2.34363598992817e-12\\
-0.434343434343434	2.70719200368734e-11\\
-0.393939393939394	5.48522754747635e-11\\
-0.353535353535353	4.61225758478916e-11\\
-0.313131313131313	7.47992238299847e-12\\
-0.272727272727273	5.97284821085756e-12\\
-0.232323232323232	5.09616312145878e-12\\
-0.191919191919192	9.32074973311083e-12\\
-0.151515151515152	4.09496490336996e-12\\
-0.111111111111111	1.41472816191178e-11\\
-0.0707070707070709	5.56754804867729e-11\\
-0.0303030303030312	1.56300981516114e-10\\
0.0101010101010104	2.27136390243351e-11\\
0.0505050505050502	5.08457766147981e-11\\
0.0909090909090917	1.26701771223614e-10\\
0.131313131313131	8.7654433324603e-12\\
0.171717171717171	4.94380977534994e-11\\
0.212121212121213	1.22530263868084e-12\\
0.252525252525253	6.77371383100453e-13\\
0.292929292929292	6.83852628345085e-12\\
0.333333333333334	1.07397786119428e-12\\
0.373737373737374	1.61437374709179e-11\\
0.414141414141413	3.27762790565475e-11\\
0.454545454545455	3.00134146620757e-12\\
0.494949494949495	2.27972920524089e-13\\
0.535353535353535	4.97244740010825e-13\\
0.575757575757576	7.95104149775853e-13\\
0.616161616161616	2.04344024886052e-12\\
0.656565656565657	2.83522438890156e-14\\
0.696969696969697	2.13306402382594e-12\\
0.737373737373737	6.31369557838998e-12\\
0.777777777777779	7.67329934088622e-12\\
0.818181818181818	5.77291244356344e-12\\
0.858585858585858	2.90528657407877e-12\\
0.8989898989899	9.1168018332382e-13\\
0.939393939393939	1.28586737512405e-13\\
0.979797979797979	2.09017281982261e-15\\
1.02020202020202	1.14706978170937e-15\\
1.06060606060606	1.10260875144307e-13\\
1.1010101010101	7.16031189168937e-13\\
1.14141414141414	2.16933059983072e-12\\
1.18181818181818	4.34953167417893e-12\\
1.22222222222222	6.53345877216728e-12\\
1.26262626262626	7.68854823709047e-12\\
1.3030303030303	7.05021457121575e-12\\
1.34343434343434	4.73833479100646e-12\\
1.38383838383838	1.90238105657837e-12\\
1.42424242424242	1.50369056828455e-13\\
1.46464646464647	3.54473469754488e-13\\
1.50505050505051	1.70767141920352e-12\\
1.54545454545454	2.24535045508119e-12\\
1.58585858585859	1.10870389465151e-12\\
1.62626262626263	2.32231764714825e-14\\
1.66666666666667	3.91091503105197e-13\\
1.70707070707071	2.5352729445727e-13\\
1.74747474747475	1.53336153969558e-13\\
1.78787878787879	7.91062118290565e-15\\
1.82828282828283	6.2126932066661e-13\\
1.86868686868687	2.87691038619673e-12\\
1.90909090909091	4.26077656841136e-11\\
1.94949494949495	1.65595479303651e-11\\
1.98989898989899	4.01271237715542e-11\\
2.03030303030303	1.50573371110269e-12\\
2.07070707070707	4.77926520132365e-13\\
2.11111111111111	6.39762775936221e-12\\
2.15151515151515	1.16575673737702e-11\\
2.19191919191919	2.41923029637881e-13\\
2.23232323232323	6.22476637351865e-13\\
2.27272727272727	5.62781281942459e-14\\
2.31313131313131	6.85609544276242e-13\\
2.35353535353535	1.63773618281306e-14\\
2.39393939393939	6.85401255944758e-12\\
2.43434343434343	3.35783389712698e-13\\
2.47474747474748	5.13506984670517e-12\\
2.51515151515152	5.88053905523303e-12\\
2.55555555555556	1.30075496685138e-13\\
2.5959595959596	4.20073686489754e-12\\
2.63636363636364	1.20644544680681e-11\\
2.67676767676768	1.22929756295078e-11\\
2.71717171717172	5.5727707641493e-12\\
2.75757575757576	2.44386634774867e-13\\
2.7979797979798	2.65460712786098e-12\\
2.83838383838384	1.32915180048579e-11\\
2.87878787878788	2.82565123010528e-11\\
2.91919191919192	4.24182198726454e-11\\
2.95959595959596	5.17858368953759e-11\\
3	5.45071039998809e-11\\
};
\addlegendentry{${\calE}^{\calV}_r(\mu)$ (see \eqref{eqn:rel:err})}
\end{axis}
\end{tikzpicture}%
}
     \label{fig1:b}
	}
    \subfigure[Decay of $\max_{\mu\in\calD} H_r^{(j)}(\mu)$ with respect to the iteration
     counter $j$ of \Cref{alg:sf}.]{
	    \scalebox{.9}{% This file was created by matlab2tikz.
%
\begin{tikzpicture}

\begin{axis}[%
	width=0.6*\imageWidth,
	height=\imageHeight,
	scale only axis,
	scaled ticks=false,
	grid=both,
	grid style={line width=.1pt, draw=gray!10},
	major grid style={line width=.2pt,draw=gray!50},
	axis lines*=left,
	axis line style={line width=\lineWidth},
xmin=0,
xmax=35,
xlabel style={font=\color{white!15!black}},
xlabel={$j$},
ymode=log,
ymin=1e-10,
ymax=10000,
yminorticks=true,
ylabel style={font=\color{white!15!black}},
	axis background/.style={fill=white},
	legend style={%
		legend cell align=left, 
		align=left, 
		font=\tiny,
		draw=white!15!black,
		at={(1.10,0.98)},
		anchor=north east,},
]

\addplot [color=mycolor1, line width=\lineWidth, mark=o, mark options={solid, mycolor1}]
  table[row sep=crcr]{%
1	192.030823024698\\
2	3.61671749100696\\
3	1.6744778124624\\
4	0.901352838934562\\
5	0.306934160644513\\
6	0.150528770669131\\
7	0.143347918277699\\
8	0.13134634816611\\
9	0.100594912142646\\
10	0.0736080655940456\\
11	0.024966359932911\\
12	0.017039400378344\\
13	0.0157820870981209\\
14	0.00446788450254395\\
15	0.00433404418104943\\
16	0.00287370403756224\\
17	0.000637147456778437\\
18	0.000212844708893374\\
19	0.000129680945030881\\
20	4.51497231447082e-05\\
21	2.79023670608338e-05\\
22	1.82688936704672e-05\\
23	7.90247106296403e-06\\
24	5.33789183735644e-06\\
25	1.42067060842657e-06\\
26	4.77391870984333e-07\\
27	8.10531941890701e-08\\
28	6.24958386574212e-08\\
29	5.21217368270505e-08\\
30	4.10477567929009e-08\\
31	1.15498702076847e-08\\
32	4.25372562106039e-09\\
};
\addlegendentry{$H^{(j)}_r(\mu_{j+1})$ (see \eqref{eqn:rel:H})}
\end{axis}
\end{tikzpicture}%
}
     \label{fig1:c}
	}
  \caption{\emph{(Concerns Example 1)} $\: A(\mu)\in\R^{\sn\times \sn}$ full matrix as in \eqref{eqn:ex1:aff}, $\sn=10^2$ and 
  				$\: A^{V}(\mu)\in\R^{\sd \times \sd}$ with $\sd=32$.}
	\label{fig1}
\end{figure}

When we run \Cref{alg:sf}, the termination condition is satisfied for a subspace $\calV$
of dimension $\sd=32$. \Cref{fig1:a} shows that $\lambda_{\min}(\mu)$ is qualitatively 
well approximated by $\lambda_{\min}^{\calV}(\mu)$ over the whole parameter domain, while 
\Cref{fig1:b} illustrates that the computed error \eqref{eqn:rel:err} over the parameter 
domain is uniformly smaller than the prescribed tolerance $10^{-8}$. 
Finally, it can be observed in \Cref{fig1:c} that the maximum value attained by $H_r^{(j)}(\mu)$ over $\mu \in \calD$ is monotonically decreasing as a function of the iteration counter $j$ of \Cref{alg:sf}. 

We provide \Cref{figINT} to validate \Cref{rmk:int}. Here, we consider four different values
$\mu_1, \mu_2, \mu_3, \mu_4$ of the parameter (the red crosses in \Cref{figINT:a}). By computing the eigenpairs
$(\lambda_{\min}(\mu_i), v_1(\mu_i))$ for $i = 1,\dots, 4$, we construct
the subspace $\calV$ spanned by $v_1(\mu_1), v_1(\mu_2), v_1(\mu_3), v_1(\mu_4)$
leading to the upper bound $\lambda_{\min}^{\calV}(\mu)$, as well as
the lower bounds $\lambda_{\SCM}(\mu)$ and $\lambda_{\LB}(\mu)$. 

The first observation is that, as
expected, $\lambda_{\min}^{\calV}(\mu)$ and $\lambda_{\SCM}(\mu)$, $\lambda_{\LB}(\mu)$
are indeed effective upper and lower bounds for $\lambda_{\min}(\mu)$, all of which also
interpolate $\lambda_{\min}(\mu)$ at the points $\mu = \mu_i$ for $i = 1,\dots, 4$. 
The second observation is that, in contrast to the claim in \cite[eq. (3.9)]{SirK16}, $\lambda_{\LB}(\mu)$ is not greater than or equal to 
$\lambda_{\SCM}(\mu)$ for all $\mu\in \calD$. 

It is apparent from \Cref{figINT:a} that the claimed property $\lambda_{\LB}(\mu) \geq \lambda_{\SCM}(\mu)$ does not hold especially for $\mu$ values not close to the interpolation points.
On the other hand, the property $\lambda_{\LB}(\mu) \geq \lambda_{\SCM}(\mu)$ holds for $\mu$
near the interpolation points, as can be observed in \Cref{figINT:b}, 
where we zoom in on the graphs of the functions depicted in \Cref{figINT:a} near the leftmost
interpolation point. 

Here, another property that we can qualitatively 
observe is the Hermitian interpolation property; $\lambda_{\min}(\mu)$ is interpolated 
tangentially by $\lambda_{\min}^{\calV}(\mu)$ and $\lambda_{\LB}(\mu)$ at $\mu = \mu_i$, 
meaning that the derivatives of $\lambda_{\min}^{\calV}(\mu)$ and $\lambda_{\LB}(\mu)$ interpolate 
$\lambda'_{\min}(\mu)$ at $\mu = \mu_i$. This tangential interpolation property does not 
seem to hold for $\lambda_{\SCM}(\mu)$; indeed, 
it is evident from \Cref{figINT:b} that the left-hand and right-hand 
derivatives of $\lambda_{\SCM}(\mu)$ at $\mu = \mu_i$ are different, indicating that
$\lambda_{\SCM}(\mu)$ is not differentiable at $\mu = \mu_i$.

\begin{figure}[t]
	\centering
	\subfigure[Plots of the lower and upper bounds introduced in \Cref{sec2}
   over $\mu \in \calD$.]{
		\input{Example_1_Plot_6.tex}
      \label{figINT:a}
   }
   \hskip 2ex
   \subfigure[Zoomed versions of the graphs of the functions in  \Cref{figINT:a} 
        near the leftmost interpolation point.]{
		\input{Example_1_Plot_7.tex}
        \label{figINT:b}
   }
	\caption{\emph{(Concerns Example 1)} $\: A(\mu)\in\R^{\sn\times \sn}$ in \eqref{eqn:ex1:aff}, $\sn=10^2$.}
	\label{figINT}
\end{figure}

%
%%%%

\subsubsection{Example 1 (continued)}\label{sec:NE:exm1bis}

We now consider the same matrices $A_1, A_2$ as in \eqref{eqn:ex1:aff}, but with a parameter-dependent matrix $A(\mu)$ that depends on two parameters, defined via the affine decomposition
\begin{equation}\label{eqn:ex1:bis:aff}
	A(\mu) = \ee^{\mu_1} A_1 + \mu_2 A_2, \quad\quad 
	\mu\, =\, ( \mu_1,\mu_2 ) \in \calD = [-10, 10] \times [-10, 10].
\end{equation}
The goal of this experiment is to illustrate both the advantages and limitations of solving the optimization problem \eqref{def:H} over a continuum domain through \EIGOPT, in contrast to discrete-domain optimization, as in \cite{SirK16}. Recall that the optimization routine \EIGOPT employs two stopping criteria: (i) a guarantee that the computed global maximum differs from the true one by no more than a prescribed tolerance, and (ii) exceeding a
prescribed upper bound on the number of objective function evaluations, which we indicate as $ \max\# H_r^{j}(\mu) $.

We examine the efficiency of employing \EIGOPT within \Cref{alg:sf} against the \SSCM approach from \cite{SirK16}, under the condition that $\max \# H_r^{j}(\mu) $ aligns with the size of the discrete domain, denoted as $|\Xi|$, and, the discrete domain $\Xi$ is generated by the Cartesian products of Chebyshev distributed points in each parameter dimension, utilized for the \SSCM. 
We report with respect to the parameter $\max \# H_r^{j}(\mu)$
i) the computational time (\CT) required to execute the algorithms, ii) the greatest relative error, 
and iii) the average relative error measured over a randomly generated set $\Xi_{test}$ 
consisting of $10^4$ points.
\begin{table}[t]
\centering
\scalebox{0.9}{\begin{tabular}{lcrrrrrrr}
\toprule
\multicolumn{2}{l}{$\max \# H_r^{j}(\mu)~\big|~|\Xi|$}  & 100 & 225 & 400 & 625 & 900&1225&1600 \\
\midrule

% Row Group A
\multirow{2}{*}{\CT [s]} 
& \SSCM    & 5  & 3  & 4  & 8 & 4 & 5& 6  \\
& \EIGOPT & 14 & 47 & 95 & 180 &305& 500 & 803 \\
\midrule
% Row Group C
\multirow{2}{*}{$\max~{\calE}^{\calV}_r(\mu)$} 
& \SSCM    & 2.0$\cdot10^{-1}$ & 3.6$\cdot10^{-2}$  & 8.1$\cdot10^{-3}$  & 1.8$\cdot10^{-3}$&2.0$\cdot10^{-3}$& 1.4$\cdot10^{-4}$&2.0$\cdot10^{-3}$  \\
& \EIGOPT  & 9.1$\cdot10^{-3}$ & 3.6$\cdot10^{-5}$ & 2.7$\cdot10^{-6}$ & 1.2$\cdot10^{-6}$&3.9$\cdot10^{-7}$&1.8$\cdot10^{-8}$&1.5$\cdot10^{-9}$ \\
\midrule
% Row Group C
\multirow{2}{*}{$\text{mean}~{\calE}^{\calV}_r(\mu)$} 
& \SSCM   & 3.6$\cdot10^{-4}$ & 1.2$\cdot10^{-4}$  & 9.5$\cdot10^{-6}$  & 2.4$\cdot10^{-6}$&1.6$\cdot10^{-6}$& 7.8$\cdot10^{-8}$&1.5$\cdot10^{-6}$  \\
& \EIGOPT  & 9.8$\cdot10^{-6}$ & 9.6$\cdot10^{-9}$ & 2.8$\cdot10^{-9}$ & 8.8$\cdot10^{-10}$&1.3$\cdot10^{-10}$&5.6$\cdot10^{-10}$&5.3$\cdot10^{-11}$ \\
\bottomrule
\end{tabular}}
\caption{Example 1: $\: A(\mu)\in\R^{\sn\times \sn}$ as in \eqref{eqn:ex1:bis:aff}, $\sn=10^2$. The construction of subspace $\calV$ is achieved by employing either \Cref{alg:sf} with \EIGOPT or the conventional \SSCM, targeting a specified tolerance of $10^{-8}$. The outcomes presented include: the computational time (\CT) in seconds required to form $\calV$, the highest relative error, and the average relative error calculated over $\Xi_{test}$ such that $|\Xi_{test}|=10^4$, where $\Xi_{test}$ consists of
points in $\mathcal{D}$ generated randomly.}
    \label{Tab:1}
\end{table}

The findings reported in \Cref{Tab:1} demonstrate that the continuum optimization via \EIGOPT yields subspaces with errors several orders of magnitude smaller throughout the test domain $\Xi_{test}$. It is important to highlight that this improved performance is not a result of constructing larger subspaces compared to those formed by \SSCM; in fact, in most cases, they share the same dimension or vary by only a few units. However, this superior quality of approximation spaces incurs a significantly higher cost during the offline construction phase of $\calV$. In summary, even when the number of permissible objective function evaluations is limited, continuum optimization offers the advantage, over uniformly distributed grids, of exploring the parametric domain at points not covered, and possibly ``far'' due to the coarseness, of the discrete grid, thereby enhancing the quality of the approximation space. In contrast, the discrete grid demonstrates significantly higher computational efficiency, notably because it allows for straightforward parallel execution of objective function evaluations and eliminates the need for 
gradient evaluations.

An effective strategy to leverage the advantages of continuum optimization while minimizing its limitations involves employing the following hybrid method. Initially, execute \Cref{alg:sf} using \EIGOPT with a limited number of objective function evaluations (approximately $100$). Subsequently, employ the generated subspace $\calV$ as the starting point for the \SSCM, where the discrete grid contains a significantly larger number of points. Our hypothesis is that due to the effectiveness of continuum optimization to 
construct accurate subspaces with minimal objective function evaluations, the initial subspace 
constructed for the discrete approach
would necessitate only minor adjustments to attain the desired accuracy on substantially larger discrete grids. This hypothesis is substantiated by the
data in \Cref{Tab:2}, which demonstrates that the hybrid continuum-discrete methodology, referred to as \EIGOPT-\SSCM with a function evaluation cap of $100$ for \EIGOPT, is considerably faster compared to the conventional \SSCM.

\begin{table}[t]
\centering

\begin{tabular}{lcrrrrrr}
\toprule
$|\Xi|$ & & $50^2$ & $100^2$ & $150^2$ & $200^2$ & $250^2$ & $300^2$ \\
\midrule

% Row Group A
\multirow{2}{*}{\CT [s]} 
& \SSCM  & 9  & 30  & 67  & 123  & 194 & 271  \\
& \EIGOPT-\SSCM & 15& 18 & 22 & 28 & 33 &41 \\
\bottomrule
\end{tabular}

\caption{Example 1: $\: A(\mu)\in\R^{\sn\times \sn}$ as in \eqref{eqn:ex1:bis:aff}, $\sn=10^2$. Comparison between the standard \SSCM and the hybrid continuum-discrete optimization (\EIGOPT-\SSCM) in terms of computation time (\CT) to construct a subspace $\calV$ satisfying the prescribed error tolerance $10^{-8}$. The maximum number of objective function evaluations set for \EIGOPT is $100$.}
 \label{Tab:2}
\end{table}

In summary, from this comparison, we can deduce the following guidelines for the application of continuum optimization in handling parametric eigenvalue problems through \EIGOPT:
\begin{itemize}
    \item \Cref{alg:sf} is preferred over the conventional \SSCM when the offline phase is constrained by a limited computational budget, yet several parameter evaluations are required during the online phase. This preference arises because employing a coarse grid with \SSCM for constructing $\calV$ might not yield precise approximations across the entire domain $\calD$.
    \item If there is some computational budget available for the offline phase, continuum optimization may still be advantageous for initializing the projection space $\calV$. This initialization can subsequently be refined using \SSCM with a finer grid $\Xi$.
\end{itemize}

\subsubsection{Example 2}\label{sec:NE:exm2}

We now turn to a larger-scale problem involving a parameter-dependent matrix
\begin{equation}\label{eqn:ex2:aff}
	A(\mu) = \mu^2 A_1 + \mu A_2, \quad\quad \mu \in \calD := [-2, 4],
\end{equation}
where $A_1, A_2 \in \R^{\sn \times \sn}$ are dense Hermitian matrices with $\sn = 2000$. The objective is to construct a subspace $\calV$ such that the smallest eigenvalue $\lambda_{\min}^{\calV}(\mu)$ of the projected problem approximates $\lambda_{\min}(\mu)$ with a relative error not exceeding $\varepsilon = 10^{-8}$ uniformly over $\calD$.

\begin{figure}[t]
	\centering
	\subfigure[Plots of $\lambda_{\min}(\mu)$ and $\lambda^{\calV}_{\min}(\mu)$ over $\calD$. The interpolation points are also shown.]{
	    % This file was created by matlab2tikz.
%
\begin{tikzpicture}

\begin{axis}[%
	width=\imageWidth,
	height=\imageHeight,
	scale only axis,
	scaled ticks=false,
	grid=both,
	grid style={line width=.1pt, draw=gray!10},
	major grid style={line width=.2pt,draw=gray!50},
	axis lines*=left,
	axis line style={line width=\lineWidth},
xmin=-2,
xmax=4,
xlabel style={font=\color{white!15!black}},
xlabel={$\mu$},,
%ymode=log,
ymin=-500,
ymax=0,
yminorticks=true,
ylabel style={font=\color{white!15!black}},
	axis background/.style={fill=white},
	legend style={%
		legend cell align=left, 
		align=left, 
		font=\tiny,
      %text height=0.4cm,
		draw=white!15!black,
		at={(0.50,0.02)},
		anchor=south,},
]

\addplot [color=mycolor1, line width=\lineWidth]
  table[row sep=crcr]{%
-2	-160.722655844132\\
-1.88345365382099	-147.912566874219\\
-1.84810126582278	-144.133613554308\\
-1.77215189873419	-136.180022899651\\
-1.69590803372074	-128.418447820348\\
-1.62025316455697	-120.933327608454\\
-1.53192636535221	-112.461923682305\\
-1.46835443037975	-106.539867436523\\
-1.39110132409616	-99.5375284723912\\
-1.31645569620252	-92.9719145651593\\
-1.24050632911394	-86.4931099984605\\
-1.13349422466416	-77.7040831650081\\
-1.08860759493672	-74.1287082539742\\
-1.01265822784808	-68.2147049233946\\
-0.936708860759495	-62.4582202211054\\
-0.860759493670912	-56.8470381539128\\
-0.784810126582272	-51.3703785878583\\
-0.708860759493689	-46.0178214652223\\
-0.632911392405049	-40.7788307257079\\
-0.556962025316466	-35.6425541474727\\
-0.405063291139243	-25.632831101699\\
-0.25316455696202	-15.8947577033668\\
-0.0253164556962133	-1.58162900389635\\
-1.25550229768123e-07	-7.84393085950796e-06\\
0.202531645569593	-12.7952228368109\\
0.354430379746816	-22.5569543794043\\
0.430379746835456	-27.5289298340385\\
0.506329113924039	-32.5776266313839\\
0.582278481012679	-37.7146080181222\\
0.673024751063565	-43.9842821454837\\
0.734177215189845	-48.2996334841802\\
0.810126582278485	-53.7707891458123\\
0.914801546196657	-61.533045701548\\
0.962025316455708	-65.1254483897256\\
1.03797468354429	-71.0285311302617\\
1.1649083925966	-81.2597985662434\\
1.18987341772151	-83.3283840240001\\
1.26582278481015	-89.7401072135095\\
1.34177215189874	-96.3357121670867\\
1.46907582034777	-107.821817689331\\
1.49367088607596	-110.105235345942\\
1.56962025316454	-117.292015265333\\
1.64556962025318	-124.688364937822\\
1.72151898734177	-132.300124640828\\
1.8095076348863	-141.395052292175\\
1.87341772151899	-148.191547921128\\
1.94936708860757	-156.481076627823\\
2.02531645569621	-165.005709483358\\
2.1012658227848	-173.769370021187\\
2.17881511358718	-182.967984750207\\
2.25316455696202	-192.027754679103\\
2.32911392405066	-201.528770926767\\
2.40506329113924	-211.281506374626\\
2.48101265822783	-221.288640042203\\
2.59416060641547	-236.674394151674\\
2.63291139240505	-242.076255795303\\
2.70886075949369	-252.861580591099\\
2.80537447038813	-266.948482552522\\
2.86075949367091	-275.226630609675\\
2.9367088607595	-286.810520096969\\
3.01594949740786	-299.184305459913\\
3.08860759493672	-310.790144804512\\
3.1645569620253	-323.188933129291\\
3.24050632911394	-335.862084509574\\
3.31645569620252	-348.810675967933\\
3.39240506329116	-362.03565709319\\
3.50822997100346	-382.738183050444\\
3.62025316455697	-403.377125894792\\
3.69620253164555	-417.715554178838\\
3.78948605216817	-435.709798531919\\
3.92405063291142	-462.413653071627\\
4	-477.875848056489\\
4	-477.875848056489\\
};
\addlegendentry{$\lambda_{{\min}}(\mu)$}

\addplot [color=mycolor2, dashed, line width=\lineWidth]
  table[row sep=crcr]{%
-2	-160.722655844132\\
-1.88345365382099	-147.912566874219\\
-1.84810126582278	-144.133613554205\\
-1.77215189873419	-136.180022899457\\
-1.69590803372074	-128.418447820348\\
-1.62025316455697	-120.933327608374\\
-1.53192636535221	-112.461923682305\\
-1.46835443037975	-106.539867436465\\
-1.39110132409616	-99.5375284723912\\
-1.31645569620252	-92.9719145650215\\
-1.24050632911394	-86.4931099984528\\
-1.13349422466416	-77.7040831650081\\
-1.08860759493672	-74.1287082539098\\
-1.01265822784808	-68.2147049233939\\
-0.936708860759495	-62.4582202209966\\
-0.860759493670912	-56.847038153895\\
-0.784810126582272	-51.3703785872792\\
-0.708860759493689	-46.0178214647664\\
-0.632911392405049	-40.7788307256979\\
-0.556962025316466	-35.6425541469511\\
-0.405063291139243	-25.6328311016989\\
-0.25316455696202	-15.8947577033641\\
-0.0253164556962133	-1.58162900383184\\
-1.25550229768123e-07	-7.84393085950796e-06\\
0.202531645569593	-12.7952228367573\\
0.354430379746816	-22.5569543793513\\
0.430379746835456	-27.528929834018\\
0.506329113924039	-32.5776266304124\\
0.582278481012679	-37.7146080167081\\
0.673024751063565	-43.9842821454836\\
0.734177215189845	-48.2996334831711\\
0.810126582278485	-53.7707891439464\\
0.914801546196657	-61.533045701548\\
0.962025316455708	-65.1254483891761\\
1.03797468354429	-71.0285311287051\\
1.1649083925966	-81.2597985662433\\
1.18987341772151	-83.3283840238697\\
1.26582278481015	-89.7401072121004\\
1.34177215189874	-96.3357121654544\\
1.46907582034777	-107.821817689331\\
1.49367088607596	-110.105235345844\\
1.56962025316454	-117.292015264181\\
1.64556962025318	-124.688364936175\\
1.72151898734177	-132.300124640035\\
1.8095076348863	-141.395052292175\\
1.87341772151899	-148.191547920734\\
1.94936708860757	-156.481076626698\\
2.02531645569621	-165.005709482299\\
2.1012658227848	-173.76937002082\\
2.17881511358718	-182.967984750207\\
2.25316455696202	-192.027754678853\\
2.32911392405066	-201.528770926197\\
2.40506329113924	-211.281506374124\\
2.48101265822783	-221.288640042004\\
2.59416060641547	-236.674394151674\\
2.63291139240505	-242.076255795294\\
2.70886075949369	-252.861580591079\\
2.80537447038813	-266.948482552522\\
2.86075949367091	-275.226630609665\\
2.9367088607595	-286.810520096948\\
3.01594949740786	-299.184305459913\\
3.08860759493672	-310.790144804428\\
3.1645569620253	-323.188933128876\\
3.24050632911394	-335.86208450877\\
3.31645569620252	-348.810675967084\\
3.39240506329116	-362.035657092728\\
3.50822997100346	-382.738183050444\\
3.62025316455697	-403.377125894546\\
3.69620253164555	-417.715554178618\\
3.78948605216817	-435.709798531919\\
3.92405063291142	-462.413653071199\\
4	-477.875848056489\\
4	-477.875848056489\\
};
\addlegendentry{$\lambda^{\mathcal{V}}_{{\min}}(\mu)$}
\addplot [color=mycolor3, line width=\lineWidth,only marks, mark=asterisk, mark options={solid, mycolor3}]
  table[row sep=crcr]{
2.17881511358718	-182.967984750207\\
-2	-160.722655844132\\
1.64255732215679e-06	-0.000103412998157637\\
0.914801546196657	-61.533045701548\\
-0.875250663380257	-57.9069480178552\\
4	-477.875848056489\\
-1.39110132409616	-99.5375284723912\\
0.414944200171931	-26.5126184660255\\
-0.406798244108131	-25.7454432791135\\
3.01594949740786	-299.184305459913\\
1.46907582034777	-107.821817689331\\
-1.13349422466416	-77.7040831650081\\
0.157275950403118	-9.92194902653966\\
-0.173115117097836	-10.8392076552327\\
-1.69590803372074	-128.418447820348\\
3.50822997100346	-382.738183050444\\
0.673024751063565	-43.9842821454837\\
2.59416060641547	-236.674394151674\\
-0.641662797019933	-41.3770581913273\\
-1.25031928857629	-87.318774153475\\
1.8095076348863	-141.395052292175\\
0.0620755282836285	-3.90916265345527\\
-0.0830391199237965	-5.18986382643874\\
-1.88345365382099	-147.912566874219\\
1.1649083925966	-81.2597985662434\\
-1.01670220350326	-68.5255104080773\\
3.78948605216817	-435.709798531919\\
0.298331049247622	-18.9267970181355\\
-1.53192636535221	-112.461923682305\\
-0.268101248751634	-16.8430114755392\\
2.80537447038813	-266.948482552522\\
-0.0303690591802592	-1.89730381004085\\
};
  \addlegendentry{$(\mu_j , \lambda_{\min}(\mu_j))$}

\end{axis}
\end{tikzpicture}%
     \label{fig2a}
	}
 	\hfil
 	\subfigure[Relative approximation error over $\calD$.]{
	    % This file was created by matlab2tikz.
%
\begin{tikzpicture}

\begin{axis}[%
	width=\imageWidth,
	height=\imageHeight,
	scale only axis,
	scaled ticks=false,
	grid=both,
	grid style={line width=.1pt, draw=gray!10},
	major grid style={line width=.2pt,draw=gray!50},
	axis lines*=left,
	axis line style={line width=\lineWidth},
xmin=-2,
xmax=4,
xlabel style={font=\color{white!15!black}},
xlabel={$\mu$},
ymode=log,
ymin=1e-16,
ymax=5e-09,
yminorticks=true,
ylabel style={font=\color{white!15!black}},
	axis background/.style={fill=white},
	legend style={%
		legend cell align=center, 
		align=left, 
		font=\tiny,
		draw=white!15!black,
        %text height=1.6cm,
        minimum height=.7cm,
		at={(0.50,0.80)},
		anchor=south,
        },
]

\addplot [color=mycolor1, line width=\lineWidth]
  table[row sep=crcr]{%
-2	1.23785887538948e-15\\
-1.92405063291139	1.2063729349952e-12\\
-1.84810126582279	7.17180084848874e-13\\
-1.77215189873418	1.42150264634192e-12\\
-1.69620253164557	4.42540306879911e-16\\
-1.62025316455696	6.61462830584784e-13\\
-1.54430379746835	3.47668022187051e-14\\
-1.46835443037975	5.51148157925315e-13\\
-1.39240506329114	4.27806075734482e-16\\
-1.31645569620253	1.48265518014104e-12\\
-1.24050632911392	8.88865298179451e-14\\
-1.16455696202532	6.82413110146053e-13\\
-1.08860759493671	8.68424303838717e-13\\
-1.0126582278481	1.08329200576382e-14\\
-0.936708860759493	1.74125472596929e-12\\
-0.860759493670885	3.11230183576667e-13\\
-0.784810126582279	1.12734374805118e-11\\
-0.708860759493671	9.90483555342513e-12\\
-0.632911392405063	2.4550844074474e-13\\
-0.556962025316455	1.46342575434861e-11\\
-0.481012658227849	1.06426671767381e-11\\
-0.405063291139241	4.9896046179753e-15\\
-0.329113924050633	2.1128634887221e-12\\
-0.253164556962025	1.72776944761577e-13\\
-0.177215189873419	2.28903634122696e-14\\
-0.101265822784811	1.69985982985687e-12\\
-0.0253164556962027	4.08000883437094e-11\\
0.0506329113924053	8.02166278849063e-13\\
0.126582278481013	2.25963679772878e-12\\
0.202531645569621	4.19432264030514e-12\\
0.278481012658228	7.6187655932692e-13\\
0.354430379746836	2.34753311515295e-12\\
0.430379746835444	7.44124351926317e-13\\
0.506329113924052	2.98185780141383e-11\\
0.582278481012658	3.74951543088004e-11\\
0.658227848101266	1.59755196386506e-12\\
0.734177215189874	2.08927629580895e-11\\
0.810126582278482	3.47007223403581e-11\\
0.886075949367088	4.07638462555059e-12\\
0.962025316455696	8.43873474953835e-12\\
1.0379746835443	2.19147395509897e-11\\
1.11392405063291	6.52831863823228e-12\\
1.18987341772152	1.56504911560569e-12\\
1.26582278481013	1.57011229386742e-11\\
1.34177215189873	1.69432967800149e-11\\
1.41772151898734	4.01788203030116e-12\\
1.49367088607595	8.84490070407491e-13\\
1.56962025316456	9.82554253367562e-12\\
1.64556962025316	1.32104898619606e-11\\
1.72151898734177	5.99948852219529e-12\\
1.79746835443038	1.20880607406043e-13\\
1.87341772151899	2.65744714454514e-12\\
1.94936708860759	7.18711211863688e-12\\
2.0253164556962	6.41619420082019e-12\\
2.10126582278481	2.11368523074639e-12\\
2.17721518987342	2.02150705319825e-15\\
2.25316455696202	1.30291742740144e-12\\
2.32911392405063	2.82907243741738e-12\\
2.40506329113924	2.37657497394881e-12\\
2.48101265822785	9.02400278749577e-13\\
2.55696202531646	7.00868254788596e-14\\
2.63291139240506	3.65139141954106e-14\\
2.70886075949367	7.90173884180763e-14\\
2.78481012658228	6.03088090236899e-15\\
2.86075949367089	3.63498317635216e-14\\
2.93670886075949	7.33308700510347e-14\\
3.0126582278481	1.90325339120655e-16\\
3.08860759493671	2.68862533520712e-13\\
3.16455696202532	1.28183484702544e-12\\
3.24050632911393	2.39348133239064e-12\\
3.31645569620253	2.43402327566702e-12\\
3.39240506329114	1.27806590445445e-12\\
3.46835443037975	1.64382758751435e-13\\
3.54430379746836	1.10381771383747e-13\\
3.62025316455696	6.07782766568569e-13\\
3.69620253164557	5.2813285593972e-13\\
3.77215189873418	2.97145517837797e-14\\
3.84810126582279	3.51940364692867e-13\\
3.92405063291139	9.2712025737312e-13\\
4	1.18950181500047e-16\\
};
\addlegendentry{${\calE}^{\calV}_r(\mu)$ (see \eqref{eqn:rel:err})}

\end{axis}
\end{tikzpicture}%
        \label{fig2b}
	}
 
	\caption{\emph{(Concerns Example 2)}\: Dense matrix $A(\mu)\in\R^{\sn\times \sn}$ as in \eqref{eqn:ex2:aff}, $\sn = 2000$, with projected matrix $A^V(\mu)\in\R^{\sd \times \sd}$, $\sd = 32$.}
	\label{fig2}
\end{figure}

\Cref{alg:sf} terminates with a reduced space of dimension $\sd = 32$. As shown in \Cref{fig2a}, the projected
eigenvalue $\lambda^{\calV}_{\min}(\mu)$ accurately captures the behavior of the full eigenvalue 
$\lambda_{\min}(\mu)$ across the domain. Notably, the algorithm selects more interpolation points 
near $\mu = 0$, where $\lambda_{\min}(\mu)$ is less smooth and not simple. The relative error is 
shown in \Cref{fig2b}, confirming that the desired tolerance is satisfied uniformly.

\begin{figure}[t]
	\centering
	    \begin{tikzpicture}
\begin{axis}[
        width=\imageWidth,
	height=0.9*\imageHeight,
 scale only axis,
	scaled ticks=false,
	grid=both,
	grid style={line width=.1pt, draw=gray!10},
	major grid style={line 
        width=.2pt,draw=gray!50},
	axis lines*=left,
	axis line style={line width=\lineWidth},
    enlargelimits=0.10,
    x=1.5cm,
    legend style={%
		legend cell align=left,
		font=\small,
		draw=white!15!black,
        minimum height=.7cm,
    at={(1.3,1)},
      anchor=north},
      ymin=0,
      ymax=135,
    ylabel={\CT (s)},
    axis background/.style={fill=white},
    symbolic x coords={A,B,C,D,E},
    xtick={A,B,C,D,E},
    bar width=25pt,
    ]
\addplot[ybar,fill=mycolor1,forget plot] coordinates {(A,127)};
\addplot[ybar,fill=mycolor2,forget plot] coordinates {(B,21)};
\addplot[ybar,fill=mycolor3,forget plot] coordinates {(C,106)};
\addplot[ybar,fill=mycolor4,forget plot] coordinates {(D,63)};
\addplot[ybar,fill=mycolor5,forget plot] coordinates {(E,58)};

 \addlegendimage{empty legend}
\addlegendentry{A. \Cref{alg:sf}}
 \addlegendimage{empty legend}
\addlegendentry{B. \Cref{alg:sf} $\setminus$ \EIGOPT}
 \addlegendimage{empty legend}
\addlegendentry{C. \EIGOPT}
 \addlegendimage{empty legend}
\addlegendentry{D. $H_r^{j}(\mu)$ evaluations}
 \addlegendimage{empty legend}
\addlegendentry{E. Linear programs}
\end{axis}
\end{tikzpicture}
	\caption{\emph{(Concerns Example 2)}\: Computational time (\CT, in seconds) for \Cref{alg:sf} and its key components.}
	\label{fig2.2}
\end{figure}

In \Cref{fig2.2}, we report the distribution of computation time among the components of \Cref{alg:sf} for constructing the subspace ${\mathcal V}$ of dimension $\sd = 32$ and the reduced matrix-valued function $A^V(\mu)$. Bar A shows the total runtime of \Cref{alg:sf} until convergence with tolerance $\varepsilon=10^{-8}$. Bar B corresponds to lines \ref{alg:large_eigs}–\ref{alg:sf:line:orth} of \Cref{alg:sf}, including updates of $A^{V_j}(\mu)$, which take only a small portion of total time. Bar C shows time inside $\EIGOPT$ (line \ref{alg:sub_prob} of \Cref{alg:sf}). Bar D isolates the time for evaluating $H_r^{(j)}(\mu)$ in $\EIGOPT$, while Bar E shows how much of that is due to solving the linear programs in line \ref{alg:lb:lp} of \Cref{alg:lb}. The evaluation of $H_r^{(j)}(\mu)$ takes considerable runtime, mostly due to linear program solutions, suggesting that evaluating the upper bound $\lambda^{\calV_j}_{\min}(\mu)$ is cheaper than computing the lower bound $\lambda^{(j)}_{\LB}(\mu)$. As Bars C and D show, evaluations of $H_r^{(j)}(\mu)$ account for more than half of \EIGOPT's runtime. However, as the number of parameters grows, other internal tasks in \EIGOPT may become significant.

\subsection{Test cases from parametrized {\PDE}s}

In this section, we examine non-Hermitian matrices. We provide an instance with $5$ parameters where we implement the hybrid approach outlined in \Cref{sec:NE:exm1bis}. Following this, we apply our method to approximate the norm of the resolvent across a compact region of the complex plane. The error we consider here is
\begin{equation}\label{eqn:err:sv}
    \calN_r(\mu)\;\vcentcolon=\;\frac{\sigma_{\min}^{\calV}(\mu)-\sigma_{\min}(\mu)}{\sigma_{\min}^{\calV}(\mu)}.
\end{equation}

\subsubsection{Heston model}
This test example pertains to the parametric semi-discretized form of the Heston equation; see \cite{Hes93} for details on the model and \cite{HouF08} for the employed discretization scheme. Consequently, one has to deal with the following $5$-parameters dependent matrix
\begin{align}\label{eqn:H:aff}
\begin{aligned}
	A(\mu) \;=&\; \mu_1 \mu_2 A_1+A_2+\mu_2A_3+\mu_1^2A_4+\mu_3\mu_4A_5+\mu_3A_6, \\
 \mu \in \calD \;=&\; [0.18, 0.4] \times [0.001, 0.2]\times[1.2,3]\times[0.08,0.15]\times[0.21,0.9],
\end{aligned}
\end{align}
where $A_i\in\R^{\sn\times\sn}$, for $i\in\{1,\ldots,6\}$ and $\sn = 10^4$, is sparse and non-Hermitian. We run \Cref{alg:sv} with $\varepsilon=10^{-4}$ allowing \EIGOPT to evaluate the objective function $S^j_r(\mu)$ (defined in (\ref{eqn:sure:err:sv_rel}))
only $100$ times for each $j$. To address the limitations of performing at most $100$ iterations within \EIGOPT, we employ the hybrid \EIGOPT-\SSCM approach as outlined in \Cref{sec:NE:exm1bis}. Consequently, after the convergence of \Cref{alg:sv}, we enhance our subspaces by evaluating them on a Cartesian product grid, which is constructed by sampling $6$ Chebyshev points in each parameter axis. The algorithm concludes by producing two subspaces, $\calV$ and $\calU$, each with dimension $\sd=110$. We validate our results by measuring the error defined in \eqref{eqn:err:sv} across a randomly generated grid comprising $7^5=16807$ elements within the domain $\calD$. The maximum error value observed from \eqref{eqn:err:sv} was $2\cdot10^{-4}$, which exceeds the prescribed target accuracy; however, it remains highly precise across the entire domain. Notably, in just $314$ instances out of $16807$, the error surpassed the specified error tolerance $\varepsilon =10^{-4}$.
\begin{figure}[t]
	\centering
	    \begin{tikzpicture}
\begin{axis}[
        width=1.2*\imageWidth,
	height=\imageHeight,
 scale only axis,
	scaled ticks=false,
	grid=both,
	grid style={line width=.1pt, draw=gray!10},
	major grid style={line 
        width=.2pt,draw=gray!50},
	axis lines*=left,
	axis line style={line width=\lineWidth},
    enlargelimits=0.10,
    x=1.5cm,
    legend style={%
		legend cell align=left,
		font=\small,
		draw=white!15!black,
        minimum height=.7cm,
    at={(1.4,1)},
      anchor=north},
      ymin=0,
      ymax=1000,
    ylabel={\CT (s)},
    axis background/.style={fill=white},
    symbolic x coords={A,B,C,D,E},
    xtick={A,B,C,D,E},
    bar width=25pt,
    ]
\addplot[ybar,fill=mycolor1,forget plot] coordinates {(A,928)};
\addplot[ybar,fill=mycolor3,forget plot] coordinates {(B,550)};
\addplot[ybar,fill=mycolor4,forget plot] coordinates {(C,160)};
\addplot[ybar,fill=mycolor2,forget plot] coordinates {(D,351)};

 \addlegendimage{empty legend}
\addlegendentry{A. \EIGOPT-\SSCM}
 \addlegendimage{empty legend}
\addlegendentry{B. \EIGOPT }
 \addlegendimage{empty legend}
\addlegendentry{C. $S_r^{j}(\mu)$ calls in \EIGOPT}
 \addlegendimage{empty legend}
 \addlegendentry{D. \SSCM}

\end{axis}
\end{tikzpicture}
	\caption{\emph{(Heston model)}\: Computational time (\CT, in seconds) for \Cref{alg:sv} with \SSCM and its key components.}
	\label{fig2.2:H}
\end{figure}
In \Cref{fig2.2:H}, we present the distribution of computational time among the ingredients of the hybrid \EIGOPT-\SSCM approach. It is noteworthy that, as indicated during the introduction of \EIGOPT, the proportion of the evaluation time taken by the computation of $S^{j}_r(\mu)$ inside \EIGOPT decreases significantly as the number of parameters increases. This is illustrated by comparing with the yellow and green bars for $p=1$ in \Cref{fig2.2}. 
Therefore, it is essential to restrict the maximum number of iterations \EIGOPT conducts when $p>2$.

\subsubsection{Black-Scholes test problem}
We next consider a model derived from the semi-discretized Black-Scholes operator \cite{BlaS73}, using the scheme from \cite{HouW09}. The parameters $\mu = (\sigma, r)$ represent volatility and interest rate, respectively, and the matrix
dependent on these parameters is
\begin{equation}\label{eqn:BS:aff}
	A(\mu) \;\vcentcolon=\; \frac{1}{2}\sigma^2 A_1 + r A_2, 
	\quad\quad \mu \in \calD := [0.05, 0.25] \times [10^{-3}, 2\times 10^{-2}],
\end{equation}
where $A_1, A_2\in\R^{\sn\times\sn}$, $\sn = 2 \times 10^4$, are sparse and non-Hermitian. Running \Cref{alg:sv} with $\varepsilon=10^{-4}$ and \EIGOPT with $1000$ as the maximum number of iterations for the inner optimization yields subspaces ${\mathcal V}$ and ${\mathcal U}$ of dimension $\sd=9$. \Cref{fig8.3} confirms that the relative error \eqref{eqn:err:sv} stays below the tolerance across $\calD$; to be precise, the maximum relative error over a discrete domain constructed using the Cartesian product of $40$ Chebyshev points in each parameter direction is $4.3\cdot10^{-5}$.
\begin{figure}[t]
	\centering{
     {\input{Example_BS_Plot_1.tex}}
   	}
	\caption{Black-Scholes example \cite{HouW09}; $A(\mu)\in\R^{\sn\times \sn}$, $\sn=2\cdot 10^4$, is sparse and non-Hermitian; 
	the reduced matrix $A^{V}_{\rm{R}}(\mu)\in\R^{\sn\times\sd}$ has $\sd=9$. Approximation error \eqref{eqn:err:sv} over $\calD$
	is illustrated where the color-bar is in $\log_{10}$ scale. Red crosses are the selected interpolation points by \Cref{alg:sv}.}
	\label{fig8.3}
\end{figure}

\subsection{Approximation of the Pseudospectra}
We now apply \Cref{alg:sv} to approximate the pseudospectra of a matrix $\PSm\in {\mathbb C}^{\sn\times \sn}$. 
The $\epsilon$-pseudospectrum \cite{TreE05} of $\PSm$ is defined by
\begin{equation}
    \sigma_{\epsilon}(\PSm)
    		\;\vcentcolon=\;
		\left\{\lapVar\in{\mathbb C}\; \left|\; \|(\lapVar I- \PSm)^{-1}\|\le \frac{1}{\epsilon}  \right. \right\} ,
\end{equation}
and contains all eigenvalues of matrices within a distance $\epsilon$ of $\PSm$ in 2-norm. 
Pseudospectra are useful in analyzing transient dynamics, contour integration, and more 
\cite{GugLN20,GugLM21,GugM23,Tre99,TreE05,GL25}. Computing them efficiently, particularly for large matrices, remains a challenge. 
The popular tool \EIGTOOL \cite{eigtool} 
to compute $ \sigma_{\epsilon}(\PSm)$ is based on evaluating $\|(\lapVar I- \PSm)^{-1}\|$ over a grid and interpolation, 
but is computationally expensive. Alternative approaches include \SSCM \cite{Sir19} and low-rank ODE-based methods (see e.g. [Cha.~3]\cite{bookGL25}).

We target efficient approximation of $\|(\lapVar I- \PSm)^{-1}\|$ over a compact domain $\calD\subset {\mathbb C}$.
Since this equals $1/\sigma_{\min}(\lapVar I - \PSm)$, the task becomes approximating the smallest singular value 
of the matrix
\begin{equation}\label{eqn:PS:ex}
    A(\mu)\; = \; 
    \PSm - \real(\lapVar)I - \imag(\lapVar)\ri I,\quad
     \mu = [\real(\lapVar), \imag(\lapVar)], \quad \lapVar \in\calD.
\end{equation}

Here we experiment with $\PSm$ from the Black-Scholes matrix in \eqref{eqn:BS:aff}, setting $\sn=10^3$ and using the parameter values $\sigma = 0.1$ and $r = 0.01$. We run \Cref{alg:sv} employing a tolerance level of $\varepsilon=10^{-6}$ with a cap of $10^3$ iterations for the inner optimization utilizing \EIGOPT. The algorithm constructs subspaces $\calV$ and $\calU$ with dimension $\sd=36$, and \Cref{figPS_BSa} demonstrates that the relative error in \eqref{eqn:err:sv} stays below $\varepsilon$ across $\calD$. In \Cref{figPS_BSb}, the issue highlighted in \cref{it3} regarding the application of \Cref{alg:sf} to the Hermitian problem $A(\mu)^*A(\mu)$ is emphasized. For the creation of \Cref{figPS_BSb}, we compute the maximum of 
the surrogate error $H_r^{(j)}(\mu)$
at every iteration $j$ on a discrete grid $\Xi$ formed from the Cartesian product of $40$ parameters sampled in each direction. The surrogate error becomes stagnant due to the inadequacy of linear programming in achieving an accurate approximation of the smallest eigenvalue.

\begin{figure}[t]
	\centering
	\subfigure[Approximation error \eqref{eqn:err:sv} over $\calD$ with color-bar in $\log_{10}$ scale. Red crosses are the selected interpolation points by \Cref{alg:sv}.]{
		\input{Example_BS_PS_Plot_1.tex}
   		\label{figPS_BSa}
   	}
    \hfil
	\subfigure[Decay of $\max_{\mu\in\Xi} H_r^{(j)}(\mu)$ with respect to the iteration
     counter $j$ of \Cref{alg:sf} for the problem $A(\mu)^*A(\mu)$.]{
		% This file was created by matlab2tikz.
%
\begin{tikzpicture}

\begin{axis}[%
	width=0.9*\imageWidth,
	height=\imageHeight,
	scale only axis,
	scaled ticks=false,
	grid=both,
	grid style={line width=.1pt, draw=gray!10},
	major grid style={line width=.2pt,draw=gray!50},
	axis lines*=left,
	axis line style={line width=\lineWidth},
xmin=0,
xmax=200,
xlabel style={font=\color{white!15!black}},
xlabel={$j$},
ymode=log,
ymin=1e-1,
ymax=1e11,
yminorticks=true,
ylabel style={font=\color{white!15!black}},
	axis background/.style={fill=white},
	legend style={%
		legend cell align=left, 
		align=left, 
		font=\tiny,
		draw=white!15!black,
		at={(0.99,0.9)},
		anchor=north east,},
]

\addplot [color=mycolor1, line width=\lineWidth, mark=none]
  table[row sep=crcr]{%
1	1841105.26870004\\
2	1938396.24556671\\
3	7043960.67101256\\
4	9914341.67766001\\
5	22216760.1629291\\
6	137860934.403098\\
7	275743963.675959\\
8	197813852.649228\\
9	259513531.985106\\
10	1239404245.3958\\
11	4538050153.53286\\
12	1271744512.37168\\
13	9672313655.94782\\
14	2388960592.72903\\
15	34423414703.71\\
16	20092237681.4896\\
17	1106145236.59079\\
18	11592513899.1011\\
19	1369670130.09945\\
20	692904697.310413\\
21	3630192.18924287\\
22	1102100.32211277\\
23	148785.674754204\\
24	142253.732476453\\
25	21880.1167600529\\
26	11699.3566686093\\
27	8232.12062869933\\
28	18374.2812677719\\
29	13918.2971004282\\
30	17031.6626781608\\
31	14107.452105026\\
32	9410.28512568494\\
33	4180.603772207\\
34	3962.41394487089\\
35	993.898718318769\\
36	737.571450138377\\
37	613.411794888586\\
38	1246.71475512234\\
39	383.809313200651\\
40	343.234031657879\\
41	216.666501135391\\
42	338.39621613203\\
43	163.891371813323\\
44	81.6962698792102\\
45	68.8193624175377\\
46	34.2789785273907\\
47	17.3577578072932\\
48	16.6700090886326\\
49	11.9393370876439\\
50	11.2470102957261\\
51	7.13039204692001\\
52	7.07966129549421\\
53	6.27321199765106\\
54	4.11957954525154\\
55	3.68839310568154\\
56	1.55940076647419\\
57	1.12618067245773\\
58	1.52550706846464\\
59	1.11733709367158\\
60	1.28441366742704\\
61	1.52730427409704\\
62	1.33657452996772\\
63	1.25955981083949\\
64	1.16729626558839\\
65	0.959562468430012\\
66	1.34340749998318\\
67	1.36305910433263\\
68	1.01894503908142\\
69	1.18817719577042\\
71	1.1129264456998\\
72	1.13699393664026\\
73	1.53243194510608\\
74	0.987100131565957\\
75	1.17920274983176\\
76	1.26541466070092\\
77	1.47306843974588\\
78	1.34260514232751\\
79	1.19921805936899\\
81	1.54738144993181\\
82	1.17079067928036\\
83	1.33864726963754\\
84	1.7264170392915\\
85	1.15931833987895\\
86	1.20149791601102\\
87	1.16152852801095\\
88	0.95261462738812\\
89	1.08191529084948\\
90	1.56058453669242\\
91	1.47982743734894\\
92	1.05862598479955\\
93	1.02553998383313\\
94	0.910544297524841\\
95	1.16289150304238\\
96	1.05156602489717\\
97	1.18803225096763\\
98	1.26536003048785\\
99	1.79742923975645\\
100	0.867813140839008\\
101	0.913550134873604\\
102	1.0384193126359\\
103	0.909275417812018\\
104	1.16830100284978\\
105	1.29965692459859\\
106	0.959197626985832\\
108	0.866421944350752\\
109	1.40506117925806\\
110	0.957597089408461\\
111	1.43318384445308\\
112	1.29185499318819\\
113	1.40750419722414\\
114	1.0532524574213\\
115	1.28367028684972\\
116	1.20890027636391\\
117	1.44628816928551\\
118	1.23907997840274\\
119	1.24215659225239\\
120	0.993920516969005\\
121	1.15406996659813\\
122	1.30966231478053\\
123	1.19912504055407\\
124	1.21928144677691\\
125	1.55566558216322\\
126	1.30671258835459\\
127	1.26250725405136\\
128	1.45124750455423\\
129	1.01413143832991\\
130	1.04942628740837\\
131	1.7312548266426\\
132	1.17166645942621\\
133	1.5726272460537\\
134	1.33907867015266\\
135	1.18797386701412\\
136	1.24676577122696\\
137	1.01216031283804\\
138	1.23308061994369\\
139	1.08138402370733\\
140	1.77090454814801\\
141	1.3667592615135\\
142	1.27773292924576\\
143	1.71309161737539\\
144	1.23111958880672\\
145	1.42303152923278\\
146	1.19412638495152\\
147	1.84906529084744\\
148	0.990161358280989\\
149	1.40922712144135\\
150	1.14581858236857\\
151	1.74171477853987\\
152	1.04893870450049\\
153	1.18875668031626\\
154	1.32140486687276\\
155	1.02596469633638\\
156	1.60373294234182\\
157	1.29612890103798\\
158	1.36754125862464\\
159	1.24181410143082\\
160	1.26602553415748\\
161	1.10464330802557\\
162	1.04631516788232\\
163	1.20589608948709\\
164	1.97762638959812\\
165	1.34767707603295\\
166	1.54582856751229\\
167	1.02280219445206\\
168	0.872396544540392\\
169	1.65151647180221\\
170	1.6832841466626\\
171	2.00490676931237\\
172	1.25570722482237\\
173	1.16829592297722\\
174	1.31782291244245\\
175	0.986668862768093\\
176	1.24384283867727\\
177	1.1849045594676\\
178	1.39451898856824\\
179	1.3455045242291\\
180	1.12048271654956\\
181	1.31744155170145\\
182	1.38246981018869\\
183	1.25968279059698\\
184	1.42585150666052\\
185	1.09910178405302\\
186	1.46855452377612\\
187	1.14662394384272\\
188	1.17030859443286\\
189	0.789954348430065\\
190	1.50153384657324\\
191	1.49033908231302\\
192	1.04686184593681\\
193	1.26385742304821\\
194	1.90511875645997\\
195	1.27212516672274\\
196	1.75132002700643\\
197	1.00808074019253\\
198	1.42291132225414\\
199	1.49852464314704\\
200	1.33457298470705\\
201	1.20794741811266\\
202	1.41241459638819\\
203	1.30109428422478\\
204	1.78140341230319\\
205	1.63514177828554\\
206	0.968348501229638\\
207	1.3183009308857\\
208	1.66299434380687\\
209	1.2556620783191\\
210	1.59082527495932\\
211	1.67046119197254\\
212	1.14906432441502\\
213	1.58141546368975\\
214	0.951005367426342\\
215	1.39958484069184\\
216	1.37119960799418\\
217	1.63647779486067\\
218	1.90669510445221\\
219	1.37179947538378\\
220	1.50973838054388\\
221	2.34513068769251\\
222	1.79319021828022\\
};
\addlegendentry{$\max_{\mu \in \Xi} H_r^{(j)}(\mu)$ (see \eqref{eqn:rel:H})}

\end{axis}
\end{tikzpicture}%
   		\label{figPS_BSb}
   	}
	
	\caption{Pseudospectrum approximation example; $A(\mu)\in\R^{\sn\times \sn}$ with $\sn=10^3$ 
	is defined in \eqref{eqn:PS:ex} using $\PSm$ from the Black-Scholes operator \eqref{eqn:BS:aff} with $\sigma = 0.1$, $r = 0.01$; 
	the reduced matrix $A^{V}_{\rm{R}}(\mu)$ obtained from \Cref{alg:sv} has $\sd=36$.}
	\label{figPS_BS}
    \end{figure}

%-----------------------------------------------------------------------------%

\section{Conclusion and outlook} 
We have considered the approximation of the smallest eigenvalue 
$\lambda_{\min}(\mu)$ of a large-scale parameter-dependent Hermitian
matrix $A(\mu)$ for all $\mu$ in a compact set $\calD \subseteq {\mathbb R}^p$
by that of $V^\ast A(\mu) V$, where the columns of $V$ form an orthonormal
basis for a carefully constructed small-dimensional subspace $\calV$. The subspace $\calV$ is
constructed iteratively as follows: at every iteration, we compute the parameter value 
$\widetilde{\mu} \in {\mathbb R}^p$ where the gap between an upper bound and a lower 
bound for $\lambda_{\min}(\mu)$ is maximized, and the subspace is expanded with the 
inclusion of the eigenvectors of $A(\widetilde{\mu})$ corresponding to its smallest eigenvalues. 
The lower and upper bounds we rely on are borrowed from \cite{SirK16}. However, for the lower bound 
we prove a series of novel results that are useful to understand the general subspace framework for eigenvalue approximations, and, beyond, that aid us to formally establish the convergence of the framework in an infinite-dimensional setting, 
which constitutes another contribution of this work.
 Moreover, unlike \cite{SirK16}, which uses the parameter value that maximizes the gap
in a finite and discrete subset of $\calD$, we use the parameter value that
maximizes the gap over the continuum of the domain $\calD$. We conduct a comprehensive examination of the benefits and drawbacks associated with this methodology through various numerical experiments. In the second part, we deal with the approximation of the smallest singular 
value $\sigma_{\min}(\mu)$ of $A(\mu)$ in cases where $A(\mu)$ is not Hermitian. 
A first thought is to apply the framework for approximating the smallest eigenvalue 
to $A^\ast(\mu) A(\mu)$, but this suffers from two important drawbacks that prevent its efficiency in many scenarios, namely
increase in the number of terms in the affine decomposition due to squaring, 
and inadequacy of the lower bound from linear programs to approximate the singular value well.
Thus, we propose an interpolatory framework that operates 
directly on $A(\mu)$ to approximate its smallest singular value, overcoming the downsides
of the $A^\ast(\mu) A(\mu)$ approach. A disadvantage
of this interpolatory framework is that it does not provide a certificate for the approximation 
error of the reduced problem; 
however, numerical experiments indicate that the framework leads to accurate projected problems.

The results obtained from this study have been effectively implemented for and applied to quantum spin systems; 
we refer to \cite{ManSZ25}. Current research efforts are focused on surpassing the limitations of the linear programming approximation in the context of the smallest singular value approximation. This is vital for ensuring error certification while preserving efficiency when utilizing the subspace framework.

%-----------------------------------------------------------------------------%

\appendix

\section{Lipschitz continuity of \texorpdfstring{$H^{(j)}(\mu)$}{TEXT}}\label{sec5}
This appendix is devoted to showing the uniform Lipschitz continuity of $H^{(j)}(\mu)$ with respect to $\mu$ over 
all $j$ in the infinite-dimensional setting, which is employed in the global convergence proof, that is 
the proof of \Cref{teo0}. The derivation of the uniform Lipschitz continuity of $H^{(j)}(\mu)$ 
here is rather involved, and can possibility be omitted at first by a reader who prefers to avoid technicalities. 
The structure of the appendix is as follows. In Section \ref{sec: 5.1}, we rely on the theory of 
invariant subspaces, as outlined in \cite[Cha.~5]{Stewart1990}, and exclusively work with \textit{simple} 
invariant subspaces, as defined in \cite[Cha.~5, Def.~1.2]{Stewart1990}, to show the Lipschitz 
continuity of eigenvectors associated with simple invariant subspaces. In Section \ref{sec: 5.2}, we derive 
the uniform Lipschitz continuity results for the lower bound defined in \eqref{eq:defn_LB}, and finally, 
in Section \ref{sec: 5.3}, we state the main results concerning the uniform Lipschitz continuity
of $H^{(j)}(\mu)$. 

Recall that in the infinite-dimensional setting we deal with here, $A_i:\ell^2(\N)\rightarrow\ell^2(\N)$ 
is a compact self-adjoint operator for $i=1,\ldots,\kappa$. With this framework, we can interpret 
$A(\mu)$ with the structure \eqref{int1} as an infinite-dimensional Hermitian matrix. The 
action of $A(\mu)$ over a subset of $\ell^2(\N)$ follows straightforwardly as in the finite-dimensional case.
Without loss of generality, we will always consider the eigenvectors of $A(\mu)$ as normalized. 
Below, we recall a Lipschitz continuity result, i.e. \cite[Lem.~2.1]{KanMMM18}, for the $j$-th smallest eigenvalue 
$\lambda_{j}^{\calV}(\mu)$ of $V^\ast A(\mu) V$ that will be employed in the subsequent subsections.
\begin{lemma}(Lipschitz Continuity of $\lambda^{\calV}_{j}(\mu)$)\label{lemma:lip:con:eig}
    Let $j$ be a positive integer.
    There exists a positive real scalar $\gamma_\lambda$ such that for every subspace $\calV$ of $\ell^2(\N)$
    such that $\text{dim} \, \calV \geq j$, we have
    \begin{equation*}
        \big| \lambda^{\calV}_{j}(\mu)-\lambda^{\calV}_{j}(\widetilde{\mu}) \big| \;\le\; 
        									\gamma_{\lambda} \|\mu-\widetilde{\mu}\|\;
        \quad\; \text{for all}\;\; \mu,\widetilde{\mu}\in\calD	\:	.
    \end{equation*}
\end{lemma}
\begin{comment}
\begin{proof}$\:$
    It follows from Weyl's theorem \cite[Thm.~4.3.1]{Horn1985} and
    its infinite-dimensional extension that
    \begin{align*}
        \begin{aligned}
            \big| \lambda^{\calV}_j(\mu)-\lambda^{\calV}_j(\widetilde{\mu}) \big| \;\le\; \|A^{V}(\mu)-A^{V}(\widetilde{\mu})\|\;\le\;\sum_{i=1}^{\kappa}|\theta_i(\mu)-\theta_i(\tmu)|\|A_i^{V}\|
        \end{aligned}
    \end{align*}
    for all $\mu,\widetilde{\mu} \in\calD$. In the last summation, by employing $\|A_i^{V}\|\le\|A_i\|$, as well as
    the real analyticity of $\theta_i(\mu)$, which implies its Lipschitz continuity, say with Lipschitz constant $\gamma_i$, 
    we obtain
    \begin{equation*}
        \big| \lambda_j^{\calV}(\mu)-\lambda_j^{\calV}(\widetilde{\mu}) \big|
        		\;\le\;
	\left(\sum_{i=1}^{\kappa}\gamma_i\|A_i\|\right)\|\mu - \widetilde{\mu}\|
	\quad\quad		
    \end{equation*}
    for all $\mu,\widetilde{\mu}\in\calD$.
\end{proof}
\end{comment}
\subsection{Lipschitz continuity of eigenspaces (eigenvectors)}\label{sec: 5.1}
First, let us recall the definitions of invariant and simple invariant subspaces.
\begin{definition}[Invariant subspace]
    Let $\calX$ be any subspace of $\ell^2(\N)$, then $\calX$ is said to be an invariant subspace with respect to $A:\ell^2(\N)\rightarrow\ell^2(\N)$ if
    \begin{equation*}
       Av\in\calX,\quad\forall\, v\in \calX.
    \end{equation*}
\end{definition}
\begin{definition}[Simple invariant subspace]
    Let $\calX$ be an invariant subspace of $A$, and $X  : {\mathcal V} \rightarrow \ell^2(\N)$ be a linear 
    isometry (i.e., $X$ satisfies $X^\ast X = I$) from a subspace ${\mathcal V}$ of 
    $\ell^2(\N)$ such that $\calX = \range(X)$. Moreover, let $\calS(A)$ denote 
    the point-spectrum of the operator $A$, and $X^\bot  : {\mathcal V}^\bot \rightarrow \ell^2(\N)$
    be a linear map such that $X \oplus X^\bot$ is unitary.
    The set $\calX$ is said to be a simple invariant subspace of $A$ if 
    \begin{equation*}
			\calS(X^* A X) \cap \calS([ X^\bot ]^*A X^\bot)		\;	=	\;	\emptyset .
    \end{equation*}
\end{definition}

Any eigenspace of $A$ is also an invariant subspace. An eigenspace is simple if the eigenvalues associated with the eigenvectors forming this space are not associated with any of the eigenvectors generating the orthogonal complement to the eigenspace. The next theorem is fundamental to show that the lower bound $\lambda^{(j)}_{\LB}(\mu)$ is a 
Lipschitz continuous function. 
\begin{theorem}[Lipschitz continuity of simple invariant subspaces]\label{Cor: lipeig}
Suppose that $m$ is such that the invariant subspace $\calX(\mu)$ 
associated with the eigenvalues $\lambda_j(\mu)$, $j = 1,\ldots,m$ of $A(\mu)$ is simple for all $\mu \in \calD$.
Moreover, suppose
\begin{equation}\label{eqn:sep}
    	\Delta(\mu,\widetilde{\mu},A,\calX)
		\; \vcentcolon= 	\;
	\left|
		\lambda_m(\mu)-\lambda_{m+1}(\widetilde{\mu})
	\right|
		\;	>	\;	0,	\quad
	\text{for all}\quad \mu,\widetilde{\mu}\in \calD	\:	.
\end{equation}
Then, there exists $\gamma > 0$ such that for all $\mu , \widetilde{\mu} \in \calD$ the following hold: For any linear isometry $X(\mu) : \calV \rightarrow \ell^2(\N)$ from a subspace ${\mathcal V}$ of 
    $\ell^2(\N)$ such that $\calX(\mu) = \range(X(\mu))$, there is a linear isometry $X(\widetilde{\mu}) :  \calV \rightarrow \ell^2(\N)$ 
that satisfies $\calX(\widetilde{\mu}) = \range(X(\widetilde{\mu}))$ and
	$\|X(\widetilde{\mu})	-	X(\mu)\|		\:	\le	\:	
		\gamma	\|\widetilde{\mu}-\mu\|$.
\end{theorem}
\begin{proof}$\:$
    This is a direct consequence of the analyticity of isolated eigenspaces (i.e., simple invariant subspaces); see \cite[Thm.~1]{GruSH22}.
\end{proof}

To be able to apply \Cref{Cor: lipeig} in our setting, we assume a separation between the $\sr$th
and $(\sr + 1)$-st smallest eigenvalues of $A(\mu)$. This assumption together with an assumption on the coarseness
of the initial points $\mu_{1,1}, \dots, \mu_{1,n}$ for \Cref{alg:sf} are formally stated next.
\begin{assumption}\label{ass:eig_sep}
The inequality
\smallskip
\begin{equation}\label{eq:eig_sep}
		\widetilde{\delta}
			\;	:=	\;
	 	\min_{\mu \in {\mathcal D}} \lambda_{\sr + 1}(\mu)
	 	\:	-	\:
		\max_{\mu \in {\mathcal D}} \lambda_{\sr}(\mu)	\;\;	>	\;\;	0
\end{equation}
\vskip -.3ex
\noindent
holds. Moreover, \Cref{alg:sf} is initialized with the multiple points 
$\mu_{1,1}, \dots , \mu_{1,\eta}$ (see \Cref{rem:initialize})  
chosen as the grid-points on a sufficiently fine uniform grid for ${\mathcal D}$.
% $\: h := \delta / (\sqrt{p} \gamma) \:$
\end{assumption}
Assuming $\ell > \sr$, we have the interpolation properties \cite[Lem.~2.3]{KanMMM18}
\[
	\lambda^{\calV_j}_{\sr}(\mu_{1,i})	\;	=	\;	\lambda_{\sr}(\mu_{1,i})
		\quad	\text{and}		\quad
	\lambda^{\calV_j}_{\sr + 1}(\mu_{1,i})	\;	=	\;	\lambda_{\sr + 1}(\mu_{1,i})
\]
for $i = 1, \dots , \eta$.
By using \Cref{lemma:lip:con:eig}, in particular the uniform Lipschitz continuity of the 
eigenvalues $\lambda^{\calV_j}_{\sr}(\mu)$ and $\lambda^{\calV_j}_{\sr+1}(\mu)$ with the Lipschitz
$\gamma_\lambda$ independent of $\calV_j$, we deduce from (\ref{eq:eig_sep}) that
\begin{equation}\label{eq:gap_projected}
		\min_{\mu \in {\mathcal D}} \lambda^{\calV_j}_{\sr + 1}(\mu)
	 		\:	-	\:
		\max_{\mu \in {\mathcal D}} \lambda^{\calV_j}_{\sr}(\mu)
			\;	\geq	\;	\widetilde{\delta}/2	\;	>	\;	0
\end{equation}
for $\mu_{1,1}, \dots , \mu_{1,\eta}$ on a sufficiently fine uniform grid, to be precise,
on a uniform grid where two adjacent points are apart from each other by 
a distance not exceeding $\, \widetilde{\delta} / (2\sqrt{p} \, \gamma_\lambda)$
(i.e., by calculations similar to those in the proof of
\Cref{lemma5.6} below
concerning the gap between $\eta^{(j)}_\ast(\mu)$ and $\lambda^{\calV_j}_1(\mu)$).
The condition in (\ref{eq:gap_projected}) in turn implies 
$\Delta(\mu,\widetilde{\mu},A^{V_j},\calX) \geq \widetilde{\delta}/2 > 0$
for all $\mu, \widetilde{\mu} \in \calD$, where $\calX(\mu)$ is the invariant subspace of $A^{V_j}$ 
spanned by its eigenvectors $w^{V_j}_1(\mu), \dots , w^{V_j}_\sr(\mu)$. Hence, we
arrive at the following result, which we will rely on in the next subsection. To deduce this
result, we note that
$\max_{\mu, \tilde{\mu} \in \calD} \Delta(\mu,\tilde{\mu},A^{V_j},\calX) \geq  \widetilde{\delta}/2$
uniformly over all $j$, and $\| V_j^\ast A_\ell V_j \| \leq \| A_\ell \|$ for $\ell = 1, \dots , \kappa$,
which implies that the Lipschitz constant $\gamma$ can be chosen independent of $j$.
Note that % we view $V_j$ in \Cref{alg:sf} in the infinite-dimensional setting as 
$V_j$ is a linear isometry
from ${\mathbb C}^{\sd}$ to $\ell^2({\mathbb N})$ so that $A^{V_j}(\mu) = V_j^\ast A(\mu) V_j$ is a linear map
acting on ${\mathbb C}^{\sd}$ and the invariant subspaces of $A^{V_j}(\mu)$ are subspaces of ${\mathbb C}^{\sd}$.
\begin{theorem}\label{thm:lipeig2}
Suppose that \Cref{ass:eig_sep} holds, $\ell > \sr$,
and let $\calX_j(\mu)$ denote the invariant subspace of $A^{V_j}(\mu)$ 
spanned by its eigenvectors $w^{V_j}_1(\mu), \dots , w^{V_j}_\sr(\mu)$.
There exists $\gamma_X > 0$ independent of $j$ such that for all $\mu , \widetilde{\mu} \in \calD$
the following hold:
For any matrix $X_j(\mu)$ whose orthonormal columns span $\calX_j(\mu)$, there is a matrix $X_j(\widetilde{\mu})$ 
with orthonormal columns spanning $\calX_j(\widetilde{\mu})$, that satisfies
	\begin{equation*}
		\|X_j(\widetilde{\mu})	-	X_j(\mu)\|		\:	\le	\:	
		\gamma_X	\|\widetilde{\mu}-\mu\|	\:	.
	\end{equation*}
\end{theorem}

\subsection{Uniform Lipschitz continuity of \texorpdfstring{$\lambda^{(j)}_\LB(\mu)$}{TEXT}}\label{sec: 5.2}
In this subsection, we state and prove a series of lemmas with the eventual aim of establishing
the Lipschitz continuity of the lower bound 
$\lambda^{(j)}_\LB(\mu)$ over the domain ${\mathcal D}$ with a Lipschitz constant uniform over all $j$. 

We start by establishing the uniform Lipschitz continuity of $\rho^{(j)}(\mu)^2$, recalling $\rho^{(j)}(\mu)$ 
is defined as in (\ref{eq:lbound_int}). In the proof of this Lipschitz continuity result, we benefit from the 
representation of $U_j(\mu)$ as in \eqref{eq:6} but in the infinite-dimensional setting of the form
\begin{equation}\label{eq:57}
	U_j(\mu)\;
					=	\;
		V_j W_j(\mu),\quad\text{where}\quad
		W_j(\mu) : {\mathbb C}^{\sr} \rightarrow {\mathbb C}^{\sd}, \;\; 
		W_j(\mu) x \;\vcentcolon=\;
								[w^{V_j}_1(\mu),\ldots,w^{V_j}_{\sr}(\mu)] x \;,
\end{equation}
the linear isometry $V_j : {\mathbb C}^{\sd} \rightarrow \ell^2({\mathbb N})$ is such that 
$\calV_j = \text{Im} (V_j)$, and $w^{V_j}_k(\mu)$ 
is the eigenvector of $A^{V_j}(\mu) = V_j^\ast A(\mu) V_j$ corresponding to its $k$-th smallest
eigenvalue $\lambda_k^{\calV_j}(\mu)$.

\begin{lemma}\label{Lemma2}
	Suppose that \Cref{ass:eig_sep} holds, and $\ell > \sr$.
	There exists a real positive scalar $\gamma_\rho$ such that
	\begin{equation}\label{eq:9}
		|\rho^{(j)}(\widetilde{\mu})^2-\rho^{(j)}(\mu)^2|
		\;	\le	\;	\gamma_\rho \|\widetilde{\mu}-\mu\|
		\quad \;
		\text{for all}\;\; \widetilde{\mu},\,\mu \in \calD	\;	,
	\end{equation}
	and for all $j$, where $\rho^{(j)}(\mu)$ is defined as in (\ref{eq:lbound_int}).
\end{lemma}
\begin{proof}$\:$
	Using \eqref{eq:rhomu} and by Weyl's theorem (see \cite[Thm.~4.3.1]{Horn1985} for 
	the finite-dimensional case), we have
		\begin{align}\label{eq:5}
          \begin{aligned}
			&	|\rho^{(j)}(\widetilde{\mu})^2-\rho^{(j)}(\mu)^2|		% \\[.4em]
			% &	\quad\quad\quad	
                =	\;\;	
					\left| \;
	\lambda_{\max} \left(U_j(\widetilde{\mu})^\ast A(\widetilde{\mu})^\ast A(\widetilde{\mu}) U_j(\widetilde{\mu})  
									- \Lambda^{{\mathcal U}_j}(\widetilde{\mu})^2\right)
					\right.
					\; - \;		\\
			&	\quad\quad\quad\quad\quad\quad\quad\quad\quad\quad\quad\quad\quad\quad\quad\quad	\left.
	\lambda_{\max} \left( U_j(\mu)^\ast A(\mu)^\ast A(\mu) U_j(\mu)  - \Lambda^{{\mathcal U}_j}(\mu)^2 \right)
					\; \right| \\[.4em]
			&	\quad\quad\quad	\leq	\;\;	
			\left\|
			U_j(\widetilde{\mu})^\ast A(\widetilde{\mu})^\ast A(\widetilde{\mu}) U_j(\widetilde{\mu})
					-	U_j(\mu)^\ast A(\mu)^\ast A(\mu) U_j(\mu)
			\right\|
			\;	+	\;	%	\\
%			&	\quad\quad\quad\quad\quad\quad\quad\quad
			\left\| \Lambda^{{\mathcal U}_j}(\widetilde{\mu})^2- \Lambda^{{\mathcal U}_j}(\mu)^2 \right\|	\:	.
   \end{aligned}
		\end{align}
	By the definition of $\Lambda^{{\mathcal U}_j}(\mu)$ in \eqref{eq:defn_Lambda_Uj}, we have
	\begin{equation}\label{eq:4}
		\begin{aligned}
			\left\| \Lambda^{{\mathcal U}_j}(\widetilde{\mu})^2- \Lambda^{{\mathcal U}_j}(\mu)^2 \right\|	\;\; = \;\;
		&	\max_{k=1,...,\sr} \: \left| \lambda^{\calV_j}_k(\widetilde{\mu})^2-\lambda^{\calV_j}_k(\mu)^2 \right|	\\[.2em]	
        % \\		\;\; = \;\;	&
	\; = \;\;	&		
			\max_{k=1,...,\sr} \:
			\left|  \left(\lambda^{\calV_j}_k(\widetilde{\mu})+\lambda^{\calV_j}_k(\mu)\right)\left(\lambda^{\calV_j}_k(\widetilde{\mu})
					-	\lambda^{\calV_j}_k(\mu)\right) 	\right|					\\[.2em]
		\;\; \le \;\;	&
			\widetilde{\gamma}
			\max_{k=1,...,\sr} \:
			\left|   \lambda^{\calV_j}_k(\widetilde{\mu})
					-	\lambda^{\calV_j}_k(\mu) 	\right|					
			\; \le  \;\;  		\gamma \,  \left\|\widetilde{\mu}-\mu \right\|  \quad\; \text{for all }\; \widetilde{\mu}, \mu \in \calD \;, 
		\end{aligned}
	\end{equation}
	for some constants $\widetilde{\gamma}$ and $\gamma$ independent of $j$,
	where 
	for the second to the last inequality and the last inequality, we have used 
	the fact that the eigenvalues are uniformly bounded for all $\mu \in \calD$ and
	the Lipschitz continuity of the eigenvalues, i.e., \Cref{lemma:lip:con:eig}, respectively.
	It follows from the representation in (\ref{eq:57}) of $U_j(\mu)$ that 
	$U_j(\mu)^\ast A(\mu)^\ast A(\mu) U_j(\mu)	\;\;	=	\;\;	
				U_j(\mu)^\ast \widehat{A}(\mu) U_j(\mu)	\;\;	=	\;\;
		W_j(\mu)^{\ast} \widehat{A}^{V_j}(\mu)W_j(\mu)$,
	with $W_j(\mu)$ denoting the linear map defined as in \eqref{eq:57}, and $
		\widehat{A}(\mu) = A(\mu)^\ast A(\mu)	\:	,	\quad
			\widehat{A}^{V_j}(\mu)=V^\ast_j \widehat{A}(\mu) V_j$.	
	Thus, we have
	\medskip
	\begin{equation}\label{eq:7}
	\hskip -1ex
		\begin{aligned}
			&	\,
	\left\| U_j(\widetilde{\mu})^\ast A(\widetilde{\mu})^\ast A(\widetilde{\mu}) U_j(\widetilde{\mu})
				-
			U_j(\mu)^\ast A(\mu)^\ast A(\mu) U_j(\mu) \right\|		\\[.5em]
			=		&\,	
	\left\|
		W_j(\widetilde{\mu})^{\ast} \widehat{A}^{V_j}(\widetilde{\mu}) W_j(\widetilde{\mu})
				-
		W_j(\mu)^{\ast} \widehat{A}^{V_j}(\mu)W_j(\mu)
	\right\|		\\[.3em]
			\le		&	\,
			\left\|
			W_j(\widetilde{\mu})^{\ast} \left\{ \widehat{A}^{V_j}(\widetilde{\mu})W_j(\widetilde{\mu})- \widehat{A}^{V_j}(\mu)W_j(\mu) \right\} 
			\right\|
							\,  +  \,
			\left\|
				\left\{ W_j(\mu) - W_j(\widetilde{\mu}) \right\}^\ast \, \widehat{A}^{V_j}(\mu)W_j(\mu)
			\right\| 		\\[.3em]
			\le	&  \,
			\left\|
				\widehat{A}^{V_j}(\widetilde{\mu})
				\left\{ W_j(\widetilde{\mu})-W_j(\mu) \right\} 
			\right\| 
					\, + \, 
			\left\|
					\left\{ \widehat{A}^{V_j}(\widetilde{\mu}) - \widehat{A}^{V_j}(\mu) \right\} W_j(\mu)
			\right\|		
				\,  +	\,
			C \left\|  W_j(\widetilde{\mu}) - W_j(\mu) \right\|		\\[.3em]
			\le		& \;\;
			2C \left\|
						W_j(\widetilde{\mu})-W_j(\mu)
					\right\|   
							\,   +  \,  
					\left\| 
						\widehat{A}^{V_j}(\widetilde{\mu}) - \widehat{A}^{V_j}(\mu)
					\right\| 	\;  ,
		\end{aligned}  
	\end{equation}
	where 
	$\, C\;\vcentcolon=\;\max_{\mu\in\calD} \| \widehat{A}(\mu) \|  \; \geq \;  \max_{\mu\in\calD} \|\widehat{A}^{V_j}(\mu)\|$, 
	and where we have used $\|W_j(\mu)\|=1,\; \forall\mu\in\calD$.
	Moreover, using \Cref{thm:lipeig2}, the steps in \cite[Lem.~$2.1$]{KanMMM18},
	and denoting with $M_j(\mu)$ a matrix representation of the linear map $W_j(\mu)$,
	there must be constants $\gamma_X$ and $\gamma_A$ independent of $j$ 

	such that
	\medskip
	\begin{equation}\label{eq:8}
		\begin{split}
			\left\|
		U_j(\widetilde{\mu})^\ast A(\widetilde{\mu})^\ast A(\widetilde{\mu}) U_j(\widetilde{\mu})
				-
		U_j(\mu)^\ast A(\mu)^\ast A(\mu) U_j(\mu)
			\right\|	
			\;\;	& \le	\;\;	\hskip 20ex		\\[.2em]
2C \left\| M_j(\widetilde{\mu})-M_j(\mu) \right\|	\, + \, \left\| \widehat{A}^{V_j}(\widetilde{\mu})-\widehat{A}^{V_j}(\mu) \right\|
			\;\;	& \le	\;\;	
			2C \gamma_X  \left\| \widetilde{\mu} - \mu \right\|  \, + \,  \gamma_A \left\| \widetilde{\mu} - \mu \right\|	\\
			\;\;	& =	\;\;
			(2C \gamma_X + \gamma_A) \left\|  \widetilde{\mu} - \mu  \right\|	\,	.	
			\hskip 5ex
		\end{split}
	\end{equation}
	We remark that the quantity $\rho^{(j)}(\widetilde{\mu})$ is independent of the orthonormal basis in the columns
	of $M_j(\widetilde{\mu})$ for the invariant subspace $A^{V_j}(\widetilde{\mu})$ spanned by its eigenvectors
	$w^{V_j}_1(\widetilde{\mu}), \dots , w^{V_j}_{\sr}(\widetilde{\mu})$. In the derivations above,  
	  $M_j(\widetilde{\mu})$ is the one satisfying 
	  $\| M_j(\widetilde{\mu}) - M_j(\mu) \| \leq \gamma_X \left\|  \widetilde{\mu} - \mu  \right\|$;
	  the existence of such $M_j(\widetilde{\mu})$ is guaranteed by \Cref{thm:lipeig2}.
	Finally, using inequalities
	\eqref{eq:4} and \eqref{eq:8} in \eqref{eq:5}, we deduce \eqref{eq:9}.
\end{proof}

The next lemma concerns the Lipschitz continuity of $\beta^{(i,j)}(\mu)$ defined in (\ref{eq:beta_i})
for $i = 1, \dots , j$ with a Lipschitz constant independent of $j$,
\begin{lemma}\label{Lemma0.1}
	Suppose that \Cref{ass:eig_sep} holds, and $\ell > \sr$.
	Then, there exists a real positive scalar $\gamma_\beta$ independent of $j$ such that 
	the scalar function $\beta^{(i,j)}(\mu)$ defined in \eqref{eq:beta_i} satisfies
	\begin{equation*}
		\left| \beta^{(i,j)}(\widetilde{\mu}) - \beta^{(i,j)}(\mu) \right|
		\;	\le	\;
		\gamma_\beta  \left\| \widetilde{\mu} - \mu \right\| \qquad
        \mbox{\rm for all} \;\; \widetilde{\mu}, \:  \mu  \in  \calD, \; \mbox{\rm for} \ i=1, \dots, j.
	\end{equation*}
\end{lemma}
\begin{proof}$\:$
It follows from the definition of $\beta^{(i,j)}(\mu)$ that

\begin{eqnarray}
&
% \hskip -3.5ex
\begin{aligned}
&	\hskip -8ex
	\left|
		\beta^{(i,j)}(\widetilde{\mu}) - 	\beta^{(i,j)} (\mu) 
	\right|	\;\;
			\le   	\\
		\:	
&	\quad\quad		\left| \:
		\lambda_{\min}
		\left(
			\left\{ \Lambda^{(i)} - \lambda^{(i)}_1 I_\ell \right\} 	- 	
			[V^{(i)}]^\ast U_j(\mu) U_j(\mu)^\ast V^{(i)}	 \left\{ \Lambda^{(i)} - \lambda^{(i)}_{\ell+1} I_\ell \right\}
		\right) \right.				\\
&					\left.		\quad\quad\quad\quad	 - \;\;
	\lambda_{\min}\left(
	\left\{ \Lambda^{(i)} - \lambda^{(i)}_1 I_\ell \right\} 	- 	
			[V^{(i)}]^\ast U_j(\widetilde{\mu}) U_j(\widetilde{\mu})^\ast V^{(i)}		
								\left\{ \Lambda^{(i)} - \lambda^{(i)}_{\ell+1} I_\ell \right\} \right) \:
		\right|	\label{eq:52}
\end{aligned}				\\
&	\hskip -19.8ex
	\le 
		\;\;	\left\|
			[V^{(i)}]^\ast \left( U_j(\widetilde{\mu}) U_j(\widetilde{\mu})^\ast
				-	U_j(\mu)U_j(\mu)^\ast \right)V^{(i)} 
			\left\{	  \Lambda^{(i)} - \lambda^{(i)}_{\ell+1} I_\ell \right\}   
			\right\|			 \label{eq:54} \\
&
		\hskip -5.5ex
	\le
		\;\;
		\left| \lambda^{(i)}_1 - \lambda^{(i)}_{l+1} \right|	
		\left\|
			W_j(\widetilde{\mu}) W_j(\widetilde{\mu})^\ast	-
			W_j(\widetilde{\mu}) W_j(\mu)^\ast	+
			W_j(\widetilde{\mu}) W_j(\mu)^\ast 	-
			W_j(\mu)W_j(\mu)^\ast
		\right\| \label{eq:55} 	\\
&	\hskip -33.7ex 
	\le	
	\;\;
	2
	\left|
		\lambda^{(i)}_1	-	\lambda^{(i)}_{l+1}
	\right|
			\left\|
				W_j(\widetilde{\mu})	-	W_j(\mu)
			\right\|
			\;	\le	\;
	\gamma_\beta
					\left\|
							\mu - \widetilde{\mu} 
					\right\|	\;		\label{eq:56}
\end{eqnarray}
% \end{footnotesize}
\noindent
for some $\gamma_\beta$ independent of $j$ and $i$,
where from \eqref{eq:52} to \eqref{eq:54} we have used Weyl's theorem, from \eqref{eq:54} to \eqref{eq:55} the definition of $U_j(\mu)$ in \eqref{eq:57}. Finally, for the last inequality in \eqref{eq:56}, we have used \Cref{thm:lipeig2}. Letting
 $M_j(\widehat{\mu})$
denote a matrix representation of the linear map $W_j(\widehat{\mu})$ for any $\widehat{\mu} \in \calD$, we again remark that,
 $\beta^{(i,j)}(\widetilde{\mu})$ is independent of the orthonormal basis in the columns of
 $M_j(\widetilde{\mu})$ for the invariant subspace ${\mathcal W}_j(\widetilde{\mu})$
 of $A^{V_j}(\widetilde{\mu})$ spanned by its eigenvectors
 $w^{V_j}_1(\widetilde{\mu}), \dots , w^{V_j}_{\sr}(\widetilde{\mu})$.
 In the derivation above, we use $M_j(\widetilde{\mu})$ with orthonormal columns spanning 
 ${\mathcal W}_j(\widetilde{\mu})$, and satisfying
 $\| M_j(\widetilde{\mu}) - M_j(\mu) \| \leq \gamma_X \left\|  \widetilde{\mu} - \mu  \right\|$,
 whose existence is guaranteed by \Cref{thm:lipeig2}.  
\end{proof}

A big step forward to show the Lipschitz continuity of the lower bound $\lambda_{{\rm LB}}^{(j)}(\mu)$ 
is establishing the Lipschitz continuity of $\eta^{(j)}_\ast(\mu)$ defined in \eqref{eq:LP}, 
as $\lambda_{{\rm LB}}^{(j)}(\mu) = f(\eta^{(j)}_\ast(\mu))$ for the function $f$ in (\ref{eq:low_bound1}).
To this end, recall that $y^{(j)}(\mu) \in {\mathbb R}^{\kappa}$ denotes a minimizer of the linear program in (\ref{eq:LP}).
Without loss of generality, we can assume there are $\kappa$ linearly independent active constraints
at this minimizer $y^{(j)}(\mu)$ of (\ref{eq:LP}) out of $2\kappa+j$ constraints all together
\cite{MatG06}. Consequently, $y^{(j)}(\mu)\in\R^{\kappa}$ must satisfy a linear system
\begin{equation}\label{eqn:theta:LP}
    \Theta^{(j)}(\mu)y^{(j)}(\mu)\;=\;\Upsilon^{(j)}(\mu),
\end{equation}
where $\Theta^{(j)}(\mu)\in\R^{\kappa\times\kappa}$ is invertible, and each equation in (\ref{eqn:theta:LP}) 
corresponds either to an active inequality constraint of the form $\theta(\mu_i)^{\T}y\ge \lambda_i+\beta^{(i,j)}(\mu)$, 
or to a box constraint. 
To be precise, there may be more than $\kappa$ active constraints at the minimizer $y^{(j)}(\mu)$
of (\ref{eqn:theta:LP}) in which case we consider the $\kappa$ linearly independent active constraints 
such that the smallest singular value of $\Theta^{(j)}(\mu)$ in (\ref{eqn:theta:LP}) is maximized.
For the Lipschitz continuity of $\eta^{(j)}_\ast(\mu)$, we assume that
the smallest singular value $\sigma_{\min}(\Theta^{(j)}(\mu))$
of $\Theta^{(j)}(\mu)$ remains away from zero as $j \rightarrow \infty$ for all $\mu \in \calD$.
\begin{assumption}\label{ass:svals_away_zero}
There exists a real number $\zeta > 0$ such that
$
	\sigma_{\min}( \Theta^{(j)}(\mu) ) > \zeta \;\;\, 	\text{for all}\; \mu \in {\mathcal D} \, ,	\:  j \ge 1.
$
\end{assumption}
We deduce the Lipschitz continuity of $\eta^{(j)}_\ast(\mu)$ with a Lipschitz constant independent of $j$ next.
\begin{lemma}\label{Lemma1}
	Suppose that Assumptions \ref{ass:eig_sep} and \ref{ass:svals_away_zero} hold, and $\ell > \sr$.
	There exists a real scalar $\gamma_\eta > 0$ independent of $j$ such that
	the scalar function $\eta^{(j)}_\ast(\mu)$ defined in \eqref{eq:LP} satisfies
	\begin{equation*}
		|\eta^{(j)}_\ast(\widetilde{\mu})-\eta^{(j)}_\ast(\mu)|
		\;	\le	\;
		\gamma_\eta 	\|\widetilde{\mu}  -  \mu\|			\quad\;	\text{for all }\;\; \widetilde{\mu},	\mu  \in  \calD \: .
	\end{equation*}
\end{lemma}
\begin{proof}$\:$
By the definition of $\eta^{(j)}_\ast(\mu)$, we have 
	\begin{equation*}
		\begin{split}
			| \eta^{(j)}_\ast(\mu) - & \eta^{(j)}_\ast(\widetilde{\mu}) |
			\;	=	\;\;
			\left| \,    \min_{y \in \calY_{{\rm LB}}^{(j)}(\mu)}  \theta(\mu)^\T y
			-
			 \min_{y \in \calY_{{\rm LB}}^{(j)} \left( \widetilde{\mu} \right)}  \theta(\widetilde{\mu})^\T y		\; \right|
			\\
			=& \;\;
			\left| \,   \min_{y \in \calY_{{\rm LB}}^{(j)}(\mu)}  \theta(\mu)^\T y
			\;\; -
			\min_{y \in \calY_{{\rm LB}}^{(j)}(\mu)}  \theta(\widetilde{\mu})^\T y
			\;\; +  
			\min_{y \in \calY_{{\rm LB}}^{(j)}(\mu)}  \theta(\widetilde{\mu})^\T y
			\;\; -
			\min_{y \in \calY_{{\rm LB}}^{(j)} \left( \widetilde{\mu} \right)}  \theta(\widetilde{\mu})^\T y	\; \right|	\\[.4em]
			\le&\;\; \underbrace{ \left| \,   \min_{y \in \calY_{{\rm LB}}^{(j)}(\mu)}  \theta(\mu)^\T y
						\;\; -
			\min_{y \in \calY_{{\rm LB}}^{(j)}(\mu)}  \theta(\widetilde{\mu})^\T y \, \right|}_{= \vcentcolon \;\mathbf{A}}
					\;\;	+	\;\;
			\underbrace{\left|\, \min_{y \in \calY_{{\rm LB}}^{(j)}(\mu)}  \theta(\widetilde{\mu})^\T y
						\;\; -
			\min_{y \in \calY_{{\rm LB}}^{(j)} \left( \widetilde{\mu} \right)}  \theta(\widetilde{\mu})^\T y	\; \right|}_{=\vcentcolon\;\mathbf{B}},	
		\end{split}
	\end{equation*}
	so we deal with the terms $\mathbf{A}$ and $\mathbf{B}$. The first of these two terms is related to the results 
	of two linear programming problems over the same feasible set but with two different objective functions, while 
	the second term is the difference between two linear programming problems with the same objective function but 
	defined over two different feasible regions.
	
	Concerning the first term, we have
	\begin{align}
		\mathbf{A}\;=	&	\;\;\;
		\left| \,   \min_{y \in \calY_{{\rm LB}}^{(j)}(\mu)}  \theta(\mu)^\T y
						\;\; -
			\min_{y \in \calY_{{\rm LB}}^{(j)}(\mu)}  \theta(\widetilde{\mu})^\T y \, 
		\right|\label{eq:60}	\\[.6em]
		\le	&	\; \; \; \; 
		\theta(\widetilde{\mu})^\T   y^{(j)}(\mu)  	-    	\theta(\mu)^\T   y^{(j)}(\mu)   \;	\label{eq:58}	\\[.6em]
		\le	&	\; \; \; \;	
		\| \theta(\widetilde{\mu}) - \theta(\mu) \| \| y^{(j)}(\mu) \|	
		\;\,	\le	\;
		\gamma_{1}  \|\widetilde{\mu} - \mu\|	\quad\;	\text{for all }\quad\tilde{\mu},\mu\in\calD		\;	\label{eq:59}	
	\end{align}
	for some constant $\gamma_1 > 0$ independent of $j$,
	where $y^{(j)}(\mu) := \arg\min_{y \in {\calY}^{(j)}_{{\rm LB}}(\mu)} \theta(\mu)^\T  y$, and, in the
	first inequality, we assume, without loss of generality,
 	\[
		\min_{y \in \calY^{(j)}_{\rm LB}(\mu)} \theta(\widetilde{\mu})^\T y \: \geq \: 
					\min_{y \in {\calY}^{(j)}_{\rm LB}(\mu)} \theta(\mu)^\T  y \; ;
	\]				
	 indeed, if the opposite inequality holds, 
	 we can pass from \eqref{eq:60} to \eqref{eq:58} by
	 replacing $y^{(j)}(\mu)$ with  a minimizer of $\theta(\widetilde{\mu})^\T  y$ 
	 over all $y \in {\calY}^{(j)}_{{\rm LB}}(\mu)$ in \eqref{eq:58}
	 and negating the right-hand side of \eqref{eq:58}. 
	 We also remark that the first inequality in (\ref{eq:59}) follows from the Cauchy–Schwarz inequality, 
	 while the second inequality in (\ref{eq:59}) is due to the analyticity of $\theta_i(\mu)$ 
	 for $i = 1, \dots, \kappa$, as well as the boundedness of $\| y^{(j)}(\mu) \|$, since $y^{(j)}(\mu)$ belongs to
	 the compact set ${\calY}^{(j)}_{{\rm LB}}(\mu)$. 
	
	Now let us consider the second term ${\mathbf B}$, which is the difference of the
	minimization problems
	\[
		\min_{y \in \calY^{(j)}_{{\rm LB}}(\mu)}\theta(\widetilde{\mu})^\T y
			\quad\quad\quad	\text{and}		\quad\quad
		\min_{y \in \calY^{(j)}_{{\rm LB}} \left( \widetilde{\mu} \right)}\theta(\widetilde{\mu})^\T y
	\]
	in absolute value. For both of these minimization problems,
	the gradient of the objective $\theta(\widetilde{\mu})$ and the gradients of the constraints 
	$\theta(\mu_1), \dots, \theta(\mu_j)$ (due to the non-box constraints), as well as
	${\mathbf e}_1, \dots , {\mathbf e}_\kappa$ (due to the box constraints)
	with respect to the optimization variable $y$ are the same. From the first-order optimality
	conditions (see, e.g., \cite[Thm.~12.1]{NocW00}), for both minimization problems, there are the same 
	indices $\ell_1, \dots, \ell_m \in \{1, \dots , j \}$, $l_1, \dots , l_k \in \{ 1, \dots , \kappa \}$,
	and the same 
	Lagrange multipliers
	$\lambda_{1}, \dots , \lambda_{m}$ all positive, $\varphi_1, \dots \varphi_{k}$ all nonzero
	such that $m \in [0, j]$, $k \in [0,\kappa]$, $m + k \geq 1$ and
	\begin{equation}\label{eq:KKT}
		\theta(\widetilde{\mu})	\;	=	\;
	\sum_{i=1}^m \lambda_i \theta(\mu_{\ell_i})	+	\sum_{i=1}^k \varphi_i {\mathbf e}_{l_i}	\:	.
	\end{equation} 
	By the complementary conditions, $\ell_i \, , \; i = 1, \dots , m \: $ and $ \: l_i \, , \; i = 1, \dots, k$ above 
	correspond to the indices of active non-box and active box constraints, respectively. In particular,
	the minimization problems have the same set of active constraints. The minimizers
	$\; \underline{y}^{(j)}(\mu) \;  :=  \;
		\arg\min_{\: y \in \calY^{(j)}_{{\rm LB}} \left( \mu \right)}\;\theta(\widetilde{\mu})^\T y \; $
	and
	$\; \underline{y}^{(j)}(\widetilde{\mu})	\; = \; 
	y^{(j)}(\widetilde{\mu})	\; := \; 
		\arg\min_{\: y \in \calY^{(j)}_{{\rm LB}} \left( \widetilde{\mu} \right)}\;\theta(\widetilde{\mu})^\T y \;$
	of these minimization problems satisfy
	\begin{equation}\label{eq:LP_linsys}
		\Phi \,  \underline{y}^{(j)}(\mu) \: = \:  \Psi(\mu)
		\quad\quad	\text{and}		\quad\quad
		\Phi \, \underline{y}^{(j)}(\widetilde{\mu}) \: = \: \Psi(\widetilde{\mu})	\:	,
	\end{equation}
	where
	\begin{equation}\label{eq:LP_linsys_coeffs}
	\begin{split}
		&	\Phi
			\;	:=	\;
		\left[
			\begin{array}{cccccc}
				\theta(\mu_{\ell_1})	&	\dots		&	\theta(\mu_{\ell_m})	&
				{\mathbf e}_{l_1}	&	\dots		&	{\mathbf e}_{l_k}
			\end{array}
		\right]^\T	\;	,	\\[.3em]
		&	\Psi(\widehat{\mu})		\;	:=	\;
		\left[
			\begin{array}{cccccc}
	\lambda^{(\ell_1)}_1  + \beta^{(\ell_1, j)}(\widehat{\mu})	&	\dots		&	\lambda^{(\ell_m)}_1  + \beta^{(\ell_m, j)}(\widehat{\mu})	
	& s_{l_1}	&	\dots		&	s_{l_k}
			\end{array}
		\right]^\T	
	\end{split}
	\end{equation}
	with $s_{l_i} = -\lambda_{\max}(A_{l_i})$ if $\varphi_{i} < 0$ and
	$s_{l_i} = \lambda_{\min}(A_{l_i})$ if $\varphi_{i} > 0$ for $i = 1, \dots , k$.
	The equality in (\ref{eq:KKT}) can be expressed as
	\begin{equation}\label{eq:multiplier_coeffs}
		\theta(\widetilde{\mu})	\;	=	\;	\Phi^T \mathbf{m}	\:	,	\quad	\text{where}	\;\;
		\mathbf{m}
			:=
		\left[
			\begin{array}{cccccc}
				\lambda_1		&	\dots		&	\lambda_m	&
				\varphi_1	&	\dots		&	\varphi_k
			\end{array}
		\right]^\T	\:	.
	\end{equation}
	As a result, $\| {\mathbf m} \| \leq \| \theta(\widetilde{\mu}) \| / \sigma_{\min}(\Phi)$, where
	$\| \theta(\widetilde{\mu}) \|$ is bounded as it is the norm of a continuous function and $\widetilde{\mu}$
	belongs to the compact domain ${\mathcal D}$, while 
	$\sigma_{\min}(\Phi) \geq \sigma_{\min}(\Theta^{(j)}(\mu)) \geq \zeta > 0$ by \Cref{ass:svals_away_zero}, 
	that is $1/ \sigma_{\min}(\Phi)$ is bounded above by $1/\zeta$, a constant independent of $j$. 
	Now it follows that
	\begin{align}\label{eq:ult_bound_onB}
	\begin{aligned}
		{\mathbf B}
			\;\;	&	=	\;\;
		\left|\, \min_{y \in \calY_{{\rm LB}}^{(j)}(\mu)}  \theta(\widetilde{\mu})^\T y
			-
			\min_{y \in \calY_{{\rm LB}}^{(j)} \left( \widetilde{\mu} \right)}  \theta(\widetilde{\mu})^\T y	\; \right|	
				\;\;		=	\;\;
		\left|
				\theta(\widetilde{\mu})^\T	 \left( 	\underline{y}^{(j)}(\mu)	-	
										 \underline{y}^{(j)}(\widetilde{\mu})	\right)		
		\right|				\\[.2em]
				\;\;	&	=		\;\;
		\left|
			 \mathbf{m}^T  \Phi \left( 	\underline{y}^{(j)}(\mu)	-	
										 \underline{y}^{(j)}(\widetilde{\mu})	\right)
		\right|		
				\;\;	=	\;\;
		\left|
			\mathbf{m}^T  \left(  \Psi(\mu) - \Psi(\widetilde{\mu})  \right)	
		\right|
			\;\;\\
           & \leq		\;\;
		\left\|
			{\mathbf m}
		\right\|
			\cdot
		\left\|
			\Psi(\mu) - \Psi(\widetilde{\mu}) 
		\right\|
			\;\;	\leq	\;\;	\gamma_2 \| \mu - \widetilde{\mu} \|  \: ,
	\end{aligned}
	\end{align}
	for a constant $\gamma_2$ independent of $j$, since $\| \mathbf{m} \|$ is bounded by a constant
	independent of $j$, whereas 
		\[ \Psi(\mu) - \Psi(\widetilde{\mu}) 	=	
			\begin{bmatrix}
					 \beta^{(\ell_1, j)}(\mu) -  \beta^{(\ell_1, j)}(\widetilde{\mu}),	&	
							\dots	\:	,		&	
					\beta^{(\ell_m, j)}(\mu) -  \beta^{(\ell_m, j)}(\widetilde{\mu}),	
										& 	0,	&	\dots		\:	,		&	0
			\end{bmatrix}\, ,\]
                so
				$\| \Psi(\mu) - \Psi(\widetilde{\mu}) \| \leq \widetilde{\gamma} \| \mu - \widetilde{\mu} \|$
				for a constant $\widetilde{\gamma}$ independent of $j$ due to \Cref{Lemma0.1}. 
	Note that in (\ref{eq:ult_bound_onB}), the third equality follows from (\ref{eq:multiplier_coeffs}),
	the fourth equality from (\ref{eq:LP_linsys}), the first inequality in the last line from the
	Cauchy-Schwarz inequality.
	Thus, we conclude that
	\begin{equation*}
		\left| \eta^{(j)}_\ast(\widetilde{\mu}) - \eta^{(j)}_\ast(\mu) \right|
			\;\;	\le	\;\;
		{\mathbf A} + {\mathbf B}	
			\;\;	\leq	\;\;
		(\gamma_1 + \gamma_2) \left\| \widetilde{\mu} - \mu \right\|
	\end{equation*}
	for all $\mu, \widetilde{\mu} \in \calD$, where the constant $\gamma_1 + \gamma_2$ is independent of $j$.
\end{proof}

\begin{lemma}\label{lemma5.6}
Suppose that Assumptions \ref{ass:eig_sep} and \ref{ass:svals_away_zero} hold, and $\ell > \sr$. In particular,
suppose that \Cref{alg:sf} is initialized with the points $\mu_{1,1}, \dots , \mu_{1,\eta}$ 
(see \Cref{rem:initialize})  on a uniform grid for ${\mathcal D}$ with two adjacent grid-points 
at a distance not greater than $\: h := \delta / (\sqrt{p} \gamma) \:$ from each other, where
$\delta := \min_{\mu \in {\mathcal D}} \: \lambda_{\sr+1}(\mu) - \lambda_1(\mu) \; > 0$ and 
$\gamma := \gamma_{\lambda} + \gamma_{\eta}$ with $\gamma_{\lambda}$ and $\gamma_\eta$
denoting the Lipschitz constants in \Cref{lemma:lip:con:eig} and \Cref{Lemma1}, respectively,
both independent of $j$. Then, letting
	\begin{equation}\label{a2}
		a^{(j)}_2(\mu)		\;	:=	\;	
		\left| \lambda_1^{\mathcal V_j}(\mu) - \eta^{(j)}_{\ast}(\mu) \right|  
		+   \sqrt{	\left| \lambda_1^{\mathcal V_j}(\mu) - \eta^{(j)}_{\ast}(\mu) \right|^2   +   4 \rho^{(j)}(\mu)^2},
	\end{equation}
	we have
	\[
		\left| \lambda_1^{\mathcal V_j}(\mu) - \eta^{(j)}_{\ast}(\mu) \right|	\; \ge \;\: \delta/2 \; > \; 0
		\quad\;\;	\text{and}		\quad\;\;\,
		a^{(j)}_2(\mu) \; \ge \;\: \delta \; > \; 0	\quad\;\;		\text{for all } \mu \in \calD	\:	
	\]
	for all $j$.
\end{lemma}
\vskip -3ex
\begin{proof}$\:$
	First observe that every $\mu \in \calD$ is at a distance of at most $(\sqrt{p} h)/2 = \delta / (2\gamma)$
	to one of the grid-points $\mu_{1,1}, \dots , \mu_{1,\eta}$. This can be seen by considering the
	hypercubes centered at the grid-points with side-lengths equal to $h$, as every $\mu \in \calD$ is in one
	of these hypercubes, and the distance from the center of this hypercube to any point in
	in the hypercube cannot exceed $(\sqrt{p} h)/2$. Take any $\widehat{\mu} \in \calD$, and let $\mu_{1,c}$
	be the grid point at a distance  from $\widehat{\mu}$ of at most $(\sqrt{p} h)/2 = \delta / (2\gamma)$, 
	where $c \in {\mathbb N}, c \in [1, \eta]$.
	
	By part \ref{item6} of \Cref{lem:props_eta}, in particular from (\ref{eq:lbound_eta}), and due
	to $\lambda_1^{\calV_j}(\mu_{1,c}) = \lambda_1	(\mu_{1,c})$ (since $\calV_j$ contains
	an eigenvector of $A(\mu_{1,c})$ corresponding to its eigenvalue $\lambda_1(\mu_{1,c})$),
	we have
	\[
		\left|  \lambda_1^{\calV_j}(\mu_{1,c})-\eta^{(j)}_{\ast}(\mu_{1,c})  \right|
			\;\;	=	\;\;
		\eta^{(j)}_{\ast}(\mu_{1,c})	-   \lambda_1^{\calV_j}(\mu_{1,c})
			\;\;		\geq		\;\;
		\lambda_{\sr + 1}(\mu_{1,c})	-	\lambda_1	(\mu_{1,c})	
			\;\;		\geq		\;\;		\delta.
	\] 
	By \Cref{lemma:lip:con:eig} and \ref{Lemma1}, the functions 
	$\lambda_1^{\calV_j}(\mu)$ and  $\eta^{(j)}_{\ast}(\mu)$ are Lipschitz continuous
	with the Lipschitz contant $\gamma_{\lambda}$ and $\gamma_\eta$, respectively,
	both independent of $j$. Using these Lipschitz continuity properties, and
	recalling $\| \widehat{\mu} - \mu_{1,c} \| \leq (\sqrt{p} h)/2 = \delta / (2\gamma)$,
	we deduce
	\begin{equation*}
	\begin{split}
		\eta^{(j)}_{\ast}(\widehat{\mu})	-   \lambda_1^{\calV_j}(\widehat{\mu})
				\;\;	&	\geq		\;\;
		(\eta^{(j)}_{\ast}(\mu_{1,c}) - \gamma_\eta \| \widehat{\mu} - \mu_{1,c} \|) 	
				-   (\lambda_1^{\calV_j}(\mu_{1,c}) + \gamma_\lambda \| \widehat{\mu} - \mu_{1,c} \|) 	\\[.3em]
				&	=		\;\;
		\left\{  \eta^{(j)}_{\ast}(\mu_{1,c})	-   \lambda_1^{\calV_j}(\mu_{1,c})  \right\}
				\:	-	\:
			(\gamma_{\lambda} + \gamma_\eta) \| \widehat{\mu} - \mu_{1,c} \|		% \\[.1em] &	\geq 	\;\;
					\;\;	\geq		\;\;
				\delta -  \gamma \left(\frac{\delta}{2\gamma} \right)
					\;\;	=	\;\;	\delta/2	\:	.
	\end{split}
	\end{equation*}
	This shows that
	$| \eta^{(j)}_{\ast}(\mu)	-   \lambda_1^{\calV_j}(\mu) | \geq \delta/2$
	for all $\mu \in \calD$, and $a^{(j)}_2(\mu) \geq \delta$ for all $\mu \in \calD$ as claimed.
	% say to $\mu_{1,k}$.
\end{proof}

We finally establish the uniform Lipschitz continuity of $\lambda^{(j)}_{\rm{LB}}(\mu)$
in the next result.
\begin{theorem} \label{teo1}
	Suppose that Assumptions \ref{ass:eig_sep} and \ref{ass:svals_away_zero} hold, and $\ell > \sr$.
	Then, there exists a positive real scalar $\gamma_{\rm{LB}}$ independent of $j$ such that 
	\begin{equation}\label{eqn:LB:LC}
		\left| 	\lambda^{(j)}_{\rm{LB}}(\widetilde{\mu})	-	\lambda^{(j)}_{\rm{LB}}(\mu)	\right|
		\;	\le	\;	
		\gamma_{\rm{LB}} \| \widetilde{\mu} - \mu \|	\quad\;	\text{for all} \;\; \widetilde{\mu} , \mu \in \calD	\:	.
	\end{equation}
\end{theorem}
\begin{proof}$\:$
Recalling $\lambda^{(j)}_{\rm{LB}}(\mu) \: = \: f^{(j)}(\eta^{(j)}_{\ast}(\mu))$, we equivalently show
\[
	\left|
		f^{(j)}(\eta^{(j)}_{\ast}(\widetilde{\mu}))	-	f^{(j)}(\eta^{(j)}_{\ast}(\mu))
	\right|
			\;	\le	\;	
	\gamma_{\rm{LB}} \| \widetilde{\mu} - \mu \|	\quad\;	\text{for all} \;\; \widetilde{\mu} , \mu \in \calD
\]
for a constant $\gamma_{\rm{LB}}$ independent of $j$.
Letting
$\;
    a_1^{(j)}(\mu)		\;	\vcentcolon=	\;		
    	\min \{ \lambda^{\calV_j}_1(\mu) , \eta^{(j)}_{\ast}(\mu) \}	\; 	,	\:
$
we have
\vskip -1.5ex
\begin{equation}\label{eq:feta}
		f^{(j)}(\eta_{\ast}^{(j)}(\mu))
				\;	=	\;
		a^{(j)}_1(\mu)
				\, - \,
		\frac{2\rho^{(j)}(\mu)^2}{a^{(j)}_2(\mu)}	\,	,
\end{equation}
\vskip -2ex
\noindent
with $a^{(j)}_2(\mu)$ as defined in \eqref{a2}.
As shown in \Cref{lemma:lip:con:eig} and \Cref{Lemma1},
the functions $\lambda_1^{\calV_j}(\mu)$ and $\eta^{(j)}_{\ast}(\mu)$ are Lipschitz continuous
with Lipschitz constants $\gamma_\lambda$ and $\gamma_\eta$, both independent of $j$. 
The function $a^{(j)}_1(\mu)$, that is the minimum of the Lipschitz continuous functions 
$\lambda_1^{\calV_j}(\mu)$ and $\eta^{(j)}_{\ast}(\mu)$, is also Lipschitz continuous
with the Lipschitz constant $\gamma_1 := \max\{ \gamma_\lambda , \gamma_\eta \}$ independent of $j$. 
 
As for the Lipschitz continuity of the second term on the right-hand side of (\ref{eq:feta}), 
observe that
\begin{equation*}
\begin{split}
	\left|
		\frac{ 2 \rho^{(j)}(\widetilde{\mu})^2 }{a^{(j)}_2(\widetilde{\mu})} 
			\: - \: 
		\frac{2\rho^{(j)}(\mu)^2}{a^{(j)}_2(\mu)}
	\right|
	\;\;	 & = 	\;\;
	\left|
		\frac{2\rho^{(j)}(\widetilde{\mu})^2 a^{(j)}_2(\mu)
				\: - \:
		2\rho^{(j)}(\mu)^2 a_2^{(j)}(\widetilde{\mu})}{a_2^{(j)}(\widetilde{\mu}) a^{(j)}_2(\mu)}
	\right|	\\
	   & \le
	\;\;
	2a^{(j)}_2(\mu)
	\frac{\left| \rho^{(j)}(\widetilde{\mu})^2 \: - \: \rho^{(j)}(\mu)^2 \right|}{a_2^{(j)}(\widetilde{\mu})a^{(j)}_2(\mu)}
			\; + \;
	2\rho^{(j)}(\mu)^2
	\frac{\left| a_2^{(j)}(\widetilde{\mu}) \: - \: a^{(j)}_2(\mu) \right|}{a_2^{(j)}(\widetilde{\mu})a^{(j)}_2(\mu)}
\end{split}
\end{equation*}
for all $\mu, \widetilde{\mu} \in \calD$. 
Recalling the definition of $a^{(j)}_2(\mu)$ in (\ref{a2}),
the existence of a constant $\overline{\delta}$ independent of $j$ such that
\begin{equation}\label{eq:12}
		\overline{\delta}  \;\; \geq \;\;  a^{(j)}_2(\mu) \quad\; \text{for all}\;\; \mu \in \calD \;	
\end{equation}
can be inferred from the following observations:
\begin{enumerate}
\item[(i)] $\lambda_1^{\calV_j}(\mu) \leq \| A(\mu) \|$ and $\rho^{(j)}(\mu) \leq \| A(\mu) \|$
by its definition in (\ref{eq:lbound_int}), where the continuous function $\| A(\mu) \|$
attains a maximum over all $\mu \in \calD$; 
\item[(ii)] $| \eta^{(j)}_\ast(\mu) | \leq  \max_{\mu \in {\mathcal D}} \| \theta(\mu) \| \, \max_{y \in {\mathcal B}} \| y \|$
by the definition of $\eta^{(j)}_\ast(\mu)$ in (\ref{eq:LP}),
where the continuous functions $\| \theta(\mu) \|$ and $\| y \|$ over $\mu \in {\mathcal D}$ and 
$y \in {\mathcal B}$, respectively, attain maxima. 
\end{enumerate}
Furthermore, 
as the functions $\lambda_1^{\calV_j}(\mu)$, $\eta^{(j)}_{\ast}(\mu)$ are Lipschitz continuous functions
by constants independent of $j$, so is 
$| 
	\lambda_1^{\mathcal V_j}(\mu) - \eta^{(j)}_{\ast}(\mu) 
|^2$.
Additionally, $\rho^{(j)}(\mu)^2$ is Lipschitz continuous with a Lipschitz constant independent of $j$ 
by \Cref{Lemma2}, and \Cref{lemma5.6} implies
 \[
 	\sqrt{ \left| \lambda_1^{\mathcal V_j}(\mu) - \eta^{(j)}_{\ast}(\mu) \right|^2 + 4 \rho^{(j)}(\mu)^2}\;\ge\;\delta/2
 \]
is positive for all $\mu \in {\mathcal D}$, so is Lipschitz continuous as a function of $\mu$  
with a Lipschitz constant independent of $j$.
Thus, we deduce the existence of a constant $\gamma_{a} > 0$ independent of $j$ such that
\begin{equation}\label{eq:10}
	\left|
		a^{(j)}_2({\widetilde{\mu}}) - a^{(j)}_2(\mu)
	\right|
			\; \leq\; 
	\gamma_a \|{\widetilde{\mu}} - \mu\|
	\quad\;	\text{for all}\;\; \widetilde{\mu} , \mu \in {\mathcal D}	\:	.
\end{equation}
Finally, by using \Cref{Lemma2}, \Cref{lemma5.6}, as well as equations \eqref{eq:12}, \eqref{eq:10}, 
we have
\begin{equation}\label{eqn:LB:LC:P2}
 	\begin{split}
	\hskip -1.5ex
	\left|
		\frac{ 2 \rho^{(j)}(\widetilde{\mu})^2 }{a^{(j)}_2(\widetilde{\mu})} 
			\: - \: 
		\frac{2\rho^{(j)}(\mu)^2}{a^{(j)}_2(\mu)}
	\right|
		\; \le \;
		&
	2a^{(j)}_2(\mu)
	\frac{\left| \rho^{(j)}(\widetilde{\mu})^2 \: - \: \rho^{(j)}(\mu)^2 \right|}{a_2^{(j)}(\widetilde{\mu})a^{(j)}_2(\mu)}
			\; + \;
	2\rho^{(j)}(\mu)^2
	\frac{\left| a_2^{(j)}(\widetilde{\mu}) \: - \: a^{(j)}_2(\mu) \right|}{a_2^{(j)}(\widetilde{\mu})a^{(j)}_2(\mu)}
		\\
		\; \le \;
		&
	2\overline{\delta} \,
	\frac{\gamma_\rho\|\widetilde{\mu} \: - \: \mu\|}{\delta^2}
		\;\;  + \;\;
	2\left( \max_{\mu\in\calD} \, \| A(\mu) \| \right)^2 \frac{\gamma_a\|\widetilde{\mu} \: - \: \mu\|}{\delta^2}\;	
		\;	\le	\;
		\gamma_2 \| \widetilde{\mu} \: - \: \mu \|,
  	\end{split}
\end{equation}
 for all $\widetilde{\mu},\mu\in\calD$ for a constant $\gamma_2$ independent of $j$. 
 Combining \eqref{eqn:LB:LC:P2} with the Lipschitz continuity of $a_1^{(j)}(\mu)$ with the Lipschitz constant $\gamma_1$,
 we conclude with \eqref{eqn:LB:LC} for the constant $\gamma_{\rm{LB}} = \gamma_1 + \gamma_2$ 
 independent of $j$.
\end{proof}

\subsection{Lipschitz continuity of \texorpdfstring{$H^{(j)}(\mu)$}{NOTEXT}}\label{sec: 5.3}
Now we are ready to state the main Lipschitz continuity result.
\begin{theorem}\label{thm:main_Lips_C}
Suppose that Assumptions \ref{ass:eig_sep} and \ref{ass:svals_away_zero} hold, and $\ell > \sr$.
Then, there exists a positive real scalar $\gamma$ independent of $j$ satisfying
	\begin{equation}\label{eq:H_Lips}
		\left|
			H^{(j)}(\widetilde{\mu})-H^{(j)}(\mu)
		\right|
		\;	\leq	\;
		\gamma\|\widetilde{\mu}-\mu\|
		\quad\;	\text{for all}\;\; \mu , \widetilde{\mu} \in \calD	\:	.
	\end{equation}
%	for all $\mu,\widetilde{\mu}\in\calD$.
\end{theorem}
\vskip -3.5ex
\begin{proof}$\:$
	Recalling $H^{(j)}(\mu) = \lambda^{\calV_j}_{\min}(\mu) - \lambda^{(j)}_{\LB}(\mu)$, and 
	using \Cref{lemma:lip:con:eig} and \Cref{teo1} that assert that
	the functions $\lambda^{\calV_j}_{\min}(\mu)$ and $\lambda^{(j)}_{\rm{LB}}(\mu)$ 
	are Lipschitz continuous with Lipschitz constants independent of $j$, 
	it is immediate that $H^{(j)}(\mu)$ is also Lipschitz continuous
	with a Lipschitz constant $\gamma$ independent of $j$.
\end{proof}

%%-------------------------------------------------------------------

\section*{Acknowledgments}
The authors thank the anonymous referees, whose comments helped to improve the quality of the manuscript. Funded by Deutsche Forschungsgemeinschaft (DFG, German Research Foundation) under Germany’s Excellence Strategy - EXC 2075 – 390740016. MM acknowledges funding by the BMBF (grant no.~05M22VSA) 
and the support from the Stuttgart Center for Simulation Science (SimTech). 
EM is grateful to Gran Sasso Science Institute for hosting him, which
facilitated his contributions and involvement in this work. 
NG acknowledges that his research was supported by funds from the Italian 
MUR (Ministero dell'Universit\`a e della Ricerca) within the
PRIN 2022 Project ``Advanced numerical methods for time-dependent parametric partial differential equations with applications'' and the Pro3 joint project entitled
``Calcolo scientifico per le scienze naturali, sociali e applicazioni: sviluppo metodologico e tecnologico''.
He is also affiliated with the INdAM-GNCS (Gruppo Nazionale di Calcolo Scientifico).

\bibliographystyle{plain}
\bibliography{journalabbr,approx_smallest_eig}

\end{document}